\documentclass[reqno,11pt]{amsart}

\usepackage{amssymb,amsmath,amsthm,amsfonts} 

\usepackage[left=3cm,right=3cm,top=3cm,bottom=3cm]{geometry}

\usepackage[mathscr]{eucal}
\usepackage{graphicx,xcolor}
\usepackage{enumitem}
\usepackage{comment}
\usepackage{setspace}

\usepackage[colorlinks=true,pdfstartview=FitV,linkcolor=blue, citecolor=blue, urlcolor=blue]{hyperref}

\allowdisplaybreaks
\numberwithin{equation}{section}

\theoremstyle{plain}
\newtheorem{theorem}{Theorem}[section]
\newtheorem{proposition}[theorem]{Proposition}
\newtheorem{lemma}[theorem]{Lemma}
\newtheorem{corollary}[theorem]{Corollary}
\newtheorem{conjecture}[theorem]{Conjecture}

\theoremstyle{definition}
\newtheorem{definition}[theorem]{Definition}
\newtheorem{remark}[theorem]{Remark}
\newtheorem{example}[theorem]{Example}

\newtheoremstyle{custom}
  {\topsep}   
  {\topsep}   
  {\it}          
  {}          
  {\bfseries} 
  {.}          
  {.5em}      
  {Hypothesis \thmnote{#3}} 

\theoremstyle{custom}
\newtheorem{theoremA}{Theorem}[section]

\usepackage{amssymb,amsmath}

\newcommand{\R}{\mathbb{R}}
\newcommand{\N}{\mathbb{N}}

\newcommand{\sE}{\mathscr{E}}
\newcommand{\sF}{\mathcal{F}}

\newcommand{\cH}{\mathcal{H}}
\newcommand{\cS}{\mathcal{S}}
\newcommand{\cC}{\mathrm{Conv_b}(\Rn)}
\newcommand{\convex}{{\mathcal K}}
\newcommand{\convexsets}{\cC}

\renewcommand{\div}{\mathrm{div}}
\renewcommand{\d}{\mathrm{d}}
\newcommand{\sd}{\mathrm{sd}}
 
\newcommand{\affine}{f}
\newcommand{\anisotropicmeancurvature}{\kappa_\norm}
\newcommand{\anisotropicmeancurvatureofEt}{\kappa_\norm}
\newcommand{\anisotropicperimeter}{P_\norm}
\newcommand{\atwfunction}{h}
\newcommand{\Borelset}{A}
\newcommand{\cl}[1]{\overline{#1}}
\newcommand{\constantgminus}{c_{g^-}}
\newcommand{\constantforcingMorrey}{\gamma_\forcing}
\newcommand{\constantprescribed}{c^{}_F}
\newcommand{\constdualnorm}{\constnorm}
\newcommand{\Constdualnorm}{\Constnorm}
\newcommand{\constiso}{c_{n, \norm}}
\newcommand{\constnorm}{c_{\norm}}
\newcommand{\Constnorm}{C_{\norm}}
\newcommand{\DeGiorgiexponent}{\alpha}
\newcommand{\DeGiorgifunctional}{\mathcal G}
\newcommand{\dist}{\mathrm{dist}}
\newcommand{\dualnorm}{\norm^o}

\newcommand{\forcing}{g}
\newcommand{
\forcingexponent
}{p}
\newcommand{\generalfunctional}{\sF_{\norm,\affine,\forcing}}
\newcommand{\GMM}{{\rm GMM}}
\newcommand{\Holderexponent}{\beta}
\newcommand{\hull}[1]{\overline{co}\left(#1\right)}
\newcommand{\indexATW}{j}
\newcommand{\indexsATW}{j(s)}
\newcommand{\lowerbounddensity}{\vartheta}
\newcommand{\lowerbounddensitymc}{\vartheta}
\newcommand{\lowerbddensity}{\theta_1}
\newcommand{\naturalindex}{h}
\newcommand{\initialconvexset}{{\mathcal K}_0}
\newcommand{\initialconvexsetindex}{{\mathcal K}_{0 \naturalindex}}
\newcommand{\initialset}{E_0}
\newcommand{\Int}[1]{{\rm Int}({#1})}
\newcommand{\inverseofscale}{\scalefactor^{-1}}
\newcommand{\largers}{s''}
\newcommand{\loc}{\mathrm{loc}}
\newcommand{\meancurvature}{\kappa}
\newcommand{\meancurvatureofEt}{\kappa}
\newcommand{\MM}{{\rm MM}}
\newcommand{\modulusofcontinuity}{\omega}
\newcommand{\norm}{\varphi}
\newcommand{\normalvelocity}{v}
\newcommand{\normalvelocityofEt}{v}
\newcommand{\Om}{\Omega}
\newcommand{\p}{\partial}
\newcommand{\prescrmcf}{{\mathfrak p}}
\newcommand{\Pvdi}{\mathcal X}
\newcommand{\Qvdi}{\mathcal Y}

\newcommand{\Rn}{\mathbb R^n}
\newcommand{\scalefactor}{\lambda}
\newcommand{\sign}{\mathrm{sign}}
\newcommand{\shorthandgeneralfunctional}{\mathcal F}
\newcommand{\smallers}{s'}

\newcommand{\unitballnorm}{{B_\norm}}

\newcommand{\veloc}[3]{\affine\left(\tfrac{\mathrm{d}_{#1}}{#3}\right)}
\newcommand{\sveloc}[3]{\affine\left(\tfrac{\mathrm{sd}_{#1}}{#3}\right)}
\newcommand{\intpart}[1]{\left\lfloor #1 \right\rfloor}

\newcommand{\Wulff}{W}

\setlist[itemize]{leftmargin=7mm} %

\title[Generalized power mean curvature flow]{Minimizing movements for the generalized power mean curvature flow}

\author[G. Bellettini]{Giovanni Bellettini} 
\address[G.Bellettini]{University of Siena, Via Roma 56, 53100 Siena, Italy,
and International Centre for Theoretical Physics (ICTP), Strada Costiery 11, 34151 Trieste, Italy}
\email[G. Bellettini]{bellettini@dsiim.unisi.it}

\author[Sh. Kholmatov]{Shokhrukh Yu. Kholmatov} 
\address[Sh. Kholmatov]{University of Vienna, Oskar-Morgenstern Platz 1, 1090 Vienna, Austria}
\email[Sh. Kholmatov]{shokhrukh.kholmatov@univie.ac.at}

\subjclass[2020]{35K93, 53E10, 35D30, 49Q20}

\keywords{Affine curvature flow, power mean curvature flow, Almgren-Taylor-Wang functional, minimizing movements}

\date{\today}
\begin{document} 

\begin{abstract}
Motivated by a conjecture of De Giorgi, we consider the Almgren-Taylor-Wang scheme for mean curvature flow, where the volume penalization is replaced by a term of the form 
$$
\int_{E\Delta F}\veloc{F}{x}{\tau}~dx
$$
for $\affine$ ranging in a large class of strictly increasing
continuous functions. In particular, our analysis covers the case 
$$
\affine(
r
) = r^\DeGiorgiexponent, \qquad r \geq 0, \ \ \DeGiorgiexponent >0,
$$
considered by De Giorgi. We show that 
the generalized minimizing movement
scheme converges to the  geometric evolution equation
$$
\affine(\normalvelocityofEt) = - \meancurvatureofEt\quad \text{on $\p E(t)$,} 
$$
where $\{E(t)\}$ are evolving subsets of $\Rn,$ $\normalvelocityofEt$  is the normal velocity of $\p E(t),$ 
and $\meancurvatureofEt$ is the mean curvature of $\p E(t)$. We extend our analysis to the anisotropic setting, 
and in the presence of a driving force.
We also show that minimizing movements coincide with the 
smooth classical solution as long as the latter exists. 
Finally, we prove that in the absence of forcing,  
mean convexity and convexity are preserved by the weak flow. 
\end{abstract}

\maketitle

\section{Introduction}
Denote by $\vert \cdot\vert$
and by $\cH^{n-1}$  the Lebesgue measure and
the $(n-1)$-dimensional Hausdorff measure in $\Rn$, respectively,
and let $B_\rho(x)$ be the open ball centered at $x\in
\Rn$ of radius $\rho>0$.
In \cite{DeGiorgi:93} 
De Giorgi poses the following two conjectures\footnote{We have slightly
modified the notation with respect to the original conjectures.}, 
related to some results in \cite{ATW:1993,GH:1986,Huisken:1984}, and more generally, to mean curvature flow.

\begin{conjecture}\label{conj:3_1}
Let $\convexsets$ be the class of all open convex bounded 
subsets of $\Rn,$ endowed with the metric $d(\convex_1 , \convex_2) = \vert \convex_1 \Delta \convex_2 \vert.$ 
Given $\DeGiorgiexponent \in (0, +\infty)$ and $\initialconvexset
 \in \convexsets$, set
$$
\DeGiorgifunctional(\convex_2;\convex_1,\tau,k) =
\begin{cases}
\vert \convex_1\Delta \initialconvexset
\vert & \text{if $k\le 0,$}\\[2mm]
\displaystyle
\cH^{n-1}(\p 
\convex_2) + 
\frac{1}{\tau^
\DeGiorgiexponent
}\int_{\convex_2\Delta \convex_1}\d_{\convex_1}(x)^\DeGiorgiexponent ~dx & \text{if $k>0$},
\end{cases}
$$
where $\d_F(\cdot):=\dist(\cdot,\p F).$ Then there exists a unique minimizing 
movement $E(\cdot)$ in $\convexsets,$ associated to $\DeGiorgifunctional$ 
and starting from $\initialconvexset$ and, in the case $\DeGiorgiexponent
 = 1$, $\p E(t)$ 
moves along its mean curvature.
\end{conjecture}

\begin{conjecture}\label{conj:3_2}
Let
\begin{equation}\label{eq:def_S}
\cS:=\Big\{E\subset\Rn:\,\,\vert E\vert <+\infty,\,\,x\in E\,\,\Leftrightarrow \,\,\lim\limits_{\rho\to0^+}\rho^{-n}\vert 
B_\rho(x)\setminus E\vert=0\Big\},
\end{equation}
endowed
with the metric $d(E_1 , E_2 ) = \vert E_1\Delta  E_2\vert.$ 
Given $\DeGiorgiexponent \in (0, +\infty),$ $\initialset \in \cS$ such that $\cH^{n-1}(\p \initialset)<+\infty$ and $g\in L^1(\Rn)\cap L^n(\Rn),$ set 
$$
\widetilde \DeGiorgifunctional(E_2;E_1,\tau,k) =
\begin{cases}
\vert E_1\Delta \initialset\vert & \text{if $k\le0,$}\\[2mm]
\displaystyle
\cH^{n-1}(\p E_2) + \int_{E_2}g(x)~dx + \frac{1}{\tau^\DeGiorgiexponent}
\int_{E_2\Delta E_1}\d_{E_1}(x)^\DeGiorgiexponent ~dx  & \text{if $k>0.$}
\end{cases}
$$
Then there 
exists a generalized minimizing movement  in $\cS,$ associated to $\widetilde \DeGiorgifunctional$ and starting from $\initialset.$
\end{conjecture}
 
Clearly, if $\forcing \equiv 0$, then 
$\DeGiorgifunctional
(\convex_2; \convex_1, \tau, k) 
 = \widetilde \DeGiorgifunctional(\convex_2; \convex_1, \tau, k) 
$ for $\convex_1, \convex_2
\in \convexsets$.

\smallskip

In this paper we prove Conjecture \ref{conj:3_2} in a more general form 
(see Theorem \ref{teo:existence_gmm}). We also establish weaker versions of Conjecture \ref{conj:3_1}: (a) minimizing 
$\widetilde
\DeGiorgifunctional$ in $\cS$ 
(rather than in $\convexsets$) 
we obtain a unique minimizing movement (Theorem \ref{teo:gmm_convex_intro});
(b) minimizing 
$\widetilde \DeGiorgifunctional$ in $\convexsets$ 
we obtain the existence of generalized minimizing movements (Theorem \ref{teo:GMM_in_the_class_conv}). Of course, these two results are related; 
in fact, for $\DeGiorgiexponent\in(0,1],$ using the convexity of the map 
$x\mapsto \dist(x,\p \convex)$ for convex sets
$\convex$ and the methods of \cite{CCh:2006}, we can show 
that there exists a 
unique minimizing movement in $\convexsets$ which is also the  unique minimizing movement in $\cS$ (Corollary \ref{cor:de_giorgi_unique}).  At the moment we miss the proof of the uniqueness of generalized minimizing movements for $\DeGiorgiexponent>1.$ 

The gradient flow of $\widetilde \DeGiorgifunctional$ in case $\forcing \equiv 0,$ leads to the equation 
\begin{equation}\label{true_affine_curvature_flow}
\normalvelocityofEt = -\meancurvatureofEt^{1/\DeGiorgiexponent},
\end{equation}
where $\normalvelocityofEt$ and $\meancurvatureofEt$ 
are the normal velocity and the mean curvature of 
an evolving family
$t \to \p E(t)$ of 
smooth closed hypersurfaces 
in $\Rn$;
\eqref{true_affine_curvature_flow} is 
called $\meancurvature^{1/\DeGiorgiexponent}$ (or power of mean curvature) flow \cite{Schulze:2005,Schulze:2006}, and is meaningful also 
in the case of nonconvex sets 
provided 
$\DeGiorgiexponent \in \mathbb N$ is odd  \cite{AST:1998}. 
When $\DeGiorgiexponent=1,$ the evolution equation \eqref{true_affine_curvature_flow} 
is the classical mean curvature flow. When $\DeGiorgiexponent=3$ and $\p E(t)$ are evolving curves 
in $\R^2$, \eqref{true_affine_curvature_flow} 
is called the \emph{affine curvature flow} 
(see e.g. \cite{AL:1986,AGLM:1993,Andrews:2003,
AST:1998,Chen:2015,
ChZh:2001,
Giga:2006,
ST:1994} and 
references therein) because of the invariance of the flow with respect to affine transformations of coordinates; this equation has 
applications in image processing \cite{AGLM:1993,Cao:2003_b}.
Depending on  $\DeGiorgiexponent$, various phenomena may occur. If $\DeGiorgiexponent<8,$ the only embedded homothetically shrinking 
solutions are circles, except when $\DeGiorgiexponent=3,$ where some self-shrinking ellipses occur, while if $\DeGiorgiexponent\ge 8,$ a new family of symmetric curves resembling either circles or polygons arise  (see e.g. \cite{AL:1986,Andrews:2003,ChZh:2001}).
The flow  \eqref{true_affine_curvature_flow} exists, for instance, 
in the case of bounded convex 
initial subsets of $\Rn$ \cite{Schulze:2005,Schulze:2006}, with 
finite-time extinction towards a point. See also \cite{DNV:2024}, 
for an asymptotic study  
of a time-fractional Allen-Cahn equation and its convergence to the power mean curvature flow.

A by-product of our analysis is the study of  a 
generalization of \eqref{true_affine_curvature_flow} of the form
\begin{equation}\label{affine_curvature_flow}
\affine(\normalvelocityofEt) =
- \anisotropicmeancurvatureofEt
- \forcing,
\end{equation}
where $\norm$ is an (even) anisotropy in $\Rn,$ 
$\anisotropicmeancurvatureofEt$ 
is the anisotropic mean curvature of $\p E(t),$ 
and 
$\affine: \R \to \R$ is a strictly increasing continuous odd surjective
function.
Clearly, in general we cannot expect invariance of solutions to \eqref{affine_curvature_flow}
with respect to affine transformations; we refer to \eqref{affine_curvature_flow} 
as the \emph{generalized power mean curvature flow}.

In this paper we study minimizing movement solutions 
corresponding to \eqref{affine_curvature_flow} under quite general assumptions on $\affine$ and $\forcing$. Following \cite{ATW:1993,ChDgM:2023,DeGiorgi:93} we introduce an Almgren-Taylor-Wang type functional  -- generalizing the functional\footnote{We slightly abuse the notation: 
in contrast to the functionals $\DeGiorgifunctional$,
$\widetilde \DeGiorgifunctional$ in Conjectures \ref{conj:3_1} and \ref{conj:3_2}, we do not highlight the dependence on $k$.} in Conjecture \ref{conj:3_2}:
$$
\generalfunctional
(E;F,\tau):=\anisotropicperimeter(E)+
\int_{E\Delta F}\veloc{F}{x}{\tau}~dx 
+ 
\int_{E}\forcing(x)~ dx
$$
with domain 
$$
\cS^*:= BV(\Rn;\{0,1\})\cap \cS,
$$
where $\tau>0$ and
\begin{equation}\label{eq:anisotropic_perimeter}
\anisotropicperimeter(E):=\int_{\p^*E}\dualnorm(\nu_E)~d\cH^{n-1}
\end{equation}
is the $\norm$-perimeter of $E.$ Here $\p^*E$ and $\nu_E$ are the reduced boundary and the generalized outer unit normal of $E.$  
When $\norm(\cdot) = \vert \cdot\vert$ is the Euclidean norm, we write $P$
in place of $\anisotropicperimeter$; localization of the 
perimeter in a Borel set $\Borelset$ is denoted by $P(\cdot, \Borelset)$.
When $\norm$ is Euclidean and $\affine(r)=r$  
we recover the standard Almgren-Taylor-Wang functional with a driving force \cite{ChDgM:2023}. 
In what follows, if no confusion is possible, 
\begin{equation}\label{eq:we_shorthand}
\textrm{we~ shorthand}~
\generalfunctional \textrm{~with~ the~ symbol~}
\shorthandgeneralfunctional.
\end{equation}
Note that when $\norm$ is Euclidean and $\forcing\equiv0,$ we have,
for $B_\varrho = \{\xi \in \Rn : \vert \xi\vert < \varrho\}$,
$$
\shorthandgeneralfunctional(B_r;B_{r_0},\tau) = n\omega_n r^{n-1} + n\omega_n\int_0^r \affine\Big(\tfrac{s-r_0}{\tau}\Big)\,s^{n-1}ds =:n\omega_n \ell(r)
$$
for any $r,r_0>0.$ Since $\ell$ is differentiable, 
$$
\ell'(r) = (n-1)r^{n-2} + r^{n-1} \affine\Big(\tfrac{r-r_0}{\tau}\Big),\quad r > 0.
$$
Thus, $r>0$ is a critical point if and only if 
\begin{equation}\label{masul_idoralar}
\frac{n-1}{r} +  \affine\Big(\tfrac{r-r_0}{\tau}\Big) = 0\quad \Longleftrightarrow\quad 
\frac{r - r_0}{\tau} = -\affine^{-1}\Big(\tfrac{n-1}{r} \Big).
\end{equation}
Using this in  Theorem \ref{teo:evolving_ball} we show that balls shrink self-similarly and their radii satisfy 
$$
R'(t) = -\affine^{-1}\Big(\tfrac{n-1}{R(t)}\Big)\quad {\rm if}~ R(t)>0,
$$
consistently with \eqref{affine_curvature_flow}. See Section 
\ref{sec:rescalings_of_GMM_and_comparison_with_balls} for more. 

The next definition is a particular 
case of a definition given in \cite{DeGiorgi:93}.
\begin{definition}[\textbf{Flat flows, $\GMM$ and $\MM$}]
A family $\{E(t)\}_{t\in[0,+\infty)}\subset \cS^*$ 
is called a \emph{generalized minimizing movement} (shortly, $\GMM$) in $\cS^*$, associated to 
$\shorthandgeneralfunctional$ and starting from $\initialset
\in S^*,$ if there exist 
a sequence $\tau_j\to0^+$ and a family $\{E(\tau_j,k)\}_{j,k\ge0}$ of sets, so-called flat flows, defined as 
$E(\tau_j,0) = \initialset
$ and 
$$
E(\tau_j,k)\in {\rm argmin}\,
\shorthandgeneralfunctional
(\cdot;E(\tau_j,k-1),\tau_j),\quad k,j\ge1,
$$
such that for any $t\ge0$
\begin{equation}\label{partenza_binorio9}
E(\tau_j,\intpart{t/\tau_j}) \to E(t) \quad \text{in $L_\loc^1(\Rn)$ as $j\to+\infty,$}
\end{equation}
where $\intpart{x}$ is the integer part of $x\in\R.$ When $E(\cdot)$ in \eqref{partenza_binorio9} is independent of the sequence $(\tau_j),$ i.e., the limit holds as $\tau\to0^+,$ then it 
is called a minimizing movement (shortly MM) in $\cS^*,$ 
associated to $\shorthandgeneralfunctional$ 
and starting from $\initialset.$ The set of $\GMM$ and $\MM$ in $\cS^*$ 
will be denoted, respectively, as 
$\GMM(\shorthandgeneralfunctional,\cS^*,\initialset)$ and $\MM(\shorthandgeneralfunctional,\cS^*,\initialset).$
\end{definition}


Now, we list a number of assumptions on $\affine$ and $\forcing$ needed 
in the sequel.
\bigskip 

\begin{theoremA}[(H)]\label{hyp:main}
$\,$%
\begin{itemize}
\item[\rm (Ha)] $\affine:\R\to\R$ 
is a strictly increasing, continuous, surjective 
odd function;
\item[\rm(Hb)] for any $\varsigma_0,\varsigma_1>0$ and any
$\tau>0,$ the unique solution $(\rho_\tau,r_\tau)$ of the system 
\begin{equation}\label{good_density_property}
\begin{cases}
\rho \affine\Big(\tfrac{\rho}{\tau}\Big) = \varsigma_0,\\
r \affine\Big(\tfrac{r+2\rho}{\tau}\Big) = \varsigma_1,\\
\rho,r>0,
\end{cases}
\end{equation}
satisfies\footnote{Note that $r_\tau$ depends on $\varsigma_0,$ $\varsigma_1$ and $\rho_\tau.$} 
\begin{equation}\label{lower_bound_r_tau}
\liminf\limits_{\tau\to0^+}\frac{r_\tau}{\tau}\in (0,+\infty];
\end{equation}

\item[\rm(Hc)] $\forcing\in L^1(\Rn)\cap L^\forcingexponent(\Rn)$ 
for some $\forcingexponent
\in [n, +\infty].$
\end{itemize}
\end{theoremA}

When 
$\DeGiorgiexponent=1,$ we have 
$\rho_\tau = \sqrt{\varsigma_0}\, \tau^{1/2}$ and $r_\tau = (\sqrt{\varsigma_0+4\varsigma_1}- \sqrt{\varsigma_0})\tau^{1/2}.$
More generally, when $r>0$ and 
$\affine(r)=r^\DeGiorgiexponent$ for some $\DeGiorgiexponent>0,$ 
we can explicitly find $\rho_\tau$ in $\eqref{good_density_property}:$ $\rho_\tau=\varsigma_0^\frac{1}{1+\DeGiorgiexponent}\tau^{\frac{\DeGiorgiexponent}{1+\DeGiorgiexponent}}.$ Moreover, $r_\tau>\varsigma_1^{1/\DeGiorgiexponent}\tau^{\frac{\DeGiorgiexponent}{1+\DeGiorgiexponent}}$ and by the H\"older inequality and the explicit expression of $\rho_\tau$
$$
\varsigma_1^{1/\DeGiorgiexponent}\tau = r_\tau^{\frac{1+\DeGiorgiexponent}{\DeGiorgiexponent}}+2\rho r_\tau^{1/\DeGiorgiexponent}\le  r_\tau^{\frac{1+\DeGiorgiexponent}{\DeGiorgiexponent}} + \tfrac{2\DeGiorgiexponent\epsilon}{1+\DeGiorgiexponent}\tau + \tfrac{1}{(1+\DeGiorgiexponent)\epsilon} r_\tau^{\frac{1+\DeGiorgiexponent}{\DeGiorgiexponent}}
$$
and thus, choosing $\epsilon>0$ small enough we find $r_\tau\ge \varsigma_2\tau^{\frac{\DeGiorgiexponent}{1+\DeGiorgiexponent}}$ for some $\varsigma_2>0.$ In particular $r_\tau$ satisfies \eqref{lower_bound_r_tau}.
See also Example \ref{ex:fisaffine}. 

Our first result is the following

\begin{theorem}[\textbf{Existence, time-continuity and bounds of $\GMM$}]
\label{teo:existence_gmm}
Assume Hypothesis \eqref{hyp:main} and use the notation
\eqref{eq:we_shorthand}. Then:

\begin{itemize}
\item[\rm(i)] For any $\initialset\in \cS^*,$ 
$\GMM(\shorthandgeneralfunctional,\cS^*,\initialset)$ is nonempty. 
Moreover, there exists a continuous strictly increasing function $\modulusofcontinuity:[0,+\infty)\to
[0,+\infty)$ 
(depending only on $\affine,$ $n,$ 
the constants $\constnorm$ and
$\Constnorm$ 
in \eqref{norm_bounds},  
$\anisotropicperimeter(\initialset)$ and $\|\forcing^-\|_{L^1(\Rn)}$) 
with $\modulusofcontinuity(0)=0$ 
such that for any $E(\cdot)\in 
\GMM(\shorthandgeneralfunctional,\cS^*,\initialset)$
\begin{equation}\label{unif_conto_GMM}
|E(t)\Delta E(s)| \le \modulusofcontinuity(|t-s|),\quad s,t>0.
\end{equation}
If, additionally, $|\p \initialset|=0,$ then  \eqref{unif_conto_GMM} holds for all $s,t\ge0.$ Furthermore, if 
\begin{equation*}
\exists \lim_{r\to+\infty} \frac{\affine(r)}{r^\DeGiorgiexponent} \in (0,+\infty)
\end{equation*}
for some $\DeGiorgiexponent>0,$ then $\modulusofcontinuity$ can be chosen locally  $\frac{\DeGiorgiexponent}{1+\DeGiorgiexponent}$-H\"older continuous.

\item[\rm(ii)] If $\initialset$ is bounded and 
\begin{equation}\label{forcing_linear_growth}
\forcing^-(x) \le \affine(\constantgminus(1+|x|))\quad \text{for all $|x|>\constantgminus$}
\end{equation}
for some sufficiently large constant $\constantgminus>0,$ then each $\GMM$ is locally bounded, 
i.e., for any $T>0$ there exists $R_T>0$ such that for any $E(\cdot)\in 
\GMM(\shorthandgeneralfunctional,\cS^*,\initialset)$
$$
E(t)\subset B_{R_T}(0)\quad\text{for all $t\in [0,T].$}
$$
\end{itemize}
\end{theorem}

Some comments are in order.

\begin{itemize}
\item The oddness of $\affine$ allows 
to reduce $\shorthandgeneralfunctional$ to a prescribed mean 
curvature functional: when
$\affine$ is odd,  we can write
\begin{equation}\label{ofortoo}
\int_{E\Delta F} \veloc{F}{x}{\tau}~dx = \int_E \affine\Big(\tfrac{\sd_F}{\tau}\Big)~dx - \int_F\affine\Big(\tfrac{\sd_F}{\tau}\Big)~dx
\end{equation}
whenever $E\cap F$ has finite measure,
and $\sd_F$ is the signed distance from $\partial F$
{\it positive outside} $F$.
  Indeed, one can readily check that $\chi_G\sd_G\in L^1(\Rn)$ for any $G$ with $|G|<+\infty.$ Thus, 
\begin{align*}
\int_E \affine\Big(\tfrac{\sd_F }{\tau}\Big)~dx - \int_F\affine\Big(\tfrac{\sd_F }{\tau}\Big)~dx =  &
\int_{E\setminus F} \affine\Big(\tfrac{\sd_F }{\tau}\Big)~dx - \int_{F\setminus E}\affine\Big(\tfrac{\sd_F }{\tau}\Big)~dx \\
= & \int_{E\setminus F} \affine\Big(\tfrac{\d_F }{\tau}\Big)~dx + \int_{F\setminus E}\affine\Big(\tfrac{\d_F }{\tau}\Big)~dx. 
\end{align*} 
Therefore, similarly to the classical Almgren-Taylor-Wang functional,
\begin{equation}\label{eq:prescur_funco}
\shorthandgeneralfunctional(E;F,\tau)= 
\anisotropicperimeter(E)+\int_{E}\sveloc{F}{x}{\tau}~dx + \int_{E}\forcing(x) ~dx + 
\constantprescribed,
\end{equation}
where $\constantprescribed$ is a constant independent of $E$.
In view of \eqref{eq:prescur_funco}, the minimization 
problem for $\shorthandgeneralfunctional(\cdot; F,\tau)$ 
is equivalent to the minimization of the  prescribed mean curvature functional $E\in\cS^*\mapsto \anisotropicperimeter(E) + \int_E
h_{\affine,F,\forcing}~dx,$ with
\begin{equation}\label{eq:h} 
h_{\affine,F,\forcing}:=\affine\Big(\tfrac{\sd_F}{\tau}\Big) + \forcing.
\end{equation}

\item (\ref{hyp:main}a) and (\ref{hyp:main}c) suffice for the 
well-definiteness and the $L_\loc^1(\Rn)$-lower semicontinuity of 
$\shorthandgeneralfunctional$, as well as for the existence of 
minimizers, and in particular, of flat flows (Lemma
\ref{lem:existence_minima}).

\item Assumption $\forcing\in L^\forcingexponent(\Rn)$ in (\ref{hyp:main}c) is 
used to establish the uniform density estimate
\begin{equation}\label{uniform_per_density_estos}
\frac{P(E,B_r(x))}{r^{n-1}}
\ge
\lowerbddensity,\quad x\in \p E,\quad r\in (0,r_\tau], 
\end{equation}
for minimizers $E$ of $\shorthandgeneralfunctional(\cdot;F,\tau),$ 
where $\lowerbddensity>0$ is independent of $E,F$ and $\tau$ (see 
Sections \ref{subsec:L_infty_bound_for_d_F},  \ref{subsec:density_estimates} and the inequality \eqref{eq:lower_perimeter_density_estimate}).

\item Assumption \eqref{lower_bound_r_tau} together with \eqref{uniform_per_density_estos} and Lemma \ref{lem:volume_distance} applied with $\ell=\tau$ and $p>0,$ imply that any flat flow $\{E(\tau,k)\}$ satisfies 
$$
|E(\tau,k)\Delta E(\tau,k-1)| \le Cp^\sigma \tau \anisotropicperimeter(E(\tau,k-1)) + \frac{1}{\affine(p)}\int_{E(\tau,k)\Delta E(\tau,k-1)} \veloc{E(\tau,k-1)}{x}{\tau}~dx
$$
for all $k\ge2$ and for some $\sigma\in\{1,n\},$ where $C$ is a coefficient depending only on $n,$ $\norm,$ $\theta$ and the liminf of $r_\tau/\tau$ in \eqref{lower_bound_r_tau}. This estimate, with a suitable choice of $p$ (see \eqref{suitable_p}) yields the almost uniform time-continuity of the flat flow $t\mapsto E(\tau,\intpart{t/\tau})$ (see \eqref{disc_holtime}), which in turn implies the existence of $\GMM$ and 
the validity of \eqref{unif_conto_GMM}.

\item If $\affine(r)=r^\DeGiorgiexponent$ for $r>0$ and 
for some $\DeGiorgiexponent>0,$ then $\modulusofcontinuity$ in \eqref{unif_conto_GMM} can be chosen as $\modulusofcontinuity(t)=t^{\frac{\DeGiorgiexponent}{1+\DeGiorgiexponent}},$ see also \eqref{def:unif_conto}. 

\item The strict  monotonicity and surjectivity
of $\affine$ are needed for the unique solvability of system \eqref{good_density_property}. At the moment, 
we do not know what happens if these assumptions are dropped.

\item Without assumption \eqref{lower_bound_r_tau} our method fails to apply for the existence of $\GMM$ (see Remark \ref{rem:bad_upper_bound}). However, we do not know whether $\GMM$ exists or not 
without this assumption.

\item
If 
$\affine(r) \sim r^\DeGiorgiexponent$ as $r\to+\infty$ for some $\DeGiorgiexponent>0,$ 
then $\affine$ satisfies (\ref{hyp:main}b), whereas if $\affine$ has exponential growth, then \eqref{lower_bound_r_tau} may fail (see Examples \ref{ex:fisaffine} and \ref{ex:badexample}). 
\end{itemize}

To prove the existence and uniform time continuity of 
$\GMM$ in Theorem \ref{teo:existence_gmm} we follow the 
standard Almgren-Taylor-Wang method in \cite{ATW:1993,ChDgM:2023,LS:1995}.
However, the presence of $\affine$ requires some extra care in techniques. 
Furthermore, unlike these papers, for the existence of $\GMM$ we do not assume a priori the boundedness of the initial sets and also of 
$\forcing$. 
Finally, the uniform boundedness of $\GMM$ 
will be done employing the 
isoperimetric properties of (anisotropic) 
balls. This can be also done following \cite[Lemma 3.9]{ChDgM:2023}, where a time-dependent bounded $\forcing$ is considered. 

It is well-known \cite[Theorem 12]{BCChN:2005} 
that the classical mean curvature flow preserves 
mean convexity;
our next main qualitative result  is a similar preservation
(Section \ref{sec:evolution_of_mean_convex_sets}), which to some extent generalizes \cite{ChN:2022}.

\begin{theorem}[\textbf{Mean convex evolution}]\label{teo:meanconvex_intro}
Suppose that $\affine$ satisfies (\ref{hyp:main}a), (\ref{hyp:main}b) and $g\equiv0$.
Then  any $E(\cdot) \in \GMM(\shorthandgeneralfunctional,\cS^*,\initialset)$ starting from a bounded $\delta$-mean convex set $\initialset\Subset\Omega$ (in some open set $\Omega$) is itself a flow of $\delta$-mean convex sets in $\Omega.$ Moreover, the maps $t\mapsto E(t)$ and $t\mapsto \anisotropicperimeter(E(t))$ are nonincreasing.
\end{theorem}

To prove this theorem we partially follow the ideas of \cite{ChN:2022}; 
in order to show the $\delta$-mean convexity of minimizers of $\shorthandgeneralfunctional,$ the authors of \cite{ChN:2022} used the so-called Chambolle scheme for mean  curvature flow and the Anzelotti-pairings, while here we employ comparison properties for the prescribed mean curvature functional. 
Classical mean curvature flow
even preserves convexity  \cite{Huisken:1984}.
Also,
 the $\GMM$ starting from a convex set is unique (which positively answers to the last assertion of 
Conjecture \ref{conj:3_1} when $\DeGiorgiexponent=1$), see \cite{BCChN:2005}. The following result 
shows the validity of this property also in the 
generalized power mean curvature flow setting 
(see Section \ref{sec:evolution_of_bounded_convex_sets}). 

\begin{theorem}[\textbf{Evolution of convex sets and stability}]
\label{teo:gmm_convex_intro}
Let $\norm$ be the Euclidean norm, $\forcing \equiv 0$
and $f(r) = r^\DeGiorgiexponent$ for $r>0$, with $\DeGiorgiexponent >0$.
Let $\initialconvexset \in \convexsets$. 
Then:
\begin{itemize}
\item[\rm(i)]
$\GMM(\shorthandgeneralfunctional,\cS^*,\initialconvexset)= 
\MM(\shorthandgeneralfunctional,\cS^*,\initialconvexset) = \{\convex(\cdot)\}$ 
and $\convex(t)$ is convex for any $t\ge0.$ 
Moreover, if $\initialconvexset$ is smooth, then $\convex(\cdot)$ 
is the smooth convex power mean 
curvature evolution starting from $\initialconvexset$ 
(see Theorem \ref{teo:convex_smooth_evol});

\item[\rm(ii)] Let $\initialconvexset \in \convexsets$ and let 
$(\initialconvexsetindex)$ 
be any sequence of sets such that $\p 
\initialconvexsetindex\to \p \initialconvexset$ in the Kuratowski sense and $\{P(\initialconvexsetindex)\}$ is uniformly bounded. Let $\convex_\naturalindex(\cdot)
\in \GMM(\shorthandgeneralfunctional,\cS^*,\initialconvexsetindex),$ 
and let $\{\convex(\cdot)\}$ 
be the unique minimizing movement starting from $\initialconvexset$. Then 
$$
\lim\limits_{\naturalindex\to+\infty}|\convex_\naturalindex(t)\Delta \convex(t)| = 
0\quad\text{for any $t\ge0.$}
$$
\end{itemize}
\end{theorem}
Since any convex set is also mean convex, 
from Theorem \ref{teo:meanconvex_intro} it follows 
that each $\convex(t)$ is mean convex and $t\mapsto \convex(t)$ is nonincreasing.

We expect similar uniqueness and stability properties for generalized curvature flow of convex sets also in the anisotropic case with more general $\affine$. However, we leave
this problem for future investigations.

Theorem \ref{teo:gmm_convex_intro}  is a weak formulation of Conjecture \ref{conj:3_1}, where the minimization problem for $\shorthandgeneralfunctional$ is conducted in the larger class $\cS^*$ rather than in the class $\convexsets$. Indeed, in $\convexsets$ we cannot apply the cutting-filling with balls argument used in the proof of Theorem \ref{teo:existence_gmm}; rather using minimal cutting properties of convex sets we can show that the flat flows in $\convexsets$ and their perimeter 
have the following monotonicity: $E(\tau,0)\supset E(\tau,1)\supset \ldots$ and $P(E(\tau,0))\ge P(E(\tau,1))\ge \ldots.$ 
Thus, the map $t\mapsto P(E(\tau,\intpart{t/\tau}))$ is bounded and nonincreasing, and therefore, sequentially compact (w.r.t. $\tau$) in $L_\loc^1(\R_0^+).$ Now using convexity we conclude that 
every limit point is indeed a $\GMM$ (Section \ref{sec:minimizing_movements_in_the_class_conv}):

\begin{theorem}[\textbf{$\GMM$ in the class $\convexsets$}]\label{teo:GMM_in_the_class_conv}
Assume $\norm$ 
is an anisotropy in $\Rn,$ $\affine$ 
satisfies Hypothesis (\ref{hyp:main}a) and (\ref{hyp:main}b) and 
$\forcing\equiv0.$ 
For any $\initialconvexset\in \convexsets$, $\GMM(
\shorthandgeneralfunctional,\convexsets,\initialconvexset)$ is nonempty. 
Moreover, if $\norm$ is Euclidean and $f(r)=r^\DeGiorgiexponent$ for $r>0$ and some $\DeGiorgiexponent\in (0,1],$ then $\GMM(
\shorthandgeneralfunctional,\convexsets,\initialconvexset)$ is a  singleton and concides with the unique minimizing movement in $\MM(
\shorthandgeneralfunctional,\cS^*,\initialconvexset).$
\end{theorem}

We observe that we miss the proof of uniqueness of the minimizing movement in $\convexsets$ for $\DeGiorgiexponent>1$. 

Our last result is 
consistency of $\GMM$ with the smooth classical solution 
of \eqref{affine_curvature_flow}\footnote{Thus, in this case,
$\GMM$ is unique for short times.}. Assuming the latter exists and is stable (in the sense of Definition \ref{def:stable_smooth_flow}), following 
some of the ideas of \cite{ATW:1993} and \cite{Kholmatov:2024} 
we show:

\begin{theorem}[\textbf{Consistency}]\label{teo:consistency}
Suppose that $f \in C^\Holderexponent(\R)$ and $g \in C^{\Holderexponent}(\Rn)$
for some $\Holderexponent \in (0,1]$.
Assume that $\norm$ is an elliptic $C^3$-anisotropy and  \eqref{affine_curvature_flow} admits a unique smooth stable solution $\{S(t)\}_{t\in[0,T)}$.  
Then 
for every $E(\cdot) \in \GMM(\shorthandgeneralfunctional,\cS^*,S(0))$ 
we have
$$
E(t) = S(t)\quad\text{for all $t\in [0,T).$}
$$
\end{theorem}

As it happens for the classical consistency
proof for the Almgren-Taylor-Wang functional,
the proof of this nontrivial theorem heavily relies on the stability of the flow, strong comparison principles and discrete comparison principles. 

The paper is organized as follows. In 
Section \ref{sec:some_notation_and_examples_of_affine},
after setting the notation, 
 we quickly shows some examples of interesting functions
$\affine$. 
In Section 
\ref{sec:proof_of_Theorem_teo:existence_gmm}
 we prove Theorem \ref{teo:existence_gmm}. Various comparison results 
are  studied in Section \ref{sec:rescalings_of_GMM_and_comparison_with_balls}. 
The evolution of mean convex sets and convex sets are considered in 
Sections \ref{sec:evolution_of_mean_convex_sets},
\ref{sec:evolution_of_bounded_convex_sets}, and \ref{sec:minimizing_movements_in_the_class_conv}.
The consistency of GMM with smooth solutions is proven in Section 
\ref{sec:consistency_with_smooth_flows}. Finally, we conclude the paper with two appendices, where 
we establish some technical results, needed in various proofs. 

\smallskip
Shortly after the conclusion of this paper,
we became aware of 
the paper \cite{DeGennaro:2024}, where
the author addresses a similar problem, 
showing existence of level set solutions to the power mean curvature flow,
via the minimizing movements. That paper appears to be completely independent
of the present paper. 

\subsection*{Acknowledgements}
G. Bellettini acknowledges support from PRIN 2022PJ9EFL ``Geometric Measure Theory: Structure of Singular Measures, Regularity Theory 
and Applications in the Calculus of Variations'', and from GNAMPA (INdAM). Sh. Kholmatov acknowledges support from the Austrian Science Fund (FWF) Stand-Alone project P 33716. 

\section{Notation and examples of functions $\affine$}\label{sec:some_notation_and_examples_of_affine}
We write $\R^+:=(0,+\infty),$ $\R_0^+:=[0,+\infty)$ and $\N_0=\N\cup\{0\}.$
The 
inclusion
 $E\Subset F$ means that 
$E$ is bounded, $E \subset F$ and $\dist(\p E,\p F)>0;$
we also set $E^c := \Rn \setminus E$. 
Unless otherwise stated, $B_r:=B_r(0).$ In view of the definition of $\cS$ in \eqref{eq:def_S}, every set $E$ we consider 
coincides with the set $E^{(1)}$ of its points of density $1$. 
Therefore $\p E = \cl{\p^* E}$  and $\d_E(\cdot) =\dist(\cdot,\p E) = \dist(\cdot,\p^* E).$ We denote by $\hull{E}$ the closed convex hull of $E.$  
An anisotropy $\norm: \Rn \to [0,+\infty)$ 
is  a positively one-homogeneous even convex function 
equivalent to the Euclidean norm;
its dual is defined as 
$$
\dualnorm(\eta) = \sup_{\norm(\xi)=1} \,\,\xi\cdot \eta=1,
$$
so that $\dualnorm$ is an anisotropy and 
\begin{equation}\label{norm_bounds}
c_\norm |\xi| \le \dualnorm(\xi) \le C_\norm |\xi|,\quad \xi\in\Rn,
\end{equation}
for some $C_\norm\ge c_\norm>0$. 

Recall the following anisotropic isoperimetric inequality \cite{Maggi:2012}:
\begin{equation}\label{aniso_isop_ineq}
\anisotropicperimeter(E) \ge \constiso
|E|^{\frac{n-1}{n}},\quad 
\constiso
:= \frac{\anisotropicperimeter(\unitballnorm)}{|\unitballnorm|^\frac{n-1}{n}},
\end{equation}
where $P_\norm$ 
is defined in \eqref{eq:anisotropic_perimeter} and\footnote{Sometimes
$\unitballnorm$ is called Wulff shape.} 
$\unitballnorm:=\{\norm\le 1\}$.
In the Euclidean case $\norm (\cdot) = \vert \cdot \vert$, $c_{n,|\cdot|} = n\omega_n^{1/n}$, $B = B_{\vert \cdot \vert}$,
and $\anisotropicmeancurvature=\meancurvature$. Our conventions are: 
the mean curvature is positive for (uniformly) convex sets, 
and the normal velocity is increasing for expanding sets.

Let us give some examples of function $\affine$.

\begin{example}[\textbf{$\affine$ a power}]\label{ex:fisaffine}
Assume that $\affine$ is strictly increasing, continuous and 
\begin{equation}\label{tahliqilmas}
c_1r^{\DeGiorgiexponent_1} \le \affine(r)\le c_2r^{\DeGiorgiexponent_2}\quad\text{for all large $r>1$}
\end{equation}
for some $0<c_1\le c_2$ and $0<\DeGiorgiexponent_1 \le 
\DeGiorgiexponent_2
\le \DeGiorgiexponent_1+1.$ 
Then for any $\tau>0$ the system \eqref{good_density_property} is solvable and $r_\tau$ satisfies \eqref{lower_bound_r_tau}. Indeed, solvability of \eqref{good_density_property} follows from the continuity and the monotonicity of $f.$ In particular, by the inverse monotone function theorem both maps $\tau\mapsto \rho_\tau$ and $\tau\mapsto r_\tau$  are continuous and increasing. 
Moreover, clearly, $\rho_\tau,r_\tau\to0$ as $\tau\to0^+.$ Now by \eqref{tahliqilmas} 
and the equality $\rho_\tau \affine(\rho_\tau/\tau) = C_0$ 
\begin{equation}\label{amasu87}
c_1' \tau^{\frac{\DeGiorgiexponent_2}{1+\DeGiorgiexponent_2}}
\le \rho_\tau < c_2' \tau^{\frac{\DeGiorgiexponent_1}{1+\DeGiorgiexponent_1}}
\end{equation}
for all small $\tau>0$ and for some $c_2'\ge c_1'>0$ independent of $\tau.$ 
Similarly, from the equality $r_\tau \affine(\frac{r_\tau+2\rho_\tau}{\tau}) = C_1$ and estimates \eqref{tahliqilmas} and \eqref{amasu87} we get 
\begin{equation*}
c_1'' \tau \le r_\tau^{1/\DeGiorgiexponent_2} (r_\tau + 2c_2' 
\tau^{\frac{\DeGiorgiexponent_1}{1+\DeGiorgiexponent_1}}) = 
r_\tau^{1+1/\DeGiorgiexponent_2} + c_2' r_\tau^{1/\DeGiorgiexponent_2}
\tau^{\frac{\DeGiorgiexponent_1}{1+\DeGiorgiexponent_1}}
\end{equation*}
for all small $\tau>0$ and for some $c_1''>0.$ 
This inequality implies either $\frac{\tau}{r_\tau^{1+1/\DeGiorgiexponent_2}}$ 
or $\frac{\tau}{r_\tau^{1/\DeGiorgiexponent_2}\tau^{\frac{\DeGiorgiexponent_1}{1+\DeGiorgiexponent_1}}}$ is uniformly bounded. 
Since $\DeGiorgiexponent_1\le\DeGiorgiexponent_2\le 1+\DeGiorgiexponent_1,$ either condition implies $r_\tau \ge c\tau$ for some $c>0$ and all small $\tau>0.$ Therefore, \eqref{lower_bound_r_tau} holds.
\end{example}

\begin{example}[\textbf{$\affine$ an exponential}]\label{ex:badexample}
Let $\affine(r)= e^r$ for large $r>0.$ 
Then $\rho_\tau$ in \eqref{good_density_property} satisfies 
\begin{equation}\label{perim_shahri}
\rho_\tau = -\tau \ln\rho_\tau + \tau \ln\varsigma_0.
\end{equation}
For sufficiently small $\tau>0$ setting 
$
\rho_\tau = \tau\ln(\tau^{-1}) (1+u)
$
in \eqref{perim_shahri} we get 
$$
F(u,x,y):= u +x + [\ln(1+u) - \ln\varsigma_0]y=0,\quad x: = \tfrac{\ln\ln(\tau^{-1})}{\ln(\tau^{-1})},\quad y:=\tfrac{1}{\ln(\tau^{-1})}.
$$
This analytic implicit function admits a unique solution $u=w(x,y),$ where $w(\cdot,\cdot)$ is some real analytic function in a neighborhood of $(0,0)$, and thus, for 
sufficiently small $\tau>0,$
\begin{equation}\label{explicit_rhotau}
\rho_\tau = \tau\ln(\tau^{-1})\Big[1+ w\Big(
\tfrac{\ln\ln(\tau^{-1})}{\ln(\tau^{-1})},
\tfrac{1}{\ln(\tau^{-1})}
\Big)\Big]. 
\end{equation}
Moreover, since $r_\tau$ satisfies the equation
$$
\frac{r_\tau + 2\rho_\tau}{\tau} = \ln \varsigma_1 + \ln (r_\tau^{-1}),
$$
we have 
\begin{equation}\label{adsfndf879}
\frac{r_\tau}{\tau\ln (r_\tau^{-1}) } 
+ \frac{\rho_\tau}{\tau\ln(\rho_\tau^{-1})}\frac{2\ln (\rho_\tau^{-1})}{\ln (r_\tau^{-1})} = 1 +\frac{\ln \varsigma_1}{\ln (r_\tau^{-1})}.
\end{equation}
In view of \eqref{perim_shahri} 
$$
\lim\limits_{\tau\to0^+} \frac{\rho_\tau}{\tau\ln(\rho_\tau^{-1})} = 1.
$$
Therefore, using $\frac{r_\tau}{\tau\ln (r_\tau^{-1}) }\ge0$ and  recalling that $r_\tau\to 0^+,$ from \eqref{adsfndf879} we get 
$$
\limsup_{\tau\to0} \frac{2\ln (\rho_\tau^{-1})}{\ln (r_\tau^{-1})} \le 1.
$$
In particular, for for sufficiently small $\tau>0,$ 
$$
\frac{2\ln (\rho_\tau^{-1})}{\ln (r_\tau^{-1})} \le \frac{3}{2}.
$$
This and \eqref{explicit_rhotau} imply
$$
\frac{r_\tau}{\tau} \le \frac{\rho_\tau^{4/3}}{\tau} = \tau^{1/3}[\ln(\tau^{-1}) (1+w)]^{4/3} \to0
$$
as $\tau \to 0^+.$ Hence, $\affine$ does not satisfy \eqref{lower_bound_r_tau}.
\end{example}

Throughout the paper we always assume $n\ge2.$ We refer to \cite{Pascali:1996} for the case $n=1$. 

\section{Proof of Theorem \ref{teo:existence_gmm}}
\label{sec:proof_of_Theorem_teo:existence_gmm}
This section is devoted to the proof of Theorem \ref{teo:existence_gmm}. 
Assume as usual the convention \eqref{eq:we_shorthand}.  

\begin{lemma}\label{lem:existence_minima}
Suppose that $\affine$ satisfies 
(\ref{hyp:main}a) and that $\forcing \in L^1(\Rn)$.
Then for any $F\in \cS^*$ and $\tau>0$  
there exists a minimizer of 
$\shorthandgeneralfunctional
(\cdot;F,\tau)$ in $\cS^*$. 
\end{lemma}

As in the case of prescribed mean curvature functional, the assumption $\forcing\in L^1(\Rn)$ can be relaxed to $\forcing^-\in L^1(\Rn)$.

\begin{proof}
Let us study the minimum problem 
$$
\inf_{E\in \cS^*}\,\,\, \shorthandgeneralfunctional(E;F,\tau).
$$
Let $(E_i)\subset \cS^*$ be a minimizing sequence. We may assume 
$$
\shorthandgeneralfunctional(E_i;F,\tau)\le
\shorthandgeneralfunctional
(F;F,\tau)\quad\text{for all $i\ge1.$}
$$
Then 
$$
\anisotropicperimeter(E_i) + \int_{E_i\Delta F} \veloc{F}{x}{\tau} ~dx 
\le \shorthandgeneralfunctional
(F;F,\tau) + \int_{\Rn} |\forcing | ~dx:=C.
$$
In particular, by the nonnegativity of $\affine$ in $\R_0^+,$ the sequence $(\anisotropicperimeter(E_i))$ is 
bounded, 
and hence, by compactness, there exists $E\in BV_\loc(\Rn;\{0,1\})$ such that, up to a not  relabelled subsequence, $E_i \to E$ in $L^1_\loc(\Rn)$ as $i\to+\infty.$ By the $L_\loc^1$-lower semicontinuity of the anisotropic
perimeter, 
$$
\anisotropicperimeter(E) \le \liminf_{i\to+\infty} \anisotropicperimeter(E_i)\le C<+\infty.
$$
Moreover, by the $L_\loc^1$-convergence and the isoperimetric inequality \eqref{aniso_isop_ineq}, for any 
bounded open set $U\subset\Rn$
$$
|U\cap E| = \lim\limits_{i\to+\infty} |U\cap E_i| \le \liminf\limits_{i\to+\infty} |E_i| \le 
(\constiso^{\frac{n}{n-1}})^{-1}
\,\liminf\limits_{i\to+\infty} 
\anisotropicperimeter(E_i)^{\frac{n}{n-1}} \le 
\tfrac{C^{\frac{n}{n-1}}}{\constiso^{\frac{n}{n-1}}}.
$$
Thus, letting $U\nearrow\Rn$ we get $|E|<+\infty$ and hence, $E\in \cS^*.$ 
By the $L_\loc^1$-lower semicontinuity of $\shorthandgeneralfunctional(\cdot;F,\tau,k),$ $E$ is a minimizer.
\end{proof}

\subsection{$L^\infty$-bound for $\d_F$ }\label{subsec:L_infty_bound_for_d_F}
Assume Hypothesis \eqref{hyp:main}.
For $F\in \cS^*$ and $\tau>0,$ let $F_\tau$ be a minimizer of 
$\shorthandgeneralfunctional(\cdot;F,\tau).$ 
In \eqref{L_cheksiz_baho}, 
on the basis of \eqref{nice_estimate_0o09}, we 
establish an upper bound for 
$$
\sup_{F_\tau \Delta F}\,\d_F,
$$
which is necessary to get the density estimates in Section \ref{subsec:density_estimates}
centered at points of $\partial F_\tau$.

Take any $x\in \cl{F_\tau}\setminus F$ and set
$$
r_x:=\d_F(x)>0, \qquad
B_r:=B_r(x).
$$
We observe that $B_r\cap F=\emptyset$ for all $r\in (0,r_x),$ so that $F\setminus F_\tau= F\setminus (F_\tau\setminus B_r).$ Then 
by the minimality of $F_\tau,$
\begin{multline}\label{come_si_llama}
0\le 
\shorthandgeneralfunctional
(F_\tau\setminus B_r; F,\tau) -
\shorthandgeneralfunctional
(F_\tau; F, \tau) \\
= \anisotropicperimeter(F_\tau\setminus B_r) - \anisotropicperimeter(F_\tau) -\int_{F_\tau\cap B_r} \veloc{F}{y}{\tau}dy - \int_{F_\tau\cap B_r} \forcing dy.
\end{multline}
Using the properties of the reduced boundary \cite[Theorem 16.3]{Maggi:2012}, \eqref{norm_bounds} 
and the Euclidean isoperimetric inequality, for a.e. $r\in (0,r_x)$ we get 
\begin{equation}\label{napoli_centrale}
\begin{aligned}
\anisotropicperimeter(F_\tau \setminus B_r) - \anisotropicperimeter(F_\tau ) = & \,\int_{F_\tau \cap\p B_r} \dualnorm(-\nu_{B_r})d\cH^{n-1} - 
\anisotropicperimeter(F_\tau ,B_r) 
\\
= &\, \int_{F_\tau \cap \p B_r} \Big(\dualnorm(-\nu_{B_r}) + \dualnorm(\nu_{B_r}) \Big) d\cH^{n-1} - \anisotropicperimeter(F_\tau \cap B_r) 
\\
\le &\, 2\Constdualnorm \cH^{n-1}(F_\tau \cap \p B_r) - \constdualnorm n\omega_n^{1/n}|F_\tau \cap B_r|^{\frac{n-1}{n}}.
\end{aligned}
\end{equation}
Since $\d_F \ge r_x-r$ in $B_r,$ by the increasing monotonicity of $\affine$ we obtain
\begin{equation*}
-\int_{F_\tau \cap B_r} \veloc{F}{y}{\tau}dy \le -\affine\Big(\frac{r_x-r}{\tau}\Big)|F_\tau \cap B_r|.
\end{equation*}
Finally, if we also assume $r<\constantforcingMorrey,$ 
with $\constantforcingMorrey
>0$ as in \eqref{eq:on_existence_of_gammag},
then 
\begin{equation}\label{romano_pizzo}
\int_{F_\tau \cap B_r} \forcing dy \le \frac{\constdualnorm n\omega_n^{1/n}}{4}|F_\tau \cap B_r|^{\frac{n-1}{n}}.
\end{equation}
Inserting \eqref{napoli_centrale}-\eqref{romano_pizzo}
in \eqref{come_si_llama} we get 
\begin{equation}\label{shze57}
2\Constdualnorm \cH^{n-1}(F_\tau \cap \p B_r) \ge \affine\Big(\frac{r_x-r}{\tau}\Big)|F_\tau \cap B_r| + \frac{3\constdualnorm n\omega_n^{1/n}}{4}|F_\tau \cap B_r|^{\frac{n-1}{n}}
\end{equation}
for a.e. $r\in (0, r_x\wedge \constantforcingMorrey
).$ 
Since $r_x>r$ and $f>0$ in $\R^+,$ this inequality implies 
$$
\cH^{n-1}(F_\tau \cap \p B_r) \ge \frac{3\constdualnorm n\omega_n^{1/n}}{8\Constdualnorm }|F_\tau \cap B_r|^{\frac{n-1}{n}},
$$
or, after integration,
\begin{equation}\label{eq:lower_volume_density_estimate}
|F_\tau \cap B_r| \ge \Big(\tfrac{3\constdualnorm}{8\Constdualnorm}\Big)^n 
\omega_nr^n,\quad r\in (0, r_x\wedge \constantforcingMorrey].
\end{equation}
Plugging this volume density estimate in \eqref{shze57} and using the positivity of the last term, we get 
$$
\affine\Big(\frac{r_x-r}{\tau}\Big)\,\Big(\tfrac{3\constdualnorm}{8\Constdualnorm}\Big)^n \omega_nr^n \le 2\Constdualnorm \cH^{n-1}(F_\tau \cap \p B_r) \le 2
\Constdualnorm n\omega_n r^{n-1},
$$
and hence, by the continuity of $\affine,$ 
\begin{equation}\label{nice_estimate_0o09}
r\affine\Big(\frac{r_x-r}{\tau}\Big) \le C_1
:= 2\Constdualnorm n \Big(\tfrac{8
\Constdualnorm}{3\constdualnorm}\Big)^n,\quad r\in [0,r_x\wedge \constantforcingMorrey]. 
\end{equation}

Now assume $x\in \cl{F}\setminus F_\tau $ 
and set again $r_x:=\d_F(x)>0.$ In this case for a.e. $r\in (0,r_x\wedge \constantforcingMorrey),$ we use the properties of the reduced boundary, \eqref{norm_bounds} 
and the Euclidean isoperimetric inequality to get
\begin{align*}
\anisotropicperimeter(F_\tau \cup B_r) - 
\anisotropicperimeter(F_\tau ) = 
&\, 2\int_{F_\tau ^c\cap \p^*B_r} \Big(\dualnorm(\nu_{B_r}) + \dualnorm(-\nu_{B_r}) \Big) d\cH^{n-1} - P_{\norm^\#}(F_\tau \cap B_r) \\
\le &\, 2\Constdualnorm \cH^{n-1}(F_\tau ^c\cap \p^*B_r) - \constdualnorm n\omega_n^{1/n}|F_\tau ^c\cap B_r|^{\frac{n-1}{n}}.
\end{align*}
Moreover, since $\d_F \ge r_x-r$ in $B_r$  (recall that $B_r\subset F$ by the choice of $r_x$), by the monotonicity of $\affine$  
$$
-\int_{F_\tau ^c\cap B_r} \veloc{F}{y}{\tau}dy \le -\affine\Big(\frac{r_x-r}{\tau}\Big)|F_\tau ^c\cap B_r|.
$$
Finally, using  $r<\constantforcingMorrey$ and the Morrey-type extimate \eqref{eq:on_existence_of_gammag} we find 
$$
\int_{F_\tau ^c\cap B_r} \forcing dy \le \frac{\constdualnorm n\omega_n^{1/n}}{4}|F_\tau ^c\cap B_r|^{\frac{n-1}{n}}.
$$
Inserting these estimates in 
\begin{multline*}
0\le \shorthandgeneralfunctional
(F_\tau \cup B_r; F,\tau)-
\shorthandgeneralfunctional
(F_\tau ;\affine,\tau) \\
= \anisotropicperimeter(F_\tau \cup B_r) - 
\anisotropicperimeter(F) 
- \int_{F_\tau ^c\cap B_r}\veloc{F}{y}{\tau}~dy + \int_{F_\tau ^c\cap B_r} \forcing ~dy,  
\end{multline*}
we get 
$$
2\Constdualnorm \cH^{n-1}(F_\tau ^c\cap \p B_r) \ge \affine\Big(\frac{r_x-r}{\tau}\Big)|F_\tau ^c\cap B_r| + \frac{3\constdualnorm n\omega_n^{1/n}}{4}|F_\tau ^c\cap B_r|^{\frac{n-1}{n}},
$$
and hence, repeating the same aguments above, 
we find the same inequality as \eqref{nice_estimate_0o09}. Thus, 
for any $x\in \cl{F_\tau \Delta F}$ we have, inverting \eqref{nice_estimate_0o09},
\begin{equation}\label{upper_bound_rx}
r_x \le r + \tau \affine^{-1}\Big(\tfrac{C_1}{r}\Big),\qquad 
r_x:=\d_F(x)>0, 
\quad r\in (0, r_x\wedge \constantforcingMorrey).
\end{equation}
For any $\tau>0$ let $\rho_\tau:=\rho_\tau(C_1)>0$ 
be the unique number (compare (\ref{hyp:main}b)) such that 
\begin{equation}\label{eq:rho_tau}
\rho_\tau \affine\Big(\frac{\rho_\tau}{\tau}\Big) = C_1.
\end{equation}
Clearly, $\tau\mapsto \rho_\tau$ is continuous in $\R^+$ and $\rho_\tau\to0^+$ as $\tau\to0^+.$ 
In particular
\begin{equation}\label{eq:definition_of_tau0}
\exists \tau_0 = \tau_0(\affine,\forcing)
 > 0 : \rho_\tau\le \constantforcingMorrey\quad \text{for all $\tau\in (0,\tau_0).$}
\end{equation}
Let us estimate $r_x$ with a multiple of $\rho_\tau.$
For $\tau\in(0,\tau_0),$ if $r_x\le \rho_\tau,$ we are done. If $\rho_\tau<r_x,$ by the choice of $\tau_0,$ we can apply \eqref{upper_bound_rx} with $r=\rho_\tau$ and the equality \eqref{eq:rho_tau} to get  
$$
r_x \le \rho_\tau + \tau \affine^{-1}\Big(\tfrac{C_1}{\rho_\tau}\Big) = 2\rho_\tau.
$$
Thus we finally get the desired $L^\infty$-bound
\begin{equation}\label{L_cheksiz_baho}
\sup_{x\in\cl{F_\tau \Delta F}}\,\d_F(x) \le 2\rho_\tau,\quad \tau\in (0,\tau_0). 
\end{equation}
This estimate will be used also in Theorem \ref{teo:comparison_with_balls}.

\subsection{Density estimates} \label{subsec:density_estimates}
Assume Hypothesis \eqref{hyp:main} and for $F\in \cS^*,$ $\tau_0>0$ 
(defined in \eqref{eq:definition_of_tau0}) and $\tau\in(0,\tau_0)$  let $F_\tau $ 
be a minimizer of $\shorthandgeneralfunctional(\cdot;F,\tau).$ 
Let $\constantforcingMorrey
>0$ be as in \eqref{eq:on_existence_of_gammag}. We 
establish uniform density estimates for minimizers $F_\tau $ in balls $B_r(x),$ 
centered at 
$$
x\in\p F_\tau
$$
 (see \eqref{eq:lower_volume_density_estimate_bis}, 
\eqref{eq:upper_volume_density_estimate}
and \eqref{eq:lower_perimeter_density_estimate}).  
Basically, we use the arguments of Section
\ref{subsec:L_infty_bound_for_d_F} and 
the function $\tau \to \rho_\tau$ in 
\eqref{eq:rho_tau}, but some extra care is
needed because now 
$$
B_r=B_r(x)
$$
 may intersect the boundary of $F$ and we need to
estimate the differences 
\begin{multline*} 
\int_{(F_\tau \setminus B_r)\Delta F}
\veloc{F}{y}{\tau}dy 
-
\int_{F_\tau \Delta F}\veloc{F}{y}{\tau}dy 
\\
= 
-\int_{B_r\cap (F_\tau \setminus F)} \veloc{F}{y}{\tau}dy  
+ \int_{B_r\cap F_\tau \cap F} \veloc{F}{y}{\tau}dy
\le \int_{B_r\cap F_\tau \cap F} \veloc{F}{y}{\tau}dy 
\end{multline*}
and 
\begin{multline*} 
\int_{(F_\tau \cup B_r)\Delta F}\veloc{F}{y}{\tau}dy 
-
\int_{F_\tau \Delta F}\veloc{F}{y}{\tau}dy \\
= 
-\int_{B_r\cap (F_\tau ^c\setminus F^c)} \veloc{F}{y}{\tau}dy  
+ \int_{B_r\cap F_\tau ^c\cap F^c} \veloc{F}{y}{\tau}dy 
\le  \int_{B_r\cap F_\tau ^c\cap F^c} \veloc{F}{y}{\tau}dy.
\end{multline*}
Note that if $x\in \p F_\tau $ and $y\in B_r\cap F_\tau \cap F$ 
(resp. $y\in B_r\cap F_\tau ^c\cap F^c$), by the $1$-lipschitzianity of $\d_F$ and \eqref{L_cheksiz_baho}  
$$
\d_F(y) \le \d_F(x) + |y-x| \le 2\rho_\tau + r,
$$
and hence, by the monotonicity of $\affine$
$$
\int_{B_r\cap F_\tau \cap F} \veloc{F}{y}{\tau}dy \le \affine\Big(\tfrac{r+2\rho_\tau}{\tau}\Big) |B_r\cap F_\tau |
$$
and similarly  
$$
\int_{B_r\cap F_\tau ^c\cap F^c} \veloc{F}{y}{\tau}dy \le \affine\Big(\tfrac{r+2\rho_\tau}{\tau}\Big) |B_r\cap F_\tau ^c|.
$$
For any $\tau\in (0,\tau_0)$ let $r_\tau>0$ be the 
unique number (compare (\ref{hyp:main}b)) satisfying 
\begin{equation}\label{def_r_tau}
r_\tau \affine\Big(\tfrac{r_\tau + 2\rho_\tau}{\tau}\Big) = \frac{\constdualnorm n}{2}.
\end{equation}
By the continuity of $\affine,$ $\tau\mapsto r_\tau$ is continuous and $r_\tau\to0$ as $\tau\to0^+.$  

Possibly decreasing $\tau_0$ in \eqref{eq:definition_of_tau0}, 
we assume 
\begin{equation}\label{eq:definition_of_tau0_bis}
r_\tau<\constantforcingMorrey\qquad
 {\rm for~ all}~ \tau\in (0,\tau_0).
\end{equation}
 Now, by the increasing monotonicity of $r\mapsto r\affine((r+2\rho_\tau)/\tau)$ and the definition of $r_\tau$, for any $r\in (0,r_\tau]$ we have
$$
\affine\Big(\tfrac{r+2\rho_\tau}{\tau}\Big) |B_r\cap F_\tau | =
\omega_n^{1/n} r \affine\Big(\tfrac{r+2\rho_\tau}{\tau}\Big)|B_r\cap F_\tau |^{\frac{n-1}{n}} \le \frac{\constdualnorm n\omega_n^{1/n}}{2}|B_r\cap F_\tau |^{\frac{n-1}{n}}
$$
and 
$$
\affine\Big(\tfrac{r+2\rho_\tau}{\tau}\Big) |B_r\cap F_\tau ^c| = 
\omega_n^{1/n}r\affine\Big(\tfrac{r+2\rho_\tau}{\tau}\Big)|B_r\cap F_\tau ^c|^{\frac{n-1}{n}} \le \frac{\constdualnorm n\omega_n^{1/n}}{2}|B_r\cap F_\tau ^c|^{\frac{n-1}{n}}.
$$
Recalling the inequality $r_\tau<\constantforcingMorrey$ and the estimates  \eqref{napoli_centrale} and \eqref{romano_pizzo},  for a.e. $r\in(0,r_\tau]$  we have
\begin{multline*}
0\le 
\shorthandgeneralfunctional
(F_\tau \setminus B_r;F,\tau) -  
\shorthandgeneralfunctional
(F_\tau ;F,\tau) 
\le 2\Constdualnorm\cH^{n-1}(F_\tau \cap B_r) \\
- \constdualnorm n\omega_n^{1/n}|F_\tau \cap B_r|^{\frac{n-1}{n}} 
+ \frac{\constdualnorm n\omega_n^{1/n}}{2}|F_\tau \cap B_r|^{\frac{n-1}{n}} 
+ \frac{\constdualnorm n\omega_n^{1/n}}{2}|F_\tau \cap B_r|^{\frac{n-1}{n}}.
\end{multline*}
This implies, similarly to \eqref{eq:lower_volume_density_estimate}, 
\begin{equation}\label{eq:lower_volume_density_estimate_bis}
|F_\tau \cap B_r|\ge \Big(\tfrac{3\constdualnorm}{8\Constdualnorm}\Big)^n\omega_n r^n,\quad r\in (0,r_\tau].
\end{equation}
Analogously,
\begin{equation}\label{eq:upper_volume_density_estimate}
|F_\tau ^c\cap B_r|\ge \Big(\tfrac{3\constdualnorm}{8\Constdualnorm}\Big)^n\omega_n r^n,\quad r\in (0,r_\tau].
\end{equation}
From \eqref{eq:lower_volume_density_estimate_bis}, \eqref{eq:upper_volume_density_estimate},
and the 
relative isoperimetric inequality in balls, we obtain 
\begin{equation}\label{eq:lower_perimeter_density_estimate}
P(F_\tau ,B_r) \ge 
\lowerbddensity
r^{n-1},\quad \tau\in (0,\tau_0),
\quad r\in (0,r_\tau),
\end{equation}
with $\lowerbddensity := 2^{\frac{n-1}{n}}\omega_{n-1} 
\Big(\tfrac{\constdualnorm}{8\Constdualnorm}\Big)^{n-1}$. 

\subsection{Existence of GMM}
As already mentioned in the introduction, the proof runs along a well-established path due to Almgren-Taylor-Wang 
\cite{ATW:1993} and 
Luckhaus-Sturzenhecker \cite{LS:1995}; however, here  various nontrivial technical modifications are required mainly due to the presence of $\affine$ (and of $\forcing$).

Let us assume Hypothesis \eqref{hyp:main}. 
Fix $\initialset\in \cS^*,$ and $\tau\in (0,\tau_0)$ 
(recall that $\tau_0$, given in \eqref{eq:definition_of_tau0}
and \eqref{eq:definition_of_tau0_bis}), and using Lemma \ref{lem:existence_minima} define the flat flows $\{E(\tau,k)\}_{k\ge0}$ inductively as follows:
$E(\tau,0)= \initialset$ and 
$$
E(\tau,k) 
\in {\rm argmin}\,
\shorthandgeneralfunctional
(\cdot; E(\tau,k-1),\tau),\quad k\ge1.
$$
From the inequality 
$$
\shorthandgeneralfunctional
(E(\tau,k);E(\tau,k-1),\tau) \le 
\shorthandgeneralfunctional
(E(\tau,k-1);E(\tau,k-1),\tau),\quad k\ge1,
$$
we get the standard estimate
\begin{equation}\label{lucky_factao}
\int_{E(\tau,k)\Delta E(\tau,k-1)} 
\veloc{E(\tau,k-1)}{x}{\tau}~dx \le \prescrmcf_{k-1} - \prescrmcf_k,\quad k\ge1,
\end{equation}
where 
$$
\prescrmcf_k:= \anisotropicperimeter(E(\tau,k)) + \int_{E(\tau,k)} \forcing ~dx,\quad k\ge0, 
$$
and hence the sequence $(\prescrmcf_k)$ is nonincreasing. In particular,
$$
\anisotropicperimeter(E(\tau,k)) = \prescrmcf_k -  \int_{E(\tau,k)} \forcing ~dx \le \prescrmcf_0 + \int_{E(\tau,k)} |\forcing| ~dx 
$$
and therefore,
\begin{equation}\label{unif_per_bounds}
\anisotropicperimeter(E(\tau,k))
\le 
\prescrmcf_0+\int_{\Rn} |\forcing| ~dx :=C_2\quad \text{for all $k\ge0$}.
\end{equation}

Consider the family $\{E(\tau,\intpart{t/\tau})\}_{t\ge0}.$ 
In view of \eqref{unif_per_bounds}, compactness in $BV_\loc(\Rn;\{0,1\})$ and a diagonal argument, we can find 
a sequence $\tau_j\to0^+$ such that for every rational $t\ge0$ there exists a set $E(t)\in BV_\loc(\Rn;\{0,1\})$ such that 
\begin{equation}\label{binorio_tredici}
E(\tau_j,\intpart{t/\tau_j}) \to E(t)\quad\text{in $L_\loc^1(\Rn)$ as $j\to+\infty.$}
\end{equation}
In view of \eqref{unif_per_bounds} and the 
isoperimetric inequality 
\eqref{aniso_isop_ineq}, the measure of $E(\tau_j,\intpart{t/\tau_j})$ is bounded uniformly in $j$ and $t,$ and therefore
$|E(t)|$ is bounded for each rational $t\ge0$, 
in particular $E(t)\in \cS^*.$

Now we will prove that the convergence \eqref{binorio_tredici} holds for any $t\ge0$ (without passing to a further subsequence). To this aim, let us establish a sort of uniform continuity of flat flows $\{E(\tau,k)\}$.

In view of \eqref{eq:lower_perimeter_density_estimate} each $E(\tau,k)$ with $k\ge1$ satisfies a uniform lower perimeter density estimate for radii $r\in (0,r_\tau]$. Thus, we can apply Lemma \ref{lem:volume_distance} 
with %
\begin{equation}\label{eq:r_0=r_tau}
r_0=r_\tau
\end{equation}
and $\lowerbounddensity
=\lowerbddensity$, 
to estimate the measure $
\vert E(\tau,k-1)\Delta E(\tau,k)\vert.$ 
In view of the expression of the ``distance''-term in $\shorthandgeneralfunctional,$ we apply that lemma 
with 
\begin{equation}\label{eq:ell=tau}
\ell=\tau
\end{equation}
 and some $p>0$ to be chosen later (see \eqref{suitable_p}). 

\begin{remark}[\textbf{Necessity of assumption \eqref{lower_bound_r_tau}}]\label{rem:bad_upper_bound}
Assumption \eqref{lower_bound_r_tau} implies 
$$
\limsup_{\tau\to0^+} \frac{\tau}{r_\tau} \in [0,+\infty).
$$
If this limsup is infinite  (for instance, in the case $\affine(r)=e^r$ for large $r,$ see Example \ref{ex:badexample}), an application of Lemma \ref{lem:volume_distance} would give a large coefficient $\frac{\tau^{n-1}}{r_\tau^{n-1}}$ in the first inequality
in estimate \eqref{eq:volume_distance_inequality}, which 
seems hard to handle. To avoid such an issue, we assume the validity of  \eqref{lower_bound_r_tau}, which is used only here. 
\end{remark}

For any $p>0$ and small $\tau>0$ we have by 
\eqref{eq:r_0=r_tau}
\eqref{eq:ell=tau}, 
and \eqref{eq:volume_distance_inequality}, 
$$
|E(\tau,k)\Delta E(\tau,k-1)| \le C_3 p^\sigma \tau 
\anisotropicperimeter(E(\tau,k-1)) + \frac{1}{\affine(p)}\int_{E(\tau,k)\Delta E(\tau,k-1)} \veloc{E(\tau,k-1)}{x}{\tau}~dx
$$
for all $k\ge2,$ where $C_3>0$ and $\sigma\in\{1,n\}.$ By the uniform perimeter bound \eqref{unif_per_bounds} and inequality \eqref{lucky_factao} we can estimate further 
$$
|E(\tau,k)\Delta E(\tau,k-1)| \le C_4 p^\sigma \tau + \frac{\prescrmcf_{k-1} - \prescrmcf_k}{\affine(p)},
$$
where $C_4:=C_2C_3.$
Summing these inequalities in $k=m_1+1,\ldots,m_2$ for 
$1\le m_1<m_2$ we obtain 
\begin{equation}\label{dasfnir4}
|E(\tau,m_2)\Delta E(\tau,m_1)| \le C_4 p^\sigma \tau(m_2-m_1) + 
\frac{\prescrmcf_{m_1} - \prescrmcf_{m_2}}{\affine(p)}.
\end{equation}
Since
\begin{align*}
\prescrmcf_{m_1} - \prescrmcf_{m_2} \le & 
\prescrmcf_0 - \prescrmcf_{m_2} \le \prescrmcf_0 +
\int_{E(\tau,m_2)} \forcing^-(x)~dx \le C_2,
\end{align*}
\eqref{dasfnir4} estimates further as 
\begin{equation}\label{discrete_holder}
|E(\tau,m_1) \Delta E(\tau,m_2)| \le C_4 p^\sigma (m_2-m_1)\tau + \frac{C_2}{\affine(p)}.
\end{equation}

Now we come back to our chosen flat flows $\{E(\tau_j,\intpart{t/\tau_j})\}_{t\ge0}.$ 
Fix any $t_2>t_1>0$ and, for $\tau_0$ as in 
\eqref{eq:definition_of_tau0}, 
\eqref{eq:definition_of_tau0_bis}, 
let $\tau_j\in(0,\tau_0)$ be so small that $\min\{t_1,t_2-t_1\}>10\tau_j.$ 
Now we choose $p.$ Let $u:\R_0^+\to\R_0^+$ be any continuous increasing 
function\footnote{In the case of the 
standard Almgren-Taylor-Wang  functional 
one can choose $u(s)=s^{1/2}$ for any $s \geq 0$ and get the
$1/2$-H\"older continuity of the flow (best exponent).} with $u(0)=0$ and
\begin{equation*} 
\limsup\limits_{s\to0^+} \frac{s}{u(s)} =0.
\end{equation*}
We set 
\begin{equation}\label{suitable_p}
p: =
\begin{cases}
\frac{1}{u(t_2-t_1)^{1/n}} & 
\text{if $\liminf\limits_{j\to+\infty}\frac{r_{\tau_j}}{\tau_j} \in (0, +\infty),$}
\\[2mm]
\frac{1}{u(t_2-t_1)} & \text{if $\liminf\limits_{j\to+\infty}\frac{r_{\tau_j}}{\tau_j}=+\infty.$}
\end{cases}
\end{equation}
Applying \eqref{discrete_holder} with such a $p,$ $m_1=\intpart{t_1/\tau_j},$  $m_2=\intpart{t_2/\tau_j}$ using 
$
\tau_j(m_2-m_1) \le t_2-t_1+\tau_j
$
and recalling $p^\sigma=1/u(t_2-t_1),$ we get 
\begin{equation}\label{disc_holtime}
|E(\tau_j,\intpart{t_2/\tau_j}) \Delta E(\tau_j,\intpart{t_1/\tau_j})| \le 
\modulusofcontinuity(t_2-t_1) + 
\frac{C_4\tau_j}{u(t_2-t_1)},
\end{equation}
where, for $s>0$, 
\begin{equation}\label{def:unif_conto}
\modulusofcontinuity(s)=
\begin{cases}
\frac{C_4s}{u(s)} + 
\frac{C_2}{\affine(u(s)^{-1/n})} & 
\text{if $\liminf\limits_{j\to+\infty}\frac{r_{\tau_j}}{\tau_j} \in (0, +\infty),$}
\\[2mm]
\frac{C_4s}{u(s)} + 
\frac{C_2}{\affine(u(s)^{-1})} & \text{if $\liminf\limits_{j\to+\infty}\frac{r_{\tau_j}}{\tau_j}=+\infty.$}
\end{cases}
\end{equation}
Now if both $t_1$ and $t_2$ are rational, then letting $j\to+\infty$ in \eqref{disc_holtime} and recalling \eqref{binorio_tredici} we deduce 
\begin{equation}\label{uniform_cont_Et}
|E(t_2)\Delta E(t_1)| \le \modulusofcontinuity(t_2-t_1).
\end{equation}
By assumptions on $w$ and $\affine,$ $\modulusofcontinuity(s)\to0$ as $s\to0^+.$ Thus, $\modulusofcontinuity$ provides a 
modulus of continuity, which is  for the moment defined only for rational times. Since 

$$
\sup_{t\ge0,\,\text{rational}}\,\,\,
\anisotropicperimeter(E(t))\le C_2,
$$
by standard arguments, we can uniquely extend $E(\cdot)$ for all $t\ge0$ still satisfying the uniform continuity \eqref{uniform_cont_Et} for all $0<t_1<t_2.$

We claim that $E(\cdot)$ is a GMM, i.e., the convergence \eqref{binorio_tredici} holds for all $t\ge0$.
Indeed, consider an irrational $t,$ and take any rational $\bar t>t.$ 
Then for any open set $U\subset\Rn,$
\begin{multline*}
|U\cap (E(\tau_j,\intpart{t/\tau_j}) \Delta E(t))| 
\le 
|U\cap (E(\tau_j,\intpart{t/\tau_j}) \Delta U\cap (E(\tau_j,\intpart{\bar t/\tau_j}))| \\
+ 
|U\cap (E(\tau_j,\intpart{\bar t/\tau_j}) \Delta E(\bar t))| 
+ |E(\bar t)\Delta E(t)|=: a_j+b_j+c.
\end{multline*}
Since $\bar t$ is rational, by \eqref{binorio_tredici}
$$
\limsup_{j\to+\infty} b_j = 0.
$$
Moreover, by \eqref{disc_holtime} and \eqref{uniform_cont_Et}
$$
\limsup_{j\to+\infty} a_j \le \modulusofcontinuity(\bar t- t)\quad\text{and}\quad 
c\le \modulusofcontinuity(\bar t-t).
$$
Therefore,
$$
\limsup_{j\to+\infty}\,|U\cap (E(\tau_j,\intpart{t/\tau_j}) \Delta E(t))| \le 2 \modulusofcontinuity(\bar t-t).
$$
Now letting $\bar t\searrow t$ and recalling that 
$\modulusofcontinuity(\bar t-t)\to 0$ we deduce the validity of \eqref{binorio_tredici} for all times $t\ge0.$ 
Whence, by definition, $E(\cdot)$ is a GMM.

\begin{remark}\label{rem:arbitrary_sequence}
In the proof of the existence we started with an arbitrary sequence and constructed a subsequence for which the $L_\loc^1$-convergence \eqref{binorio_tredici} holds.
\end{remark}

Finally, let us prove that every $\{E(t)\}_{t\ge0}\in 
\GMM(\shorthandgeneralfunctional,\cS^*,\initialset)$ starting from $\initialset$ is uniformly continuous. Let $\{E(t)\}$ be defined as an $L_\loc^1$-limit of flat flows $\{E(\tau_j,\intpart{t/\tau_j})\}$ for some $\tau_j\to0.$ Then these flat flows necessarily satisfy  \eqref{disc_holtime}, and hence, after passing to the limit we see that $E(\cdot)$ also satisfies \eqref{uniform_cont_Et}.

\subsection{Uniform continuity of GMM up to $t=0$}
Let us assume that $|\p F| = 0$ and let $F_\tau$ be a 
minimizer of $\shorthandgeneralfunctional(\cdot;F,\tau).$ 
By minimality  
\begin{equation}\label{ancha_yomon_bulsin}
\anisotropicperimeter(F_\tau) + \int_{F_\tau\Delta F_0}\veloc{F}{x}{\tau} ~dx + \int_{F_\tau} \forcing ~dx \le 
\anisotropicperimeter(F) + \int_{F} \forcing ~dx,
\end{equation}
and hence, recalling $\forcing \in L^1(\Rn),$ the family $\{\anisotropicperimeter(F_\tau)\}$ is 
uniformly bounded with respect to $\tau$. 
By compactness, for any 
sequence $\tau_k\searrow0,$ $F_{\tau_k}\overset{L_\loc^1}{\to} F_0$ (up to a
not relabelled subsequence) where $F_0\in \cS^*.$ By \eqref{ancha_yomon_bulsin}, the assumption $|\p F| = 0,$ and Fatou's lemma, $|F_0\Delta F|=0,$ and hence, $F_0=F.$ Thus, 
\begin{equation}\label{hohosa}
F_\tau\to F
\end{equation}
in $L_\loc^1(\Rn)$ as $\tau\to0.$
Of course, using \eqref{ancha_yomon_bulsin} one can also show 
$$
\lim\limits_{\tau\to0} \anisotropicperimeter(F_\tau) = 
\anisotropicperimeter(F)\quad\text{and}\quad 
\lim\limits_{\tau\to0} \int_{F_\tau\Delta F_0}\veloc{F}{x}{\tau} dx  = 0.
$$
Now to prove the uniform continuity of GMM for $s,t\ge0$ we proceed as in the standard Almgren-Taylor-Wang case: applying \eqref{disc_holtime} with $t_2=t>0$ and $t_1=\tau_j,$ as well as \eqref{hohosa} for any open set $U\subset\Rn$ we find 
$$
|U\cap (E(\tau_j,\intpart{t/\tau_j}) \Delta \initialset)| \le 
|E(\tau_j,\intpart{t/\tau_j}) \Delta E(\tau_j,1)| + |U\cap (E(\tau_j,1)\Delta \initialset)| 
\le \modulusofcontinuity(t- \tau_j) + o(1)
$$
as $j\to+\infty.$ This implies
$|U\cap (E(t)\Delta \initialset)| \le \modulusofcontinuity(t)$,
and hence, letting $U\to\Rn$ we get $|E(t)\Delta \initialset|  \le \modulusofcontinuity(t).$

\subsection{H\"older continuity of GMM}
Suppose
\begin{equation}\label{holder_fC_5}
\lim_{r\to+\infty}\,\frac{\affine(r)}{r^\DeGiorgiexponent} = C_5
\in (0,+\infty)
\end{equation}
for some $\DeGiorgiexponent\in(0,+\infty).$ Then using \eqref{eq:rho_tau} and recalling $\rho_\tau/\tau\to+\infty$ as $\tau\to0^+$ (see \eqref{amasu87}) we get
$$
\limsup_{\tau\to0^+} \frac{C_1\tau^\DeGiorgiexponent
}{\rho_\tau^{1+\DeGiorgiexponent}}
= \limsup_{\tau\to0^+}\frac{\affine(\rho_\tau/\tau)}{(\rho_\tau/\tau)^\DeGiorgiexponent
} = C_5,
$$
and hence, 
$$
\liminf\limits_{\tau\to0^+}\frac{\rho_\tau}{\tau^{\frac{\DeGiorgiexponent}{1+\DeGiorgiexponent}}} = \frac{C_1^{\frac{1}{1+
\DeGiorgiexponent
}}}{C_5}.
$$
Similarly, as $\frac{r_\tau + 2\rho_\tau}{\tau}\to+\infty$ (see Example \ref{ex:fisaffine}), by \eqref{def_r_tau}
$$
C_5  = \limsup_{\tau\to0^+}\,\frac{\affine((r_\tau + 2\rho_\tau)/\tau)}{[(r_\tau + 2\rho_\tau)/\tau]^\DeGiorgiexponent
} = 
\limsup_{\tau\to0^+}\,\frac{\constnorm n}{2r_\tau [(r_\tau + 2\rho_\tau)/\tau]^\DeGiorgiexponent
}.
$$
Thus, 
$$
\Big(\tfrac{2C_5}{\constnorm n}\Big)^{1/\DeGiorgiexponent
} = \liminf_{\tau\to0} \Big(\frac{r_\tau^{\frac{\DeGiorgiexponent
+1}{\DeGiorgiexponent
}}}{\tau} + \frac{2r_\tau^{\frac{1}{\DeGiorgiexponent
}}\rho_\tau}{\tau}\Big),
$$
and therefore,
$$
C_6:=\liminf\limits_{\tau\to0^+}\frac{r_\tau}{\tau^{\frac{
\DeGiorgiexponent
}{1+\DeGiorgiexponent
}}}<+\infty.
$$
This implies 
$$
\liminf\limits_{\tau\to0^+}\frac{r_\tau}{\tau}=+\infty.
$$
Now choosing $w(s) = s^\frac{1}{1+\DeGiorgiexponent
}$, in \eqref{discrete_holder} we can 
represent the function $\modulusofcontinuity$ in \eqref{uniform_cont_Et} as 
\begin{equation}\label{saz7888}
\modulusofcontinuity
(s) = C_4s^{\frac{\DeGiorgiexponent
}{1+\DeGiorgiexponent
}} + \frac{C_2}{\affine(s^{-\frac{1}{1+\DeGiorgiexponent
}})}.
\end{equation}
In view of \eqref{holder_fC_5} 
and the strict monotonicity of $\affine,$ there exists $s_0>0$ such that 
$$
\affine(s^{-\frac{1}{1+\DeGiorgiexponent
}}) \ge \frac{C_5}{2}\,s^{-\frac{\DeGiorgiexponent
}{1+\DeGiorgiexponent
}}\quad\text{for all $s\in (0,s_0)$}.
$$
Thus, for such $s,$ from \eqref{saz7888} we conclude
$$
\modulusofcontinuity
(s) \le \Big(C_4 + \frac{2C_2}{C_5}\Big)s^{\frac{\DeGiorgiexponent
}{1+\DeGiorgiexponent
}}.
$$
In view of \eqref{uniform_cont_Et} this implies the local $\frac{\DeGiorgiexponent
}{1+\DeGiorgiexponent
}$-continuity of GMM.

\subsection{Boundedness of $\GMM$}
The next lemma shows that boundedness of an initial set implies  
boundedness of 
minimizers of $\shorthandgeneralfunctional$.

\begin{lemma}\label{lem:obsledovanie}
Suppose Hypothesis \eqref{hyp:main}. Then,
for any bounded $F\in \cS^*$ and any $\tau>0$, every minimizer of $\shorthandgeneralfunctional(\cdot;F,\tau)$ is bounded. 
\end{lemma}

\begin{proof}
Let $E$ be a minimizer of $\shorthandgeneralfunctional(\cdot;F,\tau)$ and let $\constantforcingMorrey>0$ be given by \eqref{eq:on_existence_of_gammag}. Writing $B_r:=B_r(0),$ let $r_0>0$ be such that $F\subset B_{r_0}$ and $|E\setminus B_r|<\omega_n
\constantforcingMorrey^n$ for all $r>r_0$ (the last inequality is possible since $|E|<+\infty$).  For any $r>r_0,$ by the minimality of $E$ and the inequality $F\setminus E = F\setminus [E\cap B_r]$
\begin{multline}\label{ofortuna}
0\le 
\shorthandgeneralfunctional
(E\cap B_r;F,\tau) - \shorthandgeneralfunctional
(E;F,\tau) \\
= \anisotropicperimeter(E\cap B_r) - \anisotropicperimeter(E)
- \int_{E\setminus B_r} \veloc{F}{x}{\tau}dx - \int_{E\setminus B_r} \forcing dx.
\end{multline}
By the standard properties of the reduced boundary, \eqref{norm_bounds} and the Euclidean isoperimetric inequality
\begin{multline*} 
\anisotropicperimeter(E\cap B_r) - 
\anisotropicperimeter(E) = \int_{E\cap \p B_r} \dualnorm(\nu_{B_r}) d\cH^{n-1} - 
\anisotropicperimeter(E,B_r^c)\\
= \int_{E\cap \p B_r} \Big[\dualnorm(\nu_{B_r})+\dualnorm(-\nu_{B_r})\Big] d\cH^{n-1}- 
\anisotropicperimeter(E\setminus B_r)
\le 2\Constnorm \cH^{n-1}(E\cap \p B_r) - \constnorm n\omega_n^{1/n} |E\setminus B_r|
\end{multline*}
for a.e. $r>r_0.$ Moreover, since $\d_F \ge r-r_0$ in $E\setminus B_r$ and $\affine$ is increasing,
$$
\veloc{F}{\cdot}{\tau} \ge \affine\Big(\tfrac{r-r_0}{\tau}\Big)\quad\text{a.e. in $E\setminus B_r$}
$$
and therefore, 
\begin{equation*} 
\int_{E\setminus B_r} \veloc{F}{x}{\tau}dx \ge \affine\Big(\tfrac{r-r_0}{\tau}\Big) |E\setminus B_r|.
\end{equation*}
By the choice of $r_0$ and \eqref{eq:on_existence_of_gammag} 
\begin{equation*} 
- \int_{E\setminus B_r} \forcing dx \le \tfrac{\constnorm n\omega_n^{1/n}}{4}|E\setminus B_r|^{\frac{n-1}{n}}.
\end{equation*}
Inserting the above estimates in \eqref{ofortuna} we find
$$
2 \Constnorm \cH^{n-1}(E\cap B_r) + \tfrac{\constnorm n\omega_n^{1/n}}{4}|E\setminus B_r|^{\frac{n-1}{n}} \ge
\constnorm n\omega_n^{1/n}|E\setminus B_r|^{\frac{n-1}{n}} 
+
\affine\Big(\tfrac{r-r_0}{\tau}\Big) |E\setminus B_r|.
$$
In particular, 
$$
\cH^{n-1}(E\cap B_r^c) \ge \tfrac{3\constnorm n\omega_n^{1/n}}{8\Constnorm} |E\cap B_r^c|^{\frac{n-1}{n}}\quad \text{for a.e. $r>r_0.$}
$$
If $|E\cap B_r^c|>0$ for all $r>0,$ then integrating this differential inequality in $(r_0,r)$ we get
$$
\tfrac{3\constnorm}{8\Constnorm}\,(r-r_0) \le |E\cap B_{r_0}^c|^{1/n} -   |E\cap B_{r}^c|^{1/n}.
$$
However, letting $r\to+\infty$ and using $|E|<+\infty,$ we obtain $|E\cap B_{r_0}^c|=+\infty,$ a contradiction. Thus, $|E\cap B_r^c|=0$ for some $r>r_0,$ i.e., $E\subset B_r.$
\end{proof}

From this lemma, we deduce that each flat flow starting from a bounded set is bounded. 
However, when passing to the limit as $\tau_j\to0,$ we may loose this boundedness, and therefore, we need this lemma in some stronger form and for this we need  stronger assumptions on $\forcing.$
In what follows we write, for $\varrho >0$ and $x \in \Rn$,  
\begin{equation*}
\Wulff_\varrho(x) 
= \{\xi \in \Rn : \norm(\xi-x) < \varrho\},
\qquad
\Wulff_\varrho=
\Wulff_\varrho(0). 
\end{equation*}

\begin{lemma}[\textbf{Growth of $\forcing^-$ controlled by $\affine$}]\label{lem:bounded_wuldf}
Assume Hypothesis \eqref{hyp:main} and that $\forcing$ satisfies  \eqref{forcing_linear_growth}. Then there exist constants $C_6,C_7>0$ depending only on $\constantgminus,\constnorm,\Constnorm$ and $\constantforcingMorrey$ 
(see \eqref{eq:on_existence_of_gammag})
such that the following holds.  For any $F\in \cS^*$ and $\tau>0$ let $F_\tau$ 
be a minimizer of $\shorthandgeneralfunctional(\cdot;F,\tau).$ 
Suppose that $F\subset 
\Wulff_{r_0}$ for some $r_0>C_6.$ Then $F_\tau\subset 
\Wulff_{r_\tau}$ for any $\tau < \frac{1}{2C_7},$ where $r_\tau = (1+C_7\tau)r_0+ C_7\tau.$
\end{lemma}

\begin{proof}
By \eqref{norm_bounds}, \eqref{forcing_linear_growth} 
and the strict monotonicity of $\affine$
\begin{equation}\label{new_condos_H}
\forcing^-(x) \le \affine\Big(
\constantgminus
 + \tfrac{\constantgminus}{\constnorm}\norm(x)\Big),\quad \norm(x)>\tfrac{\constantgminus}{\constnorm}.
\end{equation}
Let $F\subset \Wulff_{r_0}$ for some $r_0>C_6:=\constantgminus/\constnorm;$ by Lemma \ref{lem:obsledovanie}  a minimizer $F_\tau$ is bounded, and let $\Wulff_{r_\tau}$ be the smallest Wulff shape containing $F_\tau.$ Thus, $\p F_\tau\cap \p \Wulff_{r_\tau}
\ne \emptyset.$
We may assume $r_\tau>r_0.$ 

Fix small $\epsilon\in (0, r_\tau-r_0)$ and consider the difference
\begin{multline}\label{increadible0}
0\le \shorthandgeneralfunctional
(F_\tau\cap \Wulff_{r_\tau-\epsilon}; F,\tau) - 
\shorthandgeneralfunctional
(F_\tau;F,\tau) = 
\anisotropicperimeter(F_\tau \cap \Wulff_{r_\tau -\epsilon}) - 
\anisotropicperimeter(F_\tau)\\
- \int_{F_\tau \setminus \Wulff_{r_\tau}} \veloc{F}{x}{,\tau} dx -\int_{F_\tau\setminus \Wulff_{r_\tau-\epsilon}} \forcing\,dx.
\end{multline}
By \eqref{norm_bounds} $\d_{F}(x) \ge \frac{1}{\Constnorm}\d_{F}^\norm(x) \ge \frac{r_\tau -\epsilon-r_0}{\Constnorm}$ in $F_\tau \setminus \Wulff_{r_\tau - \epsilon},$ where 
$$
\d_{F}^\norm(x):=\inf\{\norm(x-y):\,\,y\in \p^* F\}
$$
is the anisotropic distance. Thus, by the monotonicity of $\affine,$
$$
-\int_{F_\tau \setminus \Wulff_{r_\tau}} \veloc{F}{x}{\tau} dx \le -\affine\Big(\frac{r_\tau -\epsilon-r_0}{\Constnorm\tau}\Big)\,|F_\tau \setminus 
\Wulff_{r_\tau -\epsilon}|.
$$
In view of \eqref{new_condos_H}, 
using $r_0<r_\tau - \epsilon <\norm(\cdot)\le r_\tau $ in $E\setminus 
\Wulff_{r_\tau -\epsilon}$ and assumption \eqref{new_condos_H} we have 
\begin{multline*}
- \int_{F_\tau \setminus \Wulff_{r_\tau - \epsilon}} \forcing  dx \le 
\int_{F_\tau \setminus \Wulff_{r_\tau-\epsilon}} \forcing^- dx \\
\le 
\int_{F_\tau \setminus \Wulff_{r_\tau-\epsilon}} \affine\Big(\constantgminus
 + \tfrac{\constantgminus
}{\constnorm}\norm(x)\Big)dx
\le \affine\Big(\constantgminus
 + \tfrac{\constantgminus
}{\constnorm}r_\tau \Big)|E\setminus \Wulff_{r_\tau - \epsilon}|.
\end{multline*}
Finally, by the convexity of $\Wulff_{r_\tau-\epsilon},$
$$
\anisotropicperimeter(F_\tau \cap \Wulff_{r_\tau -\epsilon}) - 
\anisotropicperimeter(F_\tau) \le 0.
$$
Inserting these estimates in \eqref{increadible0} we get 
$$
0\le -\affine\Big(\frac{r_\tau-\epsilon-r_0}{\Constnorm\tau}\Big) + \affine\Big(\constantgminus
 + \tfrac{\constantgminus
}{\constnorm}r_\tau \Big),
$$
thus, by the strict monotonicity of $\affine$ and the arbitrariness of $\epsilon$ we get 
$$
\frac{r_\tau - r_0}{\Constnorm\tau}\le \constantgminus
 + \frac{\constantgminus
}{\constnorm}r_\tau.
$$
Now assuming $\frac{\constantgminus
\Constnorm}{\constnorm}\tau<\frac{1}{2}$ we can write 
$$
r_\tau \le \frac{r_0 + \constantgminus
\Constnorm\tau}{1 - \frac{\constantgminus
\Constnorm}{\constnorm}\tau} \le \Big(1 + \frac{2\constantgminus
\Constnorm}{\constnorm}\tau\Big)r_0 + \constantgminus
\Constnorm\tau\Big(1 + \frac{2\constantgminus
\Constnorm}{\constnorm}\tau\Big)\le (1+C_7\tau)r_0 + C_7\tau,
$$
where  $C_7:=\max\{\frac{2\constantgminus
\Constnorm}{\constnorm}, 2\constantgminus
\Constnorm\}.$
\end{proof}

Now we prove the local uniform boundedness of 
GMM, as stated in (ii). Fix $\tau\in(0,\frac{1}{2C_7})$ and a bounded $\initialset
\in\cS^*,$ and let  $\{E(\tau,k)\}$ be a flat flow starting from $\initialset.$ We may assume $\initialset\subset \Wulff_{r_0}$ for some $r_0>C_6.$
Let $\{r(\tau,k)\}_k$ be a nondecreasing sequence such that $r(\tau,0)=r_0,$ $E(\tau,k)\subset \Wulff_{r(\tau,k)}$ and for any $k\ge1$ either $r(\tau,k)=r(\tau,k-1),$ or by Lemma \ref{lem:bounded_wuldf} 
$r(\tau,k)\le (1+C_7\tau)r(\tau,k-1) + C_7\tau.$
Thus, there is no loss of generality in assuming 
$$
r(\tau,k)\le (1+C_7\tau)r(\tau,k-1) + C_7\tau,\quad k\ge1.
$$
Then\footnote{
Let $A\ge1,$ $B\ge0 $ and $\{\alpha_k\}_{k\ge0}\subset\R_0^+$ be such that 
$\alpha_{k+1} \le A\alpha_k +B$ for any $k\ge0$. 
Then
$\alpha_m \le 
A^m\alpha_0 + B\,\frac{A^m-1}{A-1}$
for any $m\ge1$.
Indeed
\begin{align*}
\alpha_m \le & A\alpha_{m-1} + B\le 
A(A\alpha_{m-2} + B) +B = A^2\alpha_{m-2} + B(1+A)\\
\le & \ldots \le A^m\alpha_0 + B(1+A+\ldots+A^{m-1}) = 
A^m\alpha_0 + B\,\frac{A^m-1}{A-1}.
\end{align*}

}
$$
r(\tau,k) \le (1+C_7\tau)^kr_0 + C_7\tau\,\frac{(1+C_7\tau)^k-1}{C_7\tau} < (1+C_7\tau)^k(r_0+1).
$$
Thus, for any $t\ge0$
\begin{align*}
r(\tau,\intpart{t/\tau}) \le (1+C_7\tau)^{\intpart{t/\tau}}(r_0+1)
\le
(1+C_7\tau)^{t/\tau}(r_0+1) \le e^{C_7t}(r_0+1) = :R(t).
\end{align*} 
This implies $E(\tau,\intpart{t/\tau})\subset \Wulff_{R(t)}$ 
for all $t\ge0$ and therefore, for any $E(\cdot)\in \GMM(\shorthandgeneralfunctional,\cS^*,\initialset)$ we have $E(t)\subset \Wulff_{R(t)}.$

\section{Rescalings of $\GMM$ and comparison with balls}
\label{sec:rescalings_of_GMM_and_comparison_with_balls}
Let us study how minimizers 
are related in rescaling. 

\begin{lemma}[\textbf{Rescaling}]
Assume that $\affine$ satisfies (\ref{hyp:main}a) and (\ref{hyp:main}b), and $\forcing\equiv0.$ 
Suppose that $\initialset\in \cS^*$ contains the origin in its interior and 
$\scalefactor>0$. Then $\scalefactor E(\cdot) \in 
\GMM(\shorthandgeneralfunctional, \cS^*,\scalefactor
\initialset)$ 
if and only if $E(\cdot) \in \GMM(\sF_{\norm, \scalefactor
\affine, 0},\cS^*,\initialset).$ 
\end{lemma}

\begin{proof}
As usual, $\shorthandgeneralfunctional = \sF_{\norm,\affine,0}$.
Since
\begin{multline}\label{advokat_navalni}
\shorthandgeneralfunctional 
(\scalefactor G;\scalefactor
 \initialset,\tau) = 
\anisotropicperimeter(\scalefactor
 G) + \int_{\scalefactor
 G}\sveloc{\scalefactor
 \initialset}{x}{\tau}~dx \\
= 
\scalefactor^{n-1}\anisotropicperimeter(G) + \scalefactor^n\int_{G}\sveloc{\initialset}{x}{\tau}~dx = \scalefactor^{n-1}\sF_{\norm,\scalefactor \affine,0}
(G;\initialset,\tau),
\end{multline}
$\scalefactor G$ is a minimizer of $\shorthandgeneralfunctional 
(\cdot;\scalefactor \initialset,\tau)$ 
if and only if $G$ is a minimizer of 
$\sF_{\norm,\scalefactor \affine,0}
(\cdot;\initialset,\tau).$ 

Let $\{\scalefactor E(\tau,k)\}$ be flat flows starting from $\scalefactor
 \initialset,$  
associated to $\shorthandgeneralfunctional.$ 
Then by \eqref{advokat_navalni},
the family $\{E(\tau,k)\}$ is a flat flow starting 
from $\initialset,$ associated to $\sF_{\norm,\scalefactor
 \affine,0}.$
This implies the thesis.
\end{proof}

\begin{corollary}[\textbf{Power case}]\label{cor:power_law_case}
Assume $\forcing \equiv 0$, that 
$$
\affine(r) = r^\DeGiorgiexponent,\quad r\ge0,
$$
for some $\DeGiorgiexponent >0,$ and $\initialset\in \cS^*$ contains the origin in its interior. Then $E(\cdot)
\in \GMM(\shorthandgeneralfunctional,\cS^*,\initialset)$ 
if and only if $\scalefactor E(t\scalefactor^{-1/\DeGiorgiexponent })
\in \GMM(\shorthandgeneralfunctional,\cS^*, \scalefactor
 \initialset).$
\end{corollary}

\begin{proof}
Since $\scalefactor \affine(r) = \affine(\scalefactor^{1/\DeGiorgiexponent } r),$ the equality  \eqref{advokat_navalni} 
becomes
$$
\shorthandgeneralfunctional(\scalefactor
 G; \scalefactor
 \initialset,\tau) = \scalefactor^{n-1}
\anisotropicperimeter(G) + \scalefactor^{n-1} \int_G 
\sveloc{\initialset}{x}{\tau \scalefactor^{-1/\DeGiorgiexponent }} ~dx
= \shorthandgeneralfunctional 
(G; \initialset,\tau\scalefactor^{-1/\DeGiorgiexponent }).
$$
Thus $\scalefactor G$ is a minimizer of $\shorthandgeneralfunctional 
(\cdot;\scalefactor \initialset,\tau)$ if and only if $G$ is a minimizer of 
$\shorthandgeneralfunctional 
(\cdot;\initialset,\tau \scalefactor^{-1/\DeGiorgiexponent }).$ 

Now if $\tau_j\to0$ and flat flows $\scalefactor E(\tau_j,k),$ starting from $\scalefactor
 \initialset,$ 
associated to $\sF_{\norm,\affine,0}$ and with the time step $\tau,$ are 
such that 
\begin{equation}\label{jahon_matbuoti}
\lim\limits_{j\to+\infty} \scalefactor
 E(\tau_j,\intpart{t/\tau_j}) = \scalefactor
 E(t)\quad\text{for all $t\ge0$ in $L_\loc^1(\Rn).$}
\end{equation}
Then $E(\tau_j,k)$ are flat flows starting from $\initialset,$ 
associated to $\sF_{\norm,\affine,0},$ but with the time step equal to 
$\tau\scalefactor^{-1/\DeGiorgiexponent }.$ Thus, by \eqref{jahon_matbuoti}
\begin{equation*} 
\lim\limits_{j\to+\infty}   E(\tau_j,\intpart{t/(\scalefactor^{-1/\DeGiorgiexponent }\tau_j)}) = \lim\limits_{j\to+\infty} E(\tau_j,\intpart{t
\scalefactor^{1/\DeGiorgiexponent}/\tau_j})  =  E(t\scalefactor^{1/\DeGiorgiexponent })\quad\text{for all $t\ge0$ in $L_\loc^1(\Rn).$}
\end{equation*}

The converse  assertion is done in a similar manner.
See also \cite{BCChN:2005}.
\end{proof}

\begin{theorem}[\textbf{Comparison with balls}]\label{teo:comparison_with_balls}
Suppose Hypothesis \eqref{hyp:main} and $\forcing\in L^\infty(\Rn).$
Given $\initialset\in \cS^*$ and $\tau>0,$ let  $\{E(\tau,k)\}$ be flat flows starting from $\initialset.$  
Let $r_0>0$. 
Then there exist $\widehat \tau_1>0$ and $C_8>0$ 
depending only on 
$n,$ $\norm,$ $\affine$, $r_0$ and $\|\forcing\|_\infty$,
such that  
\begin{equation}\label{wulff_inside}
\Wulff_{r_0}(x_0)\subset \initialset
\quad\Longrightarrow \quad \Wulff_{r_0-C_8k\tau}(x_0) \subset E(\tau,k) 
\end{equation}
and 
\begin{equation}\label{wulff_outside}
\Wulff_{r_0}(x_0)\cap \initialset
 = \emptyset \quad\Longrightarrow \quad \Wulff_{r_0-C_8k\tau} (x_0) \cap E(\tau,k) =\emptyset
\end{equation}
for all $\tau\in (0,\widehat\tau_1)$ and $0\le k\tau \le \frac{r_0}{2C_8}.$
\end{theorem}

\begin{proof}
Let $\tau_0$ be given by \eqref{eq:definition_of_tau0},
where we recall that $\constantforcingMorrey$ is given in \eqref{eq:on_existence_of_gammag}. 
We prove only  \eqref{wulff_inside}, the relation \eqref{wulff_outside} 
being similar. For shortness, we assume $x_0=0$ and let 
$\Wulff_{r_0}\subset \initialset.$ 

{\it Step 1.} Since $\rho_\tau\to0$ as $\tau\to0^+,$ there exists $\tau_1\in (0,\tau_0)$ (depending only on $r_0$) such that such that $\rho_{\tau_1}<r_0/10.$
For $\tau\in(0,\tau_1)$ let $E_\tau$ 
be a minimizer  of $\shorthandgeneralfunctional 
(\cdot;\initialset,\tau).$ By the choice of $\tau_1$ and 
the $L^\infty$-bound \eqref{L_cheksiz_baho}, we have  
$\Wulff_{4r_0/5}\subset E_\tau.$ Let 
$$
\bar r:=\sup\{r>0:\,\, \Wulff_r\subset E_\tau\}\ge 4r_0/5.
$$
We want to 
estimate $\bar r$ from below (see \eqref{eq:estimate_of_rbar_from_below}), thus, there is no loss of generality in assuming $\bar r<r_0.$ Fix $\epsilon\in(0,r_0-\bar r)$ and consider the difference 
\begin{multline}\label{look_paren}
0\le \shorthandgeneralfunctional 
(E_\tau\cup \Wulff_{\bar r+\epsilon}; \initialset,\tau) - 
\shorthandgeneralfunctional 
(E_\tau; \initialset,\tau)
\\
= \anisotropicperimeter(E_\tau\cup \Wulff_{\bar r+\epsilon}) - 
\anisotropicperimeter(E_\tau) 
+ \int_{\Wulff_{\bar r+\epsilon}\setminus E_\tau} \sveloc{\initialset}{x}{\tau} ~dx + \int_{\Wulff_{\bar r+\epsilon}\setminus E_\tau} \forcing ~dx.
\end{multline}
Using $\Wulff_{\bar r+\epsilon} \subset \Wulff_{r_0}\subset \initialset$ we find 
$-\sd_{\initialset} = \d_{\initialset} \ge \frac{1}{\Constnorm}\d_{
\initialset}^\norm \ge \frac{r_0-\bar r - \epsilon}{\Constnorm}$ in 
$\Wulff_{\bar r+\epsilon}\setminus E_\tau,$ where $d^\norm_F$
stands for the $\norm$-distance from the set $F$. 
Therefore, by the strict monotonicity and oddness of $\affine,$
\begin{equation}\label{eq:estimate_of_affine}
-\int_{\Wulff_{\bar r+\epsilon}\setminus E_\tau} \sveloc{\initialset}{x}{\tau} ~dx \ge \affine\Big(\tfrac{r_0-\bar r-\epsilon}{\Constnorm\tau}\Big)|
\Wulff_{\bar r+\epsilon}\setminus E_\tau|.
\end{equation}
Moreover, by the boundedness of $\forcing,$
\begin{equation}\label{eq:estimate_of_forcing}
\int_{\Wulff_{\bar r+\epsilon}\setminus E_\tau} \forcing ~dx \le \|\forcing\|_\infty |\Wulff_{\bar r+\epsilon}\setminus E_\tau|.
\end{equation}
Finally, using the anisotropic isoperimetric inequality \eqref{aniso_isop_ineq} for a.e. $\epsilon>0$ we have
\begin{equation}\label{eq:estimate_of_perimeter}
\begin{aligned}
\anisotropicperimeter(E_\tau\cup \Wulff_{\bar r+\epsilon}) - 
\anisotropicperimeter(E_\tau) = & 
\anisotropicperimeter(\Wulff_{\bar r+\epsilon}) - 
\anisotropicperimeter(E_\tau\cap \Wulff_{\bar r+\epsilon}) \\
\le & \constiso
 \Big(|\Wulff_{\bar r+\epsilon}|^{\frac{n-1}{n}} - 
|\Wulff_{\bar r+\epsilon}\cap E_\tau|^{\frac{n-1}{n}} \Big) \\
= & \constiso
|\Wulff_{\bar r+\epsilon}|^{\frac{n-1}{n}}\Big(1 - \Big|1 - \frac{|\Wulff_{\bar r+\epsilon}\setminus E_\tau|}{|\Wulff_{\bar r+\epsilon}|}\Big|^{\frac{n-1}{n}}\Big) \\
\le & 
\frac{\constiso}{|\unitballnorm|^{1/n}(\bar r+\epsilon)} |\Wulff_{\bar 
r+\epsilon} \setminus E_\tau|,
\end{aligned}
\end{equation}
where in the last inequality we used 
$(1-x)^\alpha\ge 1-x$ for any $x,\alpha\in(0,1)$.
Inserting \eqref{eq:estimate_of_affine}, \eqref{eq:estimate_of_forcing}, 
\eqref{eq:estimate_of_perimeter} in \eqref{look_paren} and using 
the arbitrariness of $\epsilon$ we get 
$$
\affine\Big(\frac{r_0-\bar r}{\Constnorm\tau}\Big) \le \frac{\constiso}{
|\unitballnorm|^{1/n}\,\bar r} + \|g\|_\infty.
$$
Thus, recalling $\bar r\ge 4r_0/5$ we get
$\frac{r_0-\bar r}{
\Constnorm\tau
} \le 
\affine^{-1}\Big(\frac{5
\constiso
}{4|\unitballnorm|^{1/n}\,r_0} + \|g\|_\infty\Big)
$, or equivalently
\begin{equation}\label{eq:estimate_of_rbar_from_below}
\bar r \geq r_0 - \Constnorm\tau
\affine^{-1}\Big(\frac{5
\constiso
}{4|\unitballnorm|^{1/n}\,r_0} + \|g\|_\infty\Big).
\end{equation}
Notice that this inequality holds also in case $\bar r\ge r_0.$ 

\medskip

{\it Step 2.} Let $\{E(\tau,k)\}$ be a flat flow starting from $\initialset$ and let $\{F(\tau,k)_*\}$ be a flat flow starting from $F_0:=
\Wulff_{r_0}$ and consisting of the minimal minimizers. By comparison (Theorem \ref{teo:compare} (c)), $F(\tau,k)_* \subset E(\tau,k)$ for all $k\ge0.$ 

Consider the sequence $r_0=r(\tau,0)\ge r(\tau,1)\ge \ldots$ of radii defined as follows: we assume 
$\Wulff_{r(\tau,k)}\subset F(\tau,k)_*$ and if $r(\tau,k-1)>r(\tau,k),$ then $ \Wulff_{r(\tau,k)}$ is the largest Wulff shape contained in $F(\tau,k)_*.$
Let $k_0\ge1$ be such that $r(\tau,k_0)\ge r_0/2$ and let $\tau_1>\tau_2>\ldots>\tau_{k_0}>0$ be given by step 1 applied with $r_0:=r(\tau,k)$ for $k=1,\ldots,k_0.$ Thus, for any $\tau\in(0,\tau_k)$ one has $r(\tau,k)>4r(\tau,k-1)/5.$ 
From step 1
$$
r(\tau,k) \ge r(\tau,k-1) - 
\Constnorm
\tau \affine^{-1}\Big( \frac{5\constiso
}{4|\unitballnorm|^{1/n} r(\tau,k-1)} + \|\forcing\|_\infty\Big),\quad 1\le k\le k_0.
$$
Now by the choice of $k_0,$ 
$$
r(\tau,k) \ge r(\tau,k-1) - \Constnorm\tau \affine^{-1}\Big( \frac{5
\constiso
}{2|\unitballnorm|^{1/n} r_0} + \|\forcing\|_\infty\Big),\quad 1\le k\le k_0,
$$
and hence
$$
r(\tau,k) \ge r_0 - \Constnorm \affine^{-1}\Big( \frac{5
\constiso
}{2|\unitballnorm|^{1/n} r_0} + \|\forcing\|_\infty\Big)k\tau,\quad 0\le k\le k_0.
$$
Now if we assume 
$$
C_8:=C_8(n,\norm, \affine, r_0, \|\forcing\|_\infty):=
\Constnorm \affine^{-1}\Big( \frac{5\constiso
}{2|\unitballnorm|^{1/n} r_0} + \|\forcing\|_\infty\Big)
$$
and 
$
k\tau \le \frac{r_0}{2C_8},
$
we get $r(\tau,k)\ge r_0/2$ and $r(\tau,k)\ge r_0-C_8k\tau.$
\end{proof}

It turns out that $\GMM$s for generalized power mean curvature flow share several properties with  $\GMM$s in the anisotropic mean curvature flow, as we now see. 
Let us study $\GMM$ starting from a Euclidean ball centered at origin. 

\begin{theorem}[\textbf{Evolution of balls}]\label{teo:evolving_ball}
Assume that $\norm$ is Euclidean and $\forcing
\equiv 0$. 
Suppose also that 
$r\mapsto \affine^{-1}((n-1)/r)$ is such that the ordinary differential equation
\begin{equation}\label{ode_for_Radii}
\begin{cases}
r'(t) = - \affine^{-1}\Big(\frac{n-1}{r(t)}\Big)  & \text{if $r(t)>0,$}\\
r(0) = r_0
\end{cases}
\end{equation}
admits a unique $C^1$ solution (for instance, 
$r\mapsto \affine^{-1}((n-1)/r)$ is convex in $(0,+\infty)$).
Then $\GMM(\sF,\cS^*,B_{r_0})$ is a singleton $\{B_{r(t)}\}_{t\ge0},$ where $r(\cdot)$ is a nonincreasing function satisfying \eqref{ode_for_Radii}. 
\end{theorem}

\begin{proof}
We can write
$$
\shorthandgeneralfunctional(E;\initialset,\tau):=P(E) + \int_E \sveloc{\initialset}{x}{\tau}~dx +
c_{\initialset}.
$$

{\it Step 1: properties of minimizers.} Fix $r_0>0$.  Let $\tau_0$ and $\rho_\tau$ 
be as in \eqref{eq:definition_of_tau0}, \eqref{L_cheksiz_baho} so that 
$$
\sup_{\cl{E_\tau\Delta B_{r_0}}}\,\d_{B_{r_0}}\,\le 2\rho_\tau,\quad \tau\in(0,\tau_0),
$$
for any minimizer $E_\tau$ of $\shorthandgeneralfunctional(\cdot;B_{r_0},\tau).$ 
Since $\rho_\tau\to 0^+$ as $\tau\to0^+,$
there exists $\tau_{r_0}\in(0,\tau_0)$ such that $B_{r_0/2}\subset E_\tau\subset B_{3r_0/2}$ for all $\tau\in (0,\tau_{r_0})$. 

Let us study the minimal and maximal minimizers. Owing to 
$$
\sveloc{B_{r_0}}{x}{\tau} = \affine\Big(\tfrac{|x|-r_0}{\tau}\Big),\quad x\in\Rn,
$$
the volume term of $\sF$ is radially symmetric. Therefore, the minimal and maximal minimizers of $\shorthandgeneralfunctional
(\cdot;B_{r_0},\tau)$ are radially symmetri
c, i.e., both of them are balls. By the choice of $\tau_{r_0},$ their radii are in the interval $(r_0/2,3r_0/2).$ 
In particular, the radii of minimal and maximal minimizers satisfy \eqref{masul_idoralar}. By the assumptions on $\affine,$ $r\mapsto \affine^{-1}\Big(\tfrac{n-1}{r} \Big)$ strictly decreases, convex and 
positive in $(0,+\infty).$ Therefore, using the linearity of $r\mapsto \frac{r - r_0}{\tau},$ we find that for any $\tau\in (0,\tau_{r_0})$ 
there exists a unique minimum point $r_\tau$ of $\affine$ in the interval $(r_0/2,3r_0/2).$ Clearly, $r_\tau\in (r_0/2,r_0).$ In particular, the minimal and maximal minimizers coincide, i.e., $\shorthandgeneralfunctional(\cdot; B_{r_0},\tau)$ has a unique minimizer $B_{r_\tau}.$

Without loss of generality, we may assume $r_0\mapsto \tau_{r_0}$ is increasing.
\smallskip

{\it Step 2: some properties of flat flows.} Given $r_0>0$ and $\epsilon\in(0,1/4),$ let $\bar\tau_\epsilon:=\tau_{\epsilon r_0/2}>0$ be given by step 1. For any $\tau\in (0,\bar \tau_{\epsilon})$ let $r(\tau,k),$ $k\ge0,$ be defined inductively as follows: $r(\tau,0)=r_0,$ and for $k\ge1,$ if $r(\tau,k)>\epsilon r_0,$ the ball $B_{r(\tau,k)}$ is the unique minimizer of $\shorthandgeneralfunctional(\cdot; B_{r(\tau,k-1)},\tau).$ In view of \eqref{masul_idoralar} these radii  satisfy the equation 
\begin{equation}\label{analysit}
\frac{r(\tau,k) - r(\tau,k-1)}{\tau} = -\affine^{-1}\Big(\frac{n-1}{r(\tau,k)}\Big).
\end{equation}
Thus, the sequence $k\mapsto r(\tau,k)$ strictly decreases. By \eqref{analysit} there exists a unique $\bar k_{\tau,\epsilon}>0$ such that $r(\tau,\bar k_{\tau,\epsilon}+1)\le \epsilon r_0<r(\tau,\bar k_{\tau,\epsilon}).$ 
One can readily check that both maps $\tau\mapsto k_{\tau,\epsilon}$ and $\epsilon\mapsto k_{\tau,\epsilon}$ are decreasing. Let 
$$
T_\epsilon:=\liminf\limits_{\tau\to0^+}  \tau k_{\tau,\epsilon}.
$$
By the monotonicity of $k_{\tau,\epsilon},$ $\epsilon\mapsto T_\epsilon$ is nondecreasing. Let us show that it is uniformly away from $0.$ Indeed, since $B_{r(\tau,k)},$ $0\le k\le k_{\tau,\epsilon}$ are flat flows, applying \eqref{discrete_holder} with $m_1=1$ and $m_2=k_{\tau,\epsilon}$ we get 
$$
|B_{r(\tau,1)}\setminus B_{r(\tau,k_{\epsilon,\tau}+1)}| \le C_4p^\sigma (\tau k_{\tau,\epsilon} + \tau - \tau) + \frac{1}{\affine(p)}
$$
for some $p>0$ and suitable $\sigma\in \{1,n\}.$ Now letting $\tau\to0^+$  and recalling $B_{r(\tau,1)}\to B_{r_0},$ we get 
$$
|B_{r_0}\setminus B_{\epsilon r_0}| \le C_4p^\sigma T_\epsilon \frac{1}{\affine(p)}.
$$
Now taking $p$ large enough (depending only on $r_0$) we deduce 
$
T_{\epsilon}> C_{11}
$
for some $C_{11}$ depending only on $C_4$ and $r_0.$ Let us denote
$$
T_0:=\sup_{\epsilon\in(0,1/4)} T_\epsilon.
$$
By the definition, $T_\epsilon \le T_0.$
\smallskip

{\it Step 3: properties of $\GMM$.} Let $\tau_j\to0^+$ be any sequence such that 
$$
B_{r(\tau_j,\intpart{t/\tau_j})} \to B_{r(t)}\quad \text{in $L^1(\Rn)$ as $j\to+\infty$} \quad 
\forall t\ge0,
$$
where we assume $r(\tau,k):=0$ if it does not satisfy \eqref{masul_idoralar}. Equivalently,
\begin{equation}\label{radii_convergence}
\lim\limits_{j\to+\infty} r(\tau_j,\intpart{t/\tau_j}) = r(t),\quad t\ge0.
\end{equation}

Fix any $\bar T<T_0$ and let $\bar \epsilon>0$ be such that $T_{\bar \epsilon}\in (\bar T,T_0).$ Then for any $\tau_j<t<\bar T$ so that $1\le k:=\intpart{t/\tau_j} < \intpart{T_{\bar\epsilon}/\tau_j}$ and $r(\tau_j,\intpart{t/\tau_j})\ge \bar\epsilon r_0>0,$ we can apply \eqref{analysit}: 
\begin{equation*}
\frac{r(\tau_j,\intpart{t/\tau_j}) - r(\tau_j,\intpart{(t-\tau_j)/\tau_j}  ) }{\tau_j} = -\affine^{-1}\Big(\frac{n-1}{r(\tau_j,\intpart{t/\tau_j})}\Big).
\end{equation*}
Passing to the limit in this difference equation and using \eqref{radii_convergence} we obtain 
\begin{equation}\label{islamasa}
r'(t) = - \affine^{-1}\Big(\frac{n-1}{r(t)}\Big),\quad t\in (0,\bar T),
\end{equation}
which  admits a unique solution. By the definition of $T_\epsilon$ we can show that $\lim_{\bar T\nearrow T_0} r(\bar T)=0.$ Since $\bar T<T_0$ is arbitrary, the radii of balls in each $\GMM$ necessarily satisfy \eqref{islamasa}. Now by uniqueness, $\GMM(\sF,\cS^*,B_{r_0})$ is a singleton $\{B_{r(t)}\}_{t\ge0},$ where we extend $r(t):=0$ for $t\ge T_0.$
\end{proof}

\begin{remark}
Let $\affine(r)=r^\DeGiorgiexponent$ in $(0,+\infty)$ for $\DeGiorgiexponent>0.$ Then the  
ODE in \eqref{ode_for_Radii} reads as 
$$
r' = -\frac{(n-1)^{1/\DeGiorgiexponent}}{r^{1/\DeGiorgiexponent}}.
$$
Thus, $\GMM(\sF,\cS^*,B_{r_0}) = \{B_{r(t)}\},$ where 
$$
r(t) = 
\begin{cases}
\Big(r_0^{1+\frac1\DeGiorgiexponent
} - \big(1+\tfrac{1}{\DeGiorgiexponent}\big)(n-1)^{\frac{1}{\DeGiorgiexponent
}}t\Big)^{\frac{\DeGiorgiexponent
}{1+\DeGiorgiexponent
}} & \text{if $t\in \Big[0,\tfrac{r_0^{1+1/\DeGiorgiexponent
}}{\big(1+1/\DeGiorgiexponent
\big)(n-1)^{1/\DeGiorgiexponent
} }\Big),$}\\[3mm]
0 & \text{otherwise.}
\end{cases}
$$
\end{remark}

\section{Evolution of mean convex sets: proof
of Theorem \ref{teo:meanconvex_intro}}\label{sec:evolution_of_mean_convex_sets}
In this section we mostly follow the ideas of \cite{ChN:2022}. 
Apart from some technical points due to the presence of the 
function $\affine,$ one difference here is in the proof of the $\delta$-convexity preservation of minimizers  (Proposition \ref{prop:mean_convexity_of_minimal_and_maximal_minimizers}): we apply directly the prescribed mean curvature functional in place of $1$-harmonic functions and Anzellotti-type arguments as done in \cite{ChN:2022}.
We assume here that $\forcing\equiv 0,$ $\tau>0$ and $k\ge0.$ 

Given an anisotropy $\norm$ in $\Rn,$ $\delta\ge0$ and a (nonempty)  open set $\Omega\subset\Rn,$ a set $E\Subset\Omega$ is called \emph{$\delta$-mean 
convex} (in $\Omega$) if 
$$
\anisotropicperimeter(E)-\delta|E| \le 
\anisotropicperimeter(F)-\delta|F|\quad \text{for any $F\Subset\Omega$ with $E\subset F.$}
$$
When $\delta=0,$ we simply say $E$ is mean convex. Repeating the density estimate arguments (Section \ref{subsec:density_estimates})  
we can readily show the existence of $r_0,\lowerbounddensitymc>0$ 
depending only on $n,$ $\constnorm$, $\Constnorm$ and 
$\affine$  such that every $\delta$-mean convex set $E$ satisfies 
$$
\frac{|B_r(x)\setminus E|}{|B_r(x)|} \ge \lowerbounddensitymc
>0
$$
for all $r\in (0,r_0)$ and for all $x\in\p E.$ In particular, $E=E^{(1)}$ is open.
Let us recall some more properties of $\delta$-mean convex sets.

\begin{proposition}[{\cite{ChN:2022,DPL:2020}}]\label{prop:mean_convex_propert}
Let $\delta\ge0.$

\begin{itemize}
\item[\rm(a)] Suppose there exists a $\delta$-mean convex set $E\Subset\Omega$ in $\Omega.$ Then $\emptyset$ is $\delta$-mean convex in $\Om$.

\item[\rm(b)] $E\Subset \Omega$ is $\delta$-mean convex in $\Omega$ if and only if 
$$
\anisotropicperimeter(E)  \le 
\anisotropicperimeter(E\cup F) - \delta|F\setminus E|,\quad F\Subset\Omega.
$$

\item[\rm(c)] $E\Subset \Omega$ is $\delta$-mean convex in $\Omega$ if and only if 
$$
\anisotropicperimeter(E\cap F)  \le 
\anisotropicperimeter(F) - \delta|F\setminus E|,\quad F\Subset\Omega.
$$

\item[\rm(d)] Let $E_\naturalindex
\Subset \Omega,$ $\naturalindex=1,2,\ldots,$ be $\delta$-mean convex 
in $\Omega$ with $E_\naturalindex\to E$ in $L^1(\Omega)$ for some $E\Subset\Omega.$ Then $E$ is $\delta$-mean convex.
\end{itemize}

\end{proposition}

Since we are mostly interested in bounded sets, in view of \eqref{ofortoo} we can 
write, for a constant $\constantprescribed$
independent of $E$,	 
$$
\shorthandgeneralfunctional
(E;F,\tau) := \anisotropicperimeter(E)+\int_E\sveloc{F}{x}{\tau}~dx +\constantprescribed $$
if $F$ is bounded, see \eqref{eq:prescur_funco}.

\begin{lemma}[\textbf{Inclusion of minimizers}]\label{lem:inclusion_of_minimizers}
Let $F\Subset\Omega$ be an open $\delta$-mean convex set in $\Omega$ for some $\delta\ge0.$ 
Suppose that $\affine$
satisfies (\ref{hyp:main}a), (\ref{hyp:main}b), 
and that  $\forcing \equiv 0$. 
Let $E$ be a minimizer of $\shorthandgeneralfunctional
(\cdot;F,\tau).$ Then:
\begin{itemize}
\item[\rm(i)] $\cl{E} \subset \cl{F};$ if $\delta>0,$ then $E\subset F,$ 

\item[\rm(ii)] if $\delta>0,$ 
\begin{equation}\label{massage_massage}
\bigcup_{|\eta|<\affine^{-1}(\delta)\tau} (E+\eta) \subset \cl{F}\quad\text{and}\quad \sd_E \ge \sd_F + \affine^{-1}(\delta)\tau\,\,\,\text{in $\Rn$}
\end{equation}
provided $\dist(F+\delta\tau,\p \Omega)>0,$

\item[\rm(iii)] $E$ is mean convex.
\end{itemize}
\end{lemma}

\begin{proof}
As we have seen in the proof of Theorem \ref{teo:existence_gmm}, under the assumptions of the lemma every minimizer $E$ of $\shorthandgeneralfunctional
(\cdot;F,\tau)$ satisfies $E\subset \Omega.$

(i) By the $\delta$-mean convexity of $F$ in $\Omega$  
\begin{equation}\label{F_meanconvex}
\anisotropicperimeter(F)\le \anisotropicperimeter(F\cup E) -\delta |E\setminus F|,
\end{equation}
and by the minimality of $E$
$$
\anisotropicperimeter(E) +\int_E\sveloc{F}{x}{\tau}~dx 
\le 
\anisotropicperimeter(E\cap F) +\int_{E\cap F}\sveloc{F}{x}{\tau}~dx.
$$
Summing these inequalities we get 
\begin{align}\label{dfnbg}
\int_{E\setminus F} \sveloc{F}{x}{\tau}~dx  + \delta|E\setminus F| 
\le \anisotropicperimeter(E\cup F)+
\anisotropicperimeter(E\cap F)-
\anisotropicperimeter(E)-\anisotropicperimeter(F)\le0,
\end{align}
where in the last inequality we used the submodularity of $\anisotropicperimeter.$ Since $\sd_F>0$ in $\Rn\setminus \cl{F}$ and $f>0$ in $\R^+,$ from \eqref{dfnbg}  we deduce $|E\setminus \cl{F}|=0,$ i.e., $E\subset \cl{F},$  hence, $\cl{E}\subset\cl{F}.$ Note that if $\delta>0,$ then $E\subset F.$

(ii) Assume $\delta>0$ and $\eta\in B_{\affine^{-1}(\delta)\tau}.$ 
Clearly, $F+\eta$ is $\delta$-mean convex in $\Omega+\eta$ and $E+\eta$ is a minimizer of $\shorthandgeneralfunctional(\cdot;F+\eta,\tau).$ By the choice of $\tau,$ $E+\eta\Subset\Omega$ and hence we can apply \eqref{F_meanconvex} with $E:=E+\eta.$ Moreover, 
$$
\anisotropicperimeter(E+\eta) +\int_{E+\eta}\sveloc{F+\eta}{x}{\tau}~dx 
\le 
\anisotropicperimeter([E+\eta]\cap F) +\int_{[E+\eta]\cap F}\sveloc{F+\eta}{x}{\tau}~dx.
$$
Summing this and \eqref{F_meanconvex} (applied with $E=E+\eta$) we deduce
$$
\int_{[E+\eta]\setminus F}\Big(\sveloc{F+\eta}{x}{\tau} + \delta\Big)~dx \le 
\anisotropicperimeter([E+\eta]\cup F)+ 
\anisotropicperimeter([E+\eta]\cap F) - \anisotropicperimeter(F) - 
\anisotropicperimeter(E+\eta)\le0.
$$
Thus, 
\begin{equation}\label{launch_missile}
\int_{[E+\eta]\setminus F}\Big(\sveloc{F+\eta}{x}{\tau} + \delta\Big)~dx = 
\int_{[E+\eta]\setminus F}\Big(\sveloc{F}{x-\eta}{\tau} + \delta\Big)~dx \le 0.
\end{equation}
Note that if $x\in [E+\tau]\setminus F,$ then by (a) $x-\eta \in E\subset F,$ and hence
$$
\sd_F(x - \eta)  = -\d_F(x-\eta) \ge - |x - (x-\eta)|=  -|\eta|>-\affine^{-1}(\delta)\tau.
$$
Then by the strict monotonicity of $\affine,$
$$
\sveloc{F}{x-\eta}{\tau} + \delta>0\quad\text{on $[E+\tau]\setminus F$}
$$
and therefore, by \eqref{launch_missile} $E+\eta\subset\cl{F}.$

Let us prove $\sd_E\ge \sd_F + \affine^{-1}(\delta)\tau.$ We know that $E\subset F$ and
\begin{equation}\label{chuvak_millard}
\inf_{x\in \p F,\,y\in \p E}\,\,\, |x-y| \ge \affine^{-1}(\delta)\tau.
\end{equation}
Take $x\in F^c$ and let $y\in \p E$ be such that $\d_E(x) = |x-y|.$ Then there exists $z\in [x,y]\cap\p F,$ and by the definition of $\d_F$ and \eqref{chuvak_millard} 
$$
\sd_E(x)=\d_E(x) = |x-y| = |x-z| + |z-y| \ge \d_F(x) +\affine^{-1}(\delta)\tau
=
\sd_F(x) +\affine^{-1}(\delta)\tau.
$$
Now assume that $x\in F\setminus E$ and let $y\in \p F$ and $z\in \p E$ such that $\d_F(x)=|x-y|$  and $\d_E(x)=|x-z|.$ Then  by \eqref{chuvak_millard}
$$
\d_E(x)+\d_F(x) = |x-y|+|x-z|\ge|y-z| \ge \affine^{-1}(\delta)\tau,
$$
and therefore,
$$
\sd_E(x) = \d_E(x)\ge -\d_F(x) + \affine^{-1}(\delta)\tau = \sd_F(x)+\affine^{-1}(\delta)\tau.
$$
Finally, assume that $x\in E$ and let $y\in \p F$ be such that $\d_F(x)=|y-x|.$ Then there exists $z\in [x,y]\cap\p E,$ and hence, as above 
$$
-\sd_F(x)=\d_F(x) = |x-z| + |z-y| \ge \d_E(x) +\affine^{-1}(\delta)\tau
= -\sd_E(x) +\affine^{-1}(\delta)\tau.
$$

(iii) We claim that $\anisotropicperimeter(E)\le \anisotropicperimeter(G)$ for any $G\Subset\Omega$ with $E\subset G.$ Indeed, by the $\delta$-mean convexity of $F$ and Proposition \ref{prop:mean_convex_propert} (c)
\begin{equation}\label{sjf7628udjcn}
\anisotropicperimeter(F\cap G) \le \anisotropicperimeter(G) - \delta|G\setminus F|.
\end{equation}
Moreover, by the minimality of $E,$
$$
\anisotropicperimeter(E) + \int_{E}\sveloc{F}{x}{\tau}~dx \le \anisotropicperimeter(F\cap G) + \int_{F\cap G}\sveloc{F}{x}{\tau}~dx. 
$$
Since $E\subset F\cap G$ and $\affine$ is odd, this inequality 
becomes
$$
\anisotropicperimeter(E) + \int_{[F\cap G]\setminus E} \affine\Big(-\tfrac{\sd_F }{\tau}\Big)~dx \le \anisotropicperimeter(F\cap G).
$$
The integral in this inequality is nonnegative. Therefore, by \eqref{sjf7628udjcn}
$$
\anisotropicperimeter(E) \le \anisotropicperimeter(F\cap G) \le \anisotropicperimeter(G) - \delta|G\setminus F|\le \anisotropicperimeter(G).
$$
\end{proof}

The following proposition improves the last assertion of Lemma 
\ref{lem:inclusion_of_minimizers}.

\begin{proposition}[\textbf{Mean convexity
of minimal and maximal minimizers}]\label{prop:mean_convexity_of_minimal_and_maximal_minimizers}
Suppose that $\affine$
satisfies (\ref{hyp:main}a), (\ref{hyp:main}b), 
and that  $\forcing \equiv 0$. 
Let $\initialset\Subset\Omega$ be an open $\delta$-mean convex set in $\Omega$ for some $\delta>0.$ Then the minimal and maximal minimizers of $\shorthandgeneralfunctional
(\cdot;\initialset,\tau)$ (in the sense of Corollary \ref{cor:minmaxprinciple}) are $\delta$-mean convex in $\Omega$.  
\end{proposition}

\begin{proof}
For any $s>0$ let $E_s$ be a minimizer of $\shorthandgeneralfunctional
(\cdot;\initialset,s).$
\smallskip

{\it Step 1:} For any $0<
\smallers
< 
\largers$
$$
E_{\largers}\subset E_{\smallers}\subset \initialset.
$$
The inclusion $E_{\largers},E_{\smallers}\subset \cl{\initialset}$ follows 
from Lemma \ref{lem:inclusion_of_minimizers} (i). 
To prove the first inclusion it is enough to 
observe that $s\mapsto \frac{\sd_{\initialset}}{s}$ 
is strictly decreasing in $\initialset.$ 
Now the inclusion follows from the comparison principle 
in Corollary \ref{cor:strong_comparison} for the prescribed mean curvature functional.
\smallskip

{\it Step 2:} 
$$
\lim\limits_{s\searrow 0}|\initialset\setminus E_s|=0.
$$
Note that this assertion was already shown in \eqref{hohosa} under the extra assumption $|\p \initialset|=0$. Here we do not 
have such a regularity. By  minimality 
$$
\anisotropicperimeter(E_{s}) + \int_{E_{s}}\sveloc{\initialset}{x}{s}~dx \le 
\anisotropicperimeter(\initialset) + \int_{\initialset}\sveloc{\initialset}{x}{s}~dx,
$$
and using $E_{s}\subset \initialset$, 
$$
\anisotropicperimeter(E_{s}) + \int_{\initialset\setminus E_{s}}\veloc{\initialset}{x}{s}~dx \le 
\anisotropicperimeter(\initialset).
$$
This, the monotonicity of $s\mapsto E_{s}$ and the openness of $E_{s}$ and $E_0$ imply $E_{s}\overset{L^1}{\to} E_0$ and $\anisotropicperimeter(E_{s})\to 
\anisotropicperimeter(E_0)$ as $s\searrow0.$
\smallskip

{\it Step 3.} 
$$
\lim\limits_{s\nearrow \tau}\,|E_s\setminus E_\tau^*|=0 
\quad \text{and} \quad 
\lim\limits_{s\searrow \tau}\,|E_{\tau *}\setminus E_s|=0. 
$$
We start by proving the first equality. Consider any sequence $s_i\nearrow \tau.$ Since $P_\phi(E_{s_i})\le P_\phi(E_0)$ and $E_{s_i}\subset E_0$ for any $i\ge1,$ there exists $Q\subset E_0$ such that, up to a further not relabelled subsequence, $E_{s_i}\searrow Q$ in $L^1(\R^n).$ By the $L^1$-lower semicontinuity of $\shorthandgeneralfunctional(\cdot;E_0,\tau),$ $Q$ is its minimizer. Moreover, since $E_{s_i}\supset E_{s_{i+1}},$ we have also $E_{s_i}\supset Q$ for any $i\ge1.$
By step 1, $E_{s_i}\supset E_{\tau}^*,$ and hence, we cannot have 
$|E_{\tau}^*\setminus Q|>0.$ Now arbitrariness of $s_i$ implies $E_s\searrow E_{\tau}^*$ as $s\nearrow \tau.$

The proof of the second equality is similar.
\smallskip

{\it Step 4.} Now we prove the $\delta$-mean convexity of the minimal and maximal minimizers.

 Fix any $G\in \cS^*$ with $G\Subset\Omega,$ $\epsilon\in(0,\delta)$ and $0<\underline{s}<\overline{s}.$ 
For any $N>1$ let 
$$
s_i:= \underline{s} + \frac{(\overline{s} - \underline{s})i}{N}, \quad  i=0,\ldots,N.
$$
Possibly slightly perturbing $G$ we may assume that 
$$
\sum_{i=0}^N \cH^{n-1}(\p^*G \cap \p^*E_{s_i}) = 0,\quad G_i:=G\cap [E_{s_{i-1}}\setminus E_{s_i}],\quad i=1,\ldots,N.
$$
Since $\initialset$ is bounded and $\affine$ is continuous, there exists $N>1$ such that
\begin{equation}\label{jivotnom}
\affine\Big(\tfrac{\d_{E_0}(x)}{s_i}\Big) \ge \affine\Big(\tfrac{\d_{E_0}(x)}{s_{i-1}}\Big) - \epsilon,\quad x\in E_0,\quad i=1,\ldots,N.
\end{equation}
Thus, by the minimality of $E_{s_i}$ and \eqref{jivotnom}
\begin{multline*}
\anisotropicperimeter(E_{s_{i-1}}\cap G) - \anisotropicperimeter(E_{s_i}\cap G) =  
\anisotropicperimeter(E_{s_i}\cup G_i) - \anisotropicperimeter(E_{s_i}) \\
\ge \int_{G_i} \veloc{E_0}{x}{s_i}~dx \ge 
\int_{G_i} \veloc{E_0}{x}{s_{i-1}}~dx - \epsilon|G_i|. 
\end{multline*}
By  \eqref{massage_massage} 
$$
\affine\Big(\tfrac{\d_{E_0}(x)}{s_{i-1}}\Big) = \affine\Big(\tfrac{-\sd_{E_0}(x)}{s_{i-1}}\Big) \ge \affine\Big(\tfrac{-\sd_{E_{s_{i-1}}}(x) + \affine^{-1}(\delta)s_{i-1}}{s_{i-1}}\Big) \ge \delta
\quad \text{for  $x\in E_{s_{i-1}},$}
$$
and therefore,
$$
\anisotropicperimeter(E_{s_{i-1}}\cap G) - 
\anisotropicperimeter(E_{s_i}\cap G) \ge (\delta-\epsilon)|G_i|,\quad i=1,\ldots,N.
$$
Summing these inequalities we get 
\begin{equation}\label{telefon_nummer}
\anisotropicperimeter(E_{\underline{s}}\cap G) - \anisotropicperimeter(E_{\overline{s}}\cap G) \ge (\delta-\epsilon)|G\cap [E_{\underline{s}} \setminus E_{\overline{s}}]|.
\end{equation}
Moreover, since $E_{\underline{s}}\subset \initialset$ is mean convex and $\initialset$ is $\delta$-mean convex, applying Proposition \ref{prop:mean_convex_propert} (c) twice (first with $E_{\underline{s}}$ and $\delta=0$ and then with $\initialset$ and $\delta$) we obtain 
$$
\anisotropicperimeter(E_{\underline{s}} \cap G) = \anisotropicperimeter(E_{\underline{s}} \cap [\initialset\cap G]) \le \anisotropicperimeter(\initialset\cap G) \le 
\anisotropicperimeter(G) - \delta |G\setminus \initialset|.
$$
Inserting this in \eqref{telefon_nummer} 
$$
\anisotropicperimeter(G) - \delta|G\setminus \initialset|  \ge 
\anisotropicperimeter(E_{\overline{s}}\cap G) + (\delta-\epsilon)|G\cap [E_{\underline{s}} \setminus E_{\overline{s}}]|,
$$
and hence, letting $\epsilon,\underline{s} \to 0^+$ and using step 2 we get 
$$
\anisotropicperimeter(G) \ge \anisotropicperimeter(E_{\overline{s}}\cap G) 
+ \delta|G \setminus E_{\overline{s}}|.
$$
Finally, letting $\overline{s}\searrow\tau$ and 
$\overline{s}\nearrow\tau,$ using step 3 and the $L^1$-lower semicontinuity of $\anisotropicperimeter,$ we get 
$$
\anisotropicperimeter(G) \ge \anisotropicperimeter(E_{\tau*}\cap G) + 
\delta|G \setminus E_{\tau*}|
\quad\text{and}\quad
\anisotropicperimeter(G) \ge \anisotropicperimeter(E_\tau^* \cap G) + 
\delta|G \setminus E_\tau^*|.
$$
Thus, both $E_\tau^*$ and $E_{\tau*}$ are $\delta$-mean convex by 
Proposition \ref{prop:mean_convex_propert}.
\end{proof}

\begin{corollary}[\textbf{Mean convexity of minimizers}]
\label{cor:mean_convexity_of_mininizers}
Let $\initialset\Subset\Omega$ be $\delta$-mean convex and  $\tau>0.$ Then every minimizer $E_\tau$ of $\shorthandgeneralfunctional
(\cdot;\initialset,\tau)$ is $\delta$-mean convex. 
\end{corollary}

\begin{proof}
If $\delta=0,$ the assertion follows from Lemma \ref{prop:mean_convexity_of_minimal_and_maximal_minimizers} (c), so we assume $\delta>0$ and consider the minimal and maximal  minimizers $E_{\tau*}\subset E_\tau\subset E_\tau^*.$ Then,
 for any $G\Subset\Omega$, by the $\delta$-mean convexity of $E_\tau^*$
(Proposition \ref{prop:mean_convexity_of_minimal_and_maximal_minimizers}),
\begin{equation}\label{xoroshoego_dnya}
\anisotropicperimeter(G) \ge 
\anisotropicperimeter(E_\tau^*\cap G) + \delta|G\setminus E_\tau^*|.
\end{equation}
Moreover, possibly slightly perturbing $G$ we assume that $\cH^{n-1}(\p^*G\cap\p^*E_\tau)= 0$ so that by the minimality of $E_\tau\subset E_\tau^*$
\begin{align*}
\anisotropicperimeter(E_\tau^*\cap G) - 
\anisotropicperimeter(E_\tau\cap G) = & 
\anisotropicperimeter(E_\tau\cup [G\cap(E_\tau^*\setminus E_\tau)]) - 
\anisotropicperimeter(E_\tau) 
\ge  \int_{G\cap(E_\tau^*\setminus E_\tau)} \veloc{\initialset}{x}{\tau}~dx.
\end{align*}
Moreover, by  \eqref{massage_massage}
$$
\veloc{\initialset}{x}{\tau} = 
\affine\Big( \tfrac{-\sd_{\initialset}(x) }{\tau}\Big)
\ge \affine\Big( \tfrac{-\sd_{E_\tau^*}(x) + \affine^{-1}(\delta)\tau}{\tau}\Big) \ge \delta
\quad \text{for $x\in E_\tau^*,$}
$$
and hence, 
$$
\anisotropicperimeter(E_\tau^*\cap G) \ge 
\anisotropicperimeter(E_\tau\cap G) +\delta |G\cap(E_\tau^*\setminus E_\tau)|.
$$
Adding this to \eqref{xoroshoego_dnya} we get 
$$
\anisotropicperimeter(G) \ge 
\anisotropicperimeter(E_\tau\cap G) + \delta |G\setminus E_\tau|,
$$
i.e., $E_\tau$ is $\delta$-mean convex.
\end{proof}

\begin{proof}[Proof of Theorem \ref{teo:meanconvex_intro}]
Let $\{E(\tau_j,k)\}$ be a family of flat flows starting from $\initialset$ and satisfying 
\begin{equation}\label{luchshe_zakonchit}
\lim\limits_{j\to+\infty}\,|E(\tau_j,\intpart{t/\tau_j}) \Delta E(t)| = 
0\quad\text{for any $t\ge0$}
\end{equation}
for some $E(\cdot)\in \GMM(\shorthandgeneralfunctional, \initialset).$ By Lemma \ref{lem:inclusion_of_minimizers} (a) 
$$
\cl{\initialset}\supset \cl{E(\tau_j,1)} \supset \cl{E(\tau_j,2)} \supset \ldots
$$
and by Corollary \ref{cor:mean_convexity_of_mininizers} each $E(\tau_j,k)$ is $\delta$-mean convex. Hence, $t\mapsto E(\tau_j,\intpart{t/\tau_j})$ is a nonincreasing map of $\delta$-mean convex sets. Then by \eqref{luchshe_zakonchit} and Proposition \ref{prop:mean_convex_propert} (d) each $E(t)$ is $\delta$-mean convex and the map $t\mapsto E(t)$ is nonincreasing. Therefore, by the definition of mean convexity, so is $t\mapsto \anisotropicperimeter(E(t)).$
\end{proof}

\section{Consistency with smooth flows:
proof of Theorem \ref{teo:consistency}}
\label{sec:consistency_with_smooth_flows}
If, for an anisotropy $\norm$, the map $\xi\mapsto \norm(\xi) - \lambda|\xi|$ is also an anisotropy in $\Rn$ for some $\lambda>0,$ we say $\norm$
is elliptic. 

Suppose $\norm$ is a $C^{3+\Holderexponent}$-elliptic anisotropy, and functions $\affine$ and $\forcing$ satisfy Hypothesis \eqref{hyp:main},
$\affine \in C^\Holderexponent(\R)$ and $\forcing\in C^{\Holderexponent}
(\Rn) \cap L^\infty(\Rn)$, for some $\Holderexponent\in(0,1]$. 

\begin{definition}[\textbf{Stable smooth flow}]
\label{def:stable_smooth_flow}
$\,$
\begin{itemize}
\item[\rm (a)] A $C^1$-in time family $\{S(t)\}_{t\in [0,T^\dag)}$ of $C^2$-subsets 
of $\Rn$ is called a generalized power 
smooth mean curvature flow with 
driving force $\forcing$ starting from $S_0$, if 
\begin{equation*}
\begin{cases}
\affine(\normalvelocity_{S(t)}(x)) = - \kappa_{S(t)}^\norm
(x) - \forcing(x) & \text{for $t\in (0,T^\dag)$ and $x\in \p S(t),$}\\[2mm]
S(0)=S_0,
\end{cases} 
\end{equation*}
where as usual $\normalvelocity_{S(t)}$ and $\kappa_{S(t)}^\norm$
are the normal velocity and the anisotropic mean curvature
of $\partial S(t)$, respectively.

\item[\rm (b)] The family $\{S(t)\}_{t\in [0,T^\dag)}$ 
is called  \emph{stable}
if for any $T\in (0,T^\dag)$ 
there are $\rho
=\rho(T)>0$, $\sigma = \sigma(T)>0$ such that 
for any $a\in [0,T)$ 
there exist families $L^\pm[r,s,a,t]$ for $r\in [0,\rho],$ $s\in [0,\sigma]$ and $t\in [a,T]$ of $C^{2+\Holderexponent}$-subsets of $\Rn$ 
smoothly depending\footnote{For instance, $(r,s,a,t)\mapsto \sd_{L^\pm[r,s,a,t]}$ smoothly varies, see also \cite[Corollary 7.2]{ATW:1993}.} on $r,s,a,t$,  
such that

\begin{itemize}
\item $L^\pm[0,0,a,t] = S(t)$ for all $t\in [a,T],$ 

\item $L^\pm[r,s,a,a] = \{x \in \Rn:\,\,\sd_{S(a)}(x)<\pm(r+s)\}$ for all $r\in [0,\rho]$ and $s\in[0,\sigma],$ 

\item for any $r\in[0,\rho]$ and $s\in[0,\sigma],$  
\begin{equation}\label{stable_general_mcf}
\affine(\normalvelocity_{L^\pm[r,s,a,t]}(x)) = - \kappa_{L^\pm[r,s,a,t]}^\norm(x) - \forcing(x) \pm s \quad \text{for $t\in [a,T]$ and $x\in \p L^\pm[r,s,a,t].$}
\end{equation}

\end{itemize}

\end{itemize}

\end{definition}

Using the signed distance functions 
we can rewrite \eqref{stable_general_mcf} as 
\begin{equation*}
\affine\Big(
\tfrac{\p}{\p t} 
\sd_{L^\pm[r,s,a,t]}(x)
\Big) = - 
\kappa_{L^\pm[r,s,a,t]}^\norm(x) 
- \forcing(x) \pm s \quad \text{for $t\in [a,T]$ and $x\in \p L^\pm[r,s,a,t],$}
\end{equation*}
where $\kappa_{L^\pm[r,s,a,t]}^\norm$ stands for the $\norm$-mean 
curvature of $L^\pm[r,s,a,t]$.

\begin{proposition}[\textbf{Properties of stable flows}]
\label{prop:properties_of_stable_flows}
Let $\{S(t)\}_{t\in [0,T^\dag)}$ be a stable
flow as above starting from a bounded set $S_0$ and for $T\in (0,T^\dag)$ 
let $\rho,\sigma,$ $L^\pm[r,s,a,t]$ be as in Definition \ref{def:stable_smooth_flow} (b).

\begin{itemize}
\item[\rm(a)] Assume that $r', r'' \in [0, \rho]$, 
$r'\le r''$ and $s', s'' \in [0, \sigma]$, 
$s'\le s''$ with $r'+s'<r''+s''.$  
Then 
$$
L^-[r'',s'',a,t]\Subset L^-[r',s',a,t]
\quad\text{and}\quad 
L^+[r',s',a,t]\Subset L^+[r'',s'',a,t]
$$
for any $t\in [a,T].$ 

\item[\rm(b)] For any $s\in (0,\sigma)$ there exists $\tau_2\in (0,T/2)$ such that for any $\tau\in (0,\tau_2),$ $r\in [0,\rho]$ 
and\footnote{If $a+\tau >T$ the statement becomes trivial.
} $t\in [a+\tau,T]$ 
\begin{equation*}
\affine\Big(\tfrac{  \sd_{L^+[r,s,a,t-\tau]}}{\tau}\Big) > - \kappa_{L^+[r,s,a,t]}^\norm- \forcing + \frac s2 \quad \text{on $\p L^+[r,s,a,t]$}
\end{equation*}
and 
\begin{equation*}
\affine\Big(\tfrac{  \sd_{L^-[r,s,a,t-\tau]}}{\tau}\Big) < - \kappa_{L^-[r,s,a,t]}^\norm -  \forcing - \frac s2 \quad \text{on $\p L^-[r,s,a,t].$}
\end{equation*}

\item[\rm(c)] There exists $t^* = t^*(T, \rho, \sigma, S(\cdot)) \in (0,\rho/64)$ such that 
\begin{equation*}
L^-[\rho,s,a,a+t'] \subset  L^-[\rho/2+t',s,a,a] 
\quad \text{and}\quad 
L^+[\rho/2-t',s,a,a]\subset L^+[\rho,s,a,a+t']
\end{equation*}
for all $s\in [0,\sigma], $ $a\in [0,T)$ and $t'\in [0,t^*]$ with $a+t'\le T.$

\item[\rm(d)] There exists a continuous increasing function $\atwfunction:\R_0^+\to\R_0^+$ with $\atwfunction
(0)=0$ such that for all $s\in [0,\sigma],$ $a\in[0,T),$ $t\in[a,T]$
$$
\sup_{x\in \p L^\pm[0,s,a,t]}\,\dist(x,\p L^\pm[0,0,a,t]) \le \atwfunction(s).
$$
\end{itemize}
\end{proposition}

\begin{proof}
By smoothness, the family $\{S(t)\}$ is uniformly bounded. Therefore, 
assertion (a) follows from 
the strong comparison principle (see e.g. \cite[Chapter 2]{Mantegazza:2011}). 
The remaning assertions follow from the smooth dependence of $L$ on its variables, the H\"older regularity of $\affine$ and the continuity and boundedness of $\forcing$. We refer to \cite[Corollary 7.2]{ATW:1993} for more details in the mean curvature setting.
\end{proof}

\begin{remark}\label{rem:comparison}
As in the standard mean curvature case, using the Hamilton-type arguments \cite[Chapter 2]{Mantegazza:2011}, one can show the 
following comparison principle: 
if $A_0\Subset B_0,$ and $\{A(t)\}_{t\in [0,T)}$ and $\{B(t)\}_{t\in [0,T)}$ are generalized power smooth mean curvature flows starting from $A_0$ and $B_0,$ respectively, then $A(t)\Subset B(t)$ for any $t\in [0,1).$
\end{remark}

Let $E(\cdot)$ be any $\GMM$ starting from the smooth bounded set 
$\initialset=S_0$ and let 
a sequence $\tau_j\to0^+$ and flat flows $E(\tau_j,k)$ be such that 
\begin{equation}\label{flats_converge}
\lim_{j\to+\infty}\,|E(\tau_j,\intpart{t/\tau_j}) \Delta E(t)| = 0\quad \text{for all $t\ge0.$} 
\end{equation}
In order to show Theorem 
\ref{teo:consistency}, it suffices to prove that for any $T\in (0,T^\dag)$
\begin{equation}\label{gmm_teng_smooth}
E(t) = S(t),\quad t\in [0,T). 
\end{equation}

Given $T\in(0,T^\dag)$ and $a\in [0,T),$ let $\rho,\sigma>0$ and $L^\pm[r,s,a,t]$ be as in Definition \ref{def:stable_smooth_flow} (b), and given $s\in [0,\sigma],$ let $\tau_2:=\tau_2(s)$ 
be given  by Proposition \ref{prop:properties_of_stable_flows} (b).
Let also $t^*$ be as in Proposition \ref{prop:properties_of_stable_flows} (c). 

The proof of \eqref{gmm_teng_smooth} 
basically follows applying inductively the following auxiliary lemma. 

\begin{lemma}\label{lem:vengriya_prezidenti}
Assume that $a \in [0,T)$ 
and $s\in (0,\sigma)$ 
are such that
\begin{equation}\label{good_inclusion091}
L^-[0,s,a,a] \subset 
E(\tau_j,k_0)\subset L^+[0,s,a,a],
\end{equation}
where $k_0:=\intpart{a/\tau_j}.$
Then there exists $\bar t\in(0,t^*]$ depending only on $t^*$ 
and $\rho$, such that 
\begin{equation}\label{eq:t_bar}
L^-[0,s,a,a+k\tau_j]\subset
E(\tau_j,k_0+k)\subset L^+[0,s,a,a+k\tau_j]
\end{equation}
for all $j\ge1$ with $\tau_j\in(0,\tau_2(s))$ and $k=0,1,\ldots,\intpart{\bar t/\tau_j}$ provided that $a+k\tau_j<T.$ Moreover, if $a+\bar t<T,$  for any $s\in(0,\sigma]$ with $\atwfunction(2s)<\sigma/4$ there exists 
$\indexsATW>1$ such that 
\begin{equation}\label{hatto_hindiya}
L^-[0,4\atwfunction(2s),a+\bar t,a+\bar t] \subset E(\tau_j,k_0+\bar k_j) \subset L^+[0,4\atwfunction(2s),a+\bar t,a+\bar t],
\end{equation}
whenever $j>\indexsATW$ and 
$\bar k_j:=\intpart{\bar t/\tau_j}.$
\end{lemma}

\begin{proof}
We closely follow the arguments of the proof of the consistency in \cite{ATW:1993}. 
In view of \eqref{good_inclusion091} and Proposition 
\ref{prop:properties_of_stable_flows} (a)
\begin{equation}\label{bu_jangga}
L^-[\rho/64,s,a,a] \subset L^-[0,s,a,a] \subset 
E(\tau_j,k_0) \subset L^+[0,s,a,a] \subset L^+[\rho/64,s,a,a]. 
\end{equation}
Thus, by the definition of $L^\pm[\rho,s,a,a]$ 
$$
\begin{cases}
B_{\rho/64}(x)\subset E(\tau_j,k_0) & \text{if $x\in L^-[\rho/64,s,a,a],$}\\[2mm] 
B_{\rho/64}(x)\cap E(\tau_j,k_0) = \emptyset & \text{if $x\notin L^+[\rho/64,s,a,a].$}
\end{cases}
$$
Thus, applying Theorem \ref{teo:comparison_with_balls} we find a constant $C_8>1$ depending only on $n$, $\norm$, 
$\affine$, $\rho$ and $\|\forcing\|_\infty$ such that 
\begin{equation}\label{sammitga_tayyor}
\begin{cases}
B_{\rho/64 - C_8 i\tau_j}(x)\subset E(\tau_j,k_0+i) & \text{if $x\in 
L^-[\rho/64,s,a,a],$}\\[2mm] 
B_{\rho/64-C_8i\tau_j}(x)\cap E(\tau_j,k_0+i) = \emptyset & 
\text{if $x\notin L^+[\rho/64,s,a,a]$}
\end{cases}
\end{equation}
for all $0\le i\tau_j \le \frac{\rho}{128C_8}.$ 
By \eqref{bu_jangga} and \eqref{sammitga_tayyor} for such $i$
\begin{equation}\label{izzuddin_al_qassam}
L^-[\rho/32-C_{8}i\tau_j,s,a,a ] \subset E(\tau_j,k_0+i) \subset 
L^+[\rho/32-C_{8}i\tau_j,s,a,a].
\end{equation}
Now using the inequality 
$$
\frac{\rho}{32} - C_{8}i\tau_j \le 
\frac{\rho}{2} - i\tau_j
$$
(recall that $C_{8}>1$) 
and the definition of $L^\pm[r,s,a,a]$ in \eqref{izzuddin_al_qassam}, we find 
\begin{equation}\label{qassam_brigade}
L^-[\rho/2-i\tau_j,s,a,a] \subset E(\tau_j,k_0+i) \subset 
L^+[\rho/2-i\tau_j,s,a,a].
\end{equation}
Now if $i\tau_j\le t^*,$ by Proposition \ref{prop:properties_of_stable_flows} (c) and \eqref{qassam_brigade}
\begin{equation}\label{naqba}
L^-[\rho,s,a,a+i\tau_j] \subset E(\tau_j,k_0+i) \subset 
L^+[\rho,s,a,a+i\tau_j].
\end{equation}

Let us define 
$$
\bar t:=\min\Big\{t^*, \tfrac{\rho}{128C_{8}}\Big\}.
$$
By \eqref{naqba}
\begin{equation*}
L^-[\rho,s,a,a+i\tau_j ] \subset E(\tau_j,k_0+i) \subset 
L^+[\rho,s,a,a+i\tau_j],\quad i=0,1,\ldots,\intpart{\bar t/\tau_j}
\end{equation*}
provided $a+i\tau_j<T.$ We claim that for any $j>1$ with $\tau_j\in (0,\tau_2(s))$ 
\begin{equation}\label{urushlar}
L^-[0,s,a,a+i\tau_j] \subset E(\tau_j,k_0+i) \subset 
L^+[0,s,a, a+i\tau_j],\quad i=0,1,\ldots,\intpart{\bar t/\tau_j},
\end{equation}
with $a+i\tau_j<T.$
Indeed, let 
$$
\bar r:=\inf\Big\{r\in [0,\rho]:\,\, E(\tau_j,k_0+i) \subset 
L^+[r,s,a,a+i\tau_j],\quad i=0,\ldots,\intpart{\bar t/\tau_j},\,a+i\tau_j<T\Big\}.
$$
To prove the claim we need to show that 
\begin{equation}\label{eq:rbar=0}
\bar r=0.
\end{equation}

In view of \eqref{naqba} the infimum is taken over a nonempty set. By contradiction, assume that $\bar r>0.$ By the continuity of $L^+[\cdot,s,a,a+i\tau_j]$ at $r=\bar r,$ there exists the
smallest integer $k\le \intpart{\bar t/\tau_j}$ 
(clearly, $k>0$ by \eqref{bu_jangga} and the assumption $\bar r>0$) for which 
\begin{equation}\label{tunnel_berkit}
\p E(\tau_j,k_0+k) \cap \p L^+[\bar r,s,a,a+k\tau_j] \ne\emptyset.
\end{equation}
Moreover, by the minimality of $k$ and the definition of $\bar r$
$$
E(\tau_j,k_0+k-1)\subset L^+[\bar r,s,a,a+(k-1)\tau_j],\quad 
E(\tau_j,k_0+k)\subset L^+[\bar r,s,a,a+k\tau_j].
$$
By Proposition \ref{prop:properties_of_stable_flows} (b) (recall that $k_0=\intpart{a/\tau_j}$)
$$
\affine\Big(\tfrac{ \sd_{L^+[\bar r,s,a,a+k\tau_j-\tau_j]}}{\tau_j}\Big) > - \kappa_{L[\bar r,s,a,a+k\tau_j]}^\norm - \forcing + \frac s2 \quad \text{on $\p L[\bar r,s,a,a+k\tau_j].$}
$$
Thus, applying Lemma \ref{lem:discrete_comparison} (a) with 
$E:=F(\tau_j,k_0+k-1),$ $E_\tau:=F(\tau_j,k_0+k),$ 
$F:=L^+[\bar r,s,a,a+(k-1)\tau_j]$ and  
$F_\tau:=L^+[\bar r,s,a,a+k\tau_j]$ we obtain 
$$
\p E(\tau_j,k_0+k) \cap \p L^+[\bar r,s,a,a+k\tau_j] = \emptyset,
$$
which contradicts 
\eqref{tunnel_berkit}. Thus \eqref{eq:rbar=0} is proven. Analogous contradiction argument based on Lemma \ref{lem:discrete_comparison} (b) and Proposition \ref{prop:properties_of_stable_flows} (b) for $s<0$ shows the validity of  \eqref{urushlar}. 
This concludes the proof of \eqref{eq:t_bar}.

Now, let us prove \eqref{hatto_hindiya}. By construction $L^-[0,2s,a,a]\Subset L^-[0,s,a,a]$ and $L^+[0,s,a,a]\Subset L^+[0,2s,a,a].$ Thus, by the 
comparison  principle in Remark \ref{rem:comparison}, 
$L^-[0,2s,a,t]\Subset L^-[0,s,a,t]$ and $L^+[0,s,a,t]\Subset L^+[0,2s,a,t]$ for all $t\in [a,T].$ Since $L^\pm[0,\cdot,\cdot,\cdot]$ continuously varies, there is $\indexsATW>1$ such that for any $j>\indexsATW$ 
\begin{multline}\label{misrliklar}
L^-[0,2s,a,a+\bar t]\subset  L^-[0,s,a,a+\bar k_j\tau_j] \\
\subset E(\tau_j,\bar k_j)
\subset 
L^+[0,s,a,a+\bar k_j\tau_j]
\subset 
L^+[0,2s,a,a+\bar t],
\end{multline}
where we recall that 
$\bar k_j:=\intpart{\bar t/\tau_j}$, and 
we used  $\bar k_j\tau_j = \intpart{\bar t/\tau_j}\tau_j\nearrow \bar t.$ 
Let the function $\atwfunction$ be 
given by Proposition \ref{prop:properties_of_stable_flows} (d) so that 
$$
\max_{x\in\p L^\pm[0,2s,a,a+\bar t]}\dist(x,\p L^\pm[0,0,a,a+\bar t]) \le 
\atwfunction
(2s).
$$
Since, by the definition of $L^\pm$,  
we have $L^\pm[0,0,a,a+\bar t] = E(a+\bar t)= L^\pm[0,0,a+\bar t,a+\bar t]$ and,
by the definition of $\atwfunction$, 
$$
\dist(\p L^\pm[0,\atwfunction(2s),a+\bar t,a+\bar t], \p L^\pm[0,0,a+\bar t,a+\bar t]) = \atwfunction(2s),
$$
it follows that 
$$
L^-[0,4\atwfunction(2s),a+\bar t,a+\bar t]\subset L^-[0,2s,a,a+\bar t]
\quad\text{and}\quad 
L^+[0,2s,a,a+\bar t] \subset L^+[0,4\atwfunction(2s),a+\bar t,a+\bar t].
$$
Using this and \eqref{misrliklar} we get 
$$
L^-[0,4\atwfunction(2s),a+\bar t,a+\bar t] \subset E(\tau_j,k_0+\bar k_j) \subset L^+[0,4\atwfunction(2s),a+\bar t,a+\bar t].
$$
\end{proof}

{\it Proof of Theorem \ref{teo:consistency}}. 
Let $\bar t$ be given by Lemma \ref{lem:vengriya_prezidenti},
$$
N:=\intpart{T/\bar t} + 1
$$
and let $\sigma_0\in(0,\sigma/16)$ be such that the numbers
$$
\sigma_l=4\atwfunction(2\sigma_{l-1}),\quad l=1,\ldots,N,
$$
satisfy $\sigma_l\in (0,\sigma/16).$ By the monotonicity and continuity of $h$ together with $\atwfunction(0)=0,$ and the finiteness of $N,$ 
such a choice of $\sigma_0$ is possible (indeed, it is enough to observe that if $\sigma_0\to0,$ then all $\sigma_l\to0$). 

Fix any $s\in(0,\sigma_0)$ and let 
$$
a_0(s):=s,\quad a_l(s):=4\atwfunction(2a_{l-1}(s)),\quad l=1,\ldots,N.
$$
Note that $a_l(s)\in (0,\sigma_l).$ 
In particular, the numbers 
$\indexATW(a_l(s)),$ 
given by the last assertion of Lemma \ref{lem:vengriya_prezidenti}, are well-defined. Let also 
$$
j_l^s:=\max\{j\ge1:\,\, \tau_j\notin (0,\tau_2(a_l(s)))\}
$$
and
$$
j_s:= 1 + \max_{l=0,\ldots,N}\,\max \{
\indexATW
(a_l(s)),j_l^s\}.
$$

By the definition of $L^\pm,$ 
$$
L^-[0,s,0,0] \subset S(0)=\initialset=E(\tau_j,0) \subset L^+[0,s,0,0]
$$
for all $j>j_s$ (basically, this is true for all $j$). Therefore, by Lemma \ref{lem:vengriya_prezidenti} applied with $a=0$ (so that $k_0:=\intpart{a/\tau_j}=0$) we find 
$$
L^-[0,s,0,k\tau_j]\subset E(\tau_j,k) \subset L^+[0,s,0,k\tau_j],\quad k=0,1,\ldots,\bar k_j,
$$
where $\bar k_j:=\intpart{\bar t/\tau_j}.$ Moreover, since $s\in (0,\sigma_0,)$ by the last assertion of Lemma \ref{lem:vengriya_prezidenti} and the definition of $a_l(s)$ 
$$
L^-[0,a_1(s),\bar t,\bar t]\subset E(\tau_j,\bar k_j) \subset L^+[0,a_1(s),\bar t, \bar t] 
$$
for all $j\ge 
j_s.$ Hence, we can reapply Lemma \ref{lem:vengriya_prezidenti} with $s:=a_1(s),$ $a=\bar t$ and $k_0=\bar k_j,$ to find 
$$
L^-[0,a_1(s),\bar t,\bar t + k\tau_j]\subset E(\tau_j,\bar k_j+k) \subset L^+[0,a_1(s),\bar t,\bar t+ k\tau_j],\quad k=0,1,\ldots,\bar k_j.
$$
In particular, since $j>
j_s> 
\indexATW
(a_1(s)),$ again by the last assertion of  Lemma \ref{lem:vengriya_prezidenti}  we deduce 
$$
L^-[0,a_2(s),2\bar t,2\bar t]\subset E(\tau_j,2\bar k_j) \subset L^+[0,a_2(s),2\bar t,2\bar t]. 
$$
Repeating this argument at most $N$ times, for all $j> 
j_s$ we find 
\begin{equation}\label{nadomatlar671}
L^-[0,a_l(s),l\bar t,l\bar t + k\tau_j]\subset E(\tau_j,l\bar k_j+k) \subset L^+[0,a_l(s),l\bar t,l\bar t + k\tau_j],\quad k=0,1,\ldots,\bar k_j
\end{equation}
whenever $l=0,\ldots,N$ with $l\bar t + k\tau_j < T.$ 

Now take any $t\in(0,T),$ and let $l:=\intpart{t/\bar t}$  and $k=\intpart{t/\tau_j} - l\bar k_j$ so that $l\bar k_j + k = \intpart{t/\tau_j}.$ By means of $l$ and $k,$ as well as the definition of $\bar k_j$ we represent \eqref{nadomatlar671} as 
\begin{multline}\label{headlight09}
L^-\Big[0,a_l(s),l\bar t, l\bar t + \tau_j\intpart{\tfrac{t}{\tau_j}} - l\tau_j\intpart{\tfrac{\bar t}{\tau_j}}\Big]\\
\subset 
E\Big(\tau_j, \intpart{\tfrac{t}{\tau_j}}\Big) 
\subset 
L^+\Big[0,a_l(s),l\bar t, l\bar t + \tau_j\intpart{\tfrac{t}{\tau_j}} - l\tau_j\intpart{\tfrac{\bar t}{\tau_j}}\Big]
\end{multline}
for all $j>j_s.$ Since 
$$
\lim\limits_{j\to+\infty} \Big(l\bar t + \tau_j\intpart{\tfrac{t}{\tau_j}} - l\tau_j\intpart{\tfrac{\bar t}{\tau_j}}\Big) = t,
$$
by the continuous dependence of $L^\pm$ on its parameters, as well as the convergence \eqref{flats_converge} of the flat flows, letting $j\to+\infty$ in \eqref{headlight09} we obtain
\begin{equation}\label{endiyonimga}
L^-[0,a_l(s),l\bar t, t]\subset E(t) \subset L^+[0,a_l(s),l\bar t, t],
\end{equation}
where, due to the $L^1$-convergence, the inclusions in \eqref{endiyonimga}
hold possibly up to some negligible set. Now we let $s\to0^+$ and recalling that $a_l(s)\to0$ (by the continuity of $h$ 
and assumption $h(0)=0$), from \eqref{endiyonimga} we deduce 
\begin{equation*}
L^-[0,0,l\bar t,t]\subset E(t) \subset L^+[0,0,l\bar t,t].
\end{equation*}
Now recalling $L^\pm[0,0,a,t]=S(t)$ for $t\in [a,T]$ we get 
$$
E(t)=L^\pm[0,0,l\bar t,t] = S(t).
$$
\qed
\section{Evolution of convex sets: proof of Theorem 
\ref{teo:gmm_convex_intro}
}\label{sec:evolution_of_bounded_convex_sets}
In this section we prove 
Theorem \ref{teo:gmm_convex_intro}; thus we 
assume $\norm$ is Euclidean, $\forcing\equiv0$ and $\affine(r)=
\sign(r)|r|^\DeGiorgiexponent
$ for some $\DeGiorgiexponent>0$.
We need the following result,
 proven in \cite{Schulze:2005}.

\begin{theorem}
\label{teo:convex_smooth_evol}
Let $\initialset\subset\Rn$ be a bounded $C^{2+\Holderexponent}$-convex set for 
some $\Holderexponent\in(0,1]$. Then there exist $T^*>0$ and a 
unique $C^1$-in time 
flow $\{E(t)\}_{t\in [0,T^*)}$ starting from $\initialset$ such that 
each $E(t)$ is $C^2,$ convex and 
$$
\normalvelocityofEt = - \meancurvature^{1/\DeGiorgiexponent} \quad\text{on $\p E(t)$}
$$
for all $t\in [0,T^*).$ Moreover, $\{E(t)\}$ is stable in the sense of Definition \ref{def:stable_smooth_flow} and $|E(t)|\to0$ as $t\nearrow T^*.$
\end{theorem}

\medskip
\noindent
{\it Proof of Theorem \ref{teo:gmm_convex_intro}\rm(i)}.
If $\Int{\initialconvexset}=\emptyset,$ then by convexity, 
$|\initialconvexset|=0,$ and we are done. 
Otherwise, since $\forcing \equiv 0$, translating 
if necessary, we assume that $\Int{\initialconvexset}$ contains the origin. 
Suppose that $\GMM(\shorthandgeneralfunctional,\cS^*,
\initialconvexset)$ contains at least two different $\GMM$'s, say, 
$\convex'(\cdot)$ and $\convex''(\cdot)$ so that there exist
$T>0$ and $\epsilon_0>0$ such that 
\begin{equation}\label{pistirma0}
|\convex'(T)\Delta \convex''(T)|\ge 2\epsilon_0.
\end{equation}
Also, by the uniform time-continuity  (see Theorem \ref{teo:existence_gmm}(i))
 of $\convex'(\cdot),\convex''(\cdot),$ 
possibly decreasing $T$ if necessary, we may further 
assume that $|\convex'(T)|,
|\convex''(T)|\ge 2\epsilon_0.$ 
Fix any $\scalefactor
\in(0,1)$ sufficiently close to $1$ so that 
\begin{equation}\label{bayramxon015}
|\inverseofscale \convex'(\scalefactor^{1/\DeGiorgiexponent} T)
\setminus 
\scalefactor \convex'(\scalefactor^{-1/\DeGiorgiexponent} T)| <\epsilon_0,\quad 
|\inverseofscale
 \convex''(\scalefactor^{1/\DeGiorgiexponent} T)
\setminus 
\scalefactor
 \convex''(\scalefactor^{-1/\DeGiorgiexponent} T)|
<\epsilon_0.
\end{equation}
Let us define $\initialset^\scalefactor:=\scalefactor \initialconvexset$ and 
$F_0^\scalefactor:= \inverseofscale
\initialconvexset$, so that 
 $\initialset^\scalefactor \Subset \initialconvexset
\Subset F_0^\scalefactor.$ 
Now we choose smooth convex sets $P_0$ and $Q_0$ such that 
$$
\initialset^\scalefactor
 \Subset P_0\Subset \initialconvexset
\Subset Q_0\Subset F_0^\scalefactor.
$$
If we define the corresponding flat flows, then 
by Corollary \ref{cor:strong_comparison}
\begin{equation}\label{qabrtoshi}
E^\scalefactor(\tau,k) \Subset P(\tau,k)\Subset C(\tau,k)\Subset Q(\tau,k)\Subset F^\scalefactor(\tau,k)
\end{equation}
for any $\tau>0$ and $k\ge0$ with $|E^\scalefactor(\tau,k)|>0.$ 

Since $P_0$ and $Q_0$ are smooth and convex, by Theorem \ref{teo:convex_smooth_evol} the corresponding smooth flows exist and 
disappear at a maximal time (however, GMM starting from them exists for all times). In particular, by consistency (Theorem \ref{teo:consistency}) 
both  $\GMM(\shorthandgeneralfunctional,\cS^*,P_0)$ and 
$\GMM(\shorthandgeneralfunctional,\cS^*,Q_0)$ are singletons, say, 
$\{P(\cdot)\}$ and $\{Q(\cdot)\}.$

Let $\tau_j'\searrow0$ and $\tau_j''\searrow0$ be sequences for which 
$$
|\convex(\tau_j',\intpart{t/\tau_j'})\Delta \convex'(t)| \to0 
\quad\text{and}\quad 
|\convex(\tau_j'',\intpart{t/\tau_j''})\Delta \convex''(t)| \to0 
$$
as $j\to+\infty$ for all $t\ge0.$ In view of 
Corollary \ref{cor:power_law_case}, as $j \to +\infty$ we have
$$
|E^\scalefactor(\tau_j',\intpart{t/\tau_j'})\Delta [\scalefactor 
\convex'(\scalefactor^{-1/\DeGiorgiexponent} t)]| \to0
\quad\text{and}\quad
|F^\scalefactor(\tau_j',\intpart{t/\tau_j'})\Delta [\inverseofscale
\convex''(\scalefactor^{-1/\DeGiorgiexponent} t)]| \to0
$$
and 
$$
|E^\scalefactor(\tau_j'',\intpart{t/\tau_j''})\Delta 
(\scalefactor \convex'(\scalefactor^{-1/\DeGiorgiexponent} t))| \to0
\quad\text{and}\quad
|F^\scalefactor(\tau_j'',\intpart{t/\tau_j''})\Delta [\inverseofscale 
\convex''(\scalefactor^{1/\DeGiorgiexponent} t)]| \to0.
$$
Now applying \eqref{qabrtoshi} with $\tau_j'$ we deduce 
\begin{equation}\label{bobur_askarlari1}
\scalefactor \convex'(\scalefactor^{-1/\DeGiorgiexponent} t) \subset P(t) \subset \convex'(t)
\subset Q(t)\subset \inverseofscale
\convex'(\scalefactor^{1/\DeGiorgiexponent} t)\quad \text{for all $t\in [0,T],$} 
\end{equation}
and appyling it with $\tau_j''$ we deduce
\begin{equation}\label{bobur_askarlari2}
\scalefactor \convex''(\scalefactor^{-1/\DeGiorgiexponent} t) \subset P(t) \subset 
\convex''(t)\subset Q(t)\subset \inverseofscale \convex''(\scalefactor^{1/\DeGiorgiexponent} t)
\quad \text{for all $t\in [0,T].$} 
\end{equation}
By \eqref{bayramxon015}
$$
|Q(T)\setminus P(T)| \le |[
\inverseofscale C'(\scalefactor^{1/\DeGiorgiexponent} T)]\setminus [\scalefactor C'(\scalefactor^{-1/\DeGiorgiexponent} T)]| \le \epsilon_0. 
$$
However, in view of \eqref{bobur_askarlari1} and \eqref{bobur_askarlari2}  
as well as of \eqref{pistirma0},
$$
2\epsilon_0\le |\convex''(T)\Delta \convex'(T)| \le 
|Q(T)\setminus P(T)| \le \epsilon_0,
$$
a contradiction.
\qed

\medskip
\noindent
{\it Proof of Theorem \ref{teo:gmm_convex_intro}\rm(ii)}.
By the Kuratowski convergence of $(\p \initialconvexsetindex)$ 
and comparison principles, given $\scalefactor\in(0,1),$ as in the 
proof of (i),
$$
\scalefactor \convex(t\scalefactor^{-1/\DeGiorgiexponent}) \subset 
\convex_\naturalindex(t)\subset 
\inverseofscale
 \convex(t\scalefactor^{1/\DeGiorgiexponent})\quad\text{for 
all $t\ge0$},
$$
provided $\naturalindex \in \mathbb N$ is large enough depending only on 
$\scalefactor.$ Since 
the sequence 
$(P(\convex_\naturalindex(t)))$ is bounded (by the supremum of 
$P(\initialconvexsetindex)$), up to a subsequence, 
$\convex_{\naturalindex}(t) \to \convex'(t)$ 
in $L^1(\Rn)$ as $i\to+\infty.$ Then
$$
\scalefactor \convex(t\scalefactor^{-1/\DeGiorgiexponent}) \subset \convex'(t)\subset 
\inverseofscale 
\convex(t\scalefactor^{1/\DeGiorgiexponent})\quad\text{for all $t\ge0.$} 
$$
Now letting $\scalefactor\to1$ we get $\convex'(t)=\convex(t),$ i.e., 
the limit of $(\convex_\naturalindex)$ is
independent of the subsequence. Thus, the thesis follows.
\qed

\section{Minimizing movements in the class $\convexsets$:
proof of Theorem 
\ref{teo:GMM_in_the_class_conv}}
\label{sec:minimizing_movements_in_the_class_conv}
In this section we suppose $\forcing \equiv 0$. 
Motivated by Conjecture \ref{conj:3_2}, 
we study $\GMM$ in the class $\convexsets$. 
Notice that, due to the convexity constraint, the first variation of 
$\shorthandgeneralfunctional$ is 
nonlocal, and thus, in general we cannot write a pointwise Euler-Lagrange equation;
therefore, the nature of $\GMM(\shorthandgeneralfunctional,\convexsets,
\initialconvexset)$ 
seems not  clear. 
Moreover, to prove Theorem \ref{teo:GMM_in_the_class_conv}
we cannot apply the 
techniques used in the proof of Theorem \ref{teo:existence_gmm}, based on cutting 
and filling with balls (because they lead to the  lost of convexity). 

{\it 
Proof of Theorem \ref{teo:GMM_in_the_class_conv}.}
Let $\initialconvexset\in \convexsets.$ 
Without loss of generality we assume that the interior of $\initialconvexset$ 
is nonempty. By the $L_\loc^1$-closedness of $\convexsets$ and the 
$L_\loc^1$-lower semicontinuity of $\shorthandgeneralfunctional(\cdot;\convex,\tau)$ 
for any $\convex \in\convexsets,$ there exists a minimizer of $\shorthandgeneralfunctional(\cdot;\convex,\tau)$ in $\convexsets.$ 
By truncation, we can readily show that every minimizer $
\convex_\tau$ satisfies $\convex_\tau\subset \convex.$ 
Now we define flat flows $\{\convex(\tau,k)\};$ clearly, 
\begin{equation}\label{madhaliyalr}
\initialconvexset=\convex(\tau,0)\supset \convex(\tau,1)\supset \ldots\quad\text{and}\quad 
\anisotropicperimeter(\initialconvexset)=\anisotropicperimeter(\convex(\tau,0))\ge 
\anisotropicperimeter(\convex(\tau,1)) \ge\ldots.
\end{equation} 
For any $\tau>0,$ define
$$
h_\tau(t):=\anisotropicperimeter(\convex
(\tau,\intpart{t/\tau})),\quad t\ge0.
$$
By \eqref{madhaliyalr} $\{h_\tau\}$ are nonincreasing nonnegative functions satisfying  $h_\tau(0)=P(\initialconvexset)$. Thus, Helly's selection theorem 
implies the existence of $\tau_j\to0^+$ and a nonincreasing function $h_0:\R_0^+\to\R_0^+$ such that $h_{\tau_j}\to h_0$ in $\R_0^+.$ 
Since $h_0$ is monotone, its discontinuity set $J\subset\R_0^+$ is at most countable. Let $Q\subset \R_0^+\setminus J$ be any countable dense set in $\R_0^+$. By compactness in $BV$, passing to a further not relabelled sequence $\tau_j$ and  using a diagonal argument, for any $s\in Q\cup J,$ we can define a $\convex
(s)\subset\Rn$ such that 
$$
\convex
(\tau_j,\intpart{s/\tau_j}) \to \convex
(s)\quad\text{in $L^1(\Rn)$ as $j\to+\infty.$}
$$
Since each $\convex(\tau_j,\intpart{s/\tau_j})$ is convex, so is $\convex(s)$ and hence, 
$$
\lim\limits_{j\to+\infty}\,h_{\tau_j}(s) = \lim\limits_{j\to+\infty}\, 
\anisotropicperimeter(\convex
(\tau_j,\intpart{s/\tau_j})) = 
\anisotropicperimeter(\convex
(s)) = h_0(s).
$$
Now take any $0<s\in \R_0^+\setminus [Q\cup J]$ so that $h_0$ is continuous at $s.$ By density of $Q$ in $\R_0^+,$ we can choose sequences $Q\ni a_k\nearrow s$ and $Q\ni b_k\searrow s.$ By \eqref{madhaliyalr} the maps $k\mapsto \convex
(a_k)$ and $k\mapsto 
\anisotropicperimeter(\convex
(a_k))$ are nonincreasing and the maps $k\mapsto \convex
(b_k)$ and $k\mapsto \anisotropicperimeter(\convex
(b_k))$ are nondecreasing. 
Let 
$$
\bigcap_{k\ge1} \convex
(a_k)=:\convex
(s)^*\supset \convex
(s)_*= \bigcup_{k\ge1} \convex
(b_k).
$$
By the continuity of $g$ at $s,$ both  $P_\norm(\convex
(a_k))$ and $P_\norm(\convex
(b_k))$ converges to $g(s),$ and therefore, $P_\norm(\convex
(s)^*) = P_\norm(\convex
(s)_*).$  Thus, by the convexity of both $\convex
(s)_*$ and $\convex
(s)^*$, 
it follows  
$$
\convex
(s)_*=\convex
(s)^*=:\convex(s)
$$
 and  $|\convex(a_k)\Delta \convex(s)|\to0.$ Let us show that 
\begin{equation}\label{irans_division}
\convex
(\tau_j,\intpart{s/\tau_j}) \to \convex(s)\quad\text{in $L^1(\Rn)$ as $j\to+\infty.$}
\end{equation}
Notice that
\begin{equation}\label{shsudhuhs}
\limsup\limits_{j\to+\infty} \,\limsup_{k\to+\infty}\,(h_{\tau_j}(a_k) - h_{\tau_j}(b_k)) = 0,
\end{equation}
otherwise, $h_0$ cannot be continuous at $s.$ Therefore, in view of the inclusion
$$
\convex
(\tau_j,\intpart{a_k/\tau_j}) \supset \convex
(\tau_j,\intpart{s/\tau_j}) \supset \convex
(\tau_j,\intpart{b_k/\tau_j})
$$
we have 
$$
\bigcap_{k\ge1} \convex
(\tau_j,\intpart{a_k/\tau_j}) \supset \convex
(\tau_j,\intpart{s/\tau_j}) \supset \bigcup_{k\ge1} \convex
(\tau_j,\intpart{b_k/\tau_j})
$$
and by \eqref{shsudhuhs} and convexity of these sets, for any $\epsilon,$
$$
|\convex
(\tau_j,\intpart{a_k/\tau_j}) \setminus  \convex
(\tau_j,\intpart{s/\tau_j})|<\epsilon
$$
provided that $k>k_\epsilon>0.$ Hence, letting $j\to+\infty$ 
in the estimate 
\begin{multline*}
|\convex
(\tau_j,\intpart{s/\tau_j}) \Delta \convex(s)| \\
\le |\convex
(\tau_j,\intpart{s/\tau_j}) \Delta \convex
(\tau_j,\intpart{a_k/\tau_j})|
+|\convex
(\tau_j,\intpart{a_k/\tau_j})\Delta \convex(a_k)| + |\convex(a_k)\Delta 
\convex(s)|,
\end{multline*}
and then $k\to+\infty$ and $\epsilon\to 0^+$ in 
$$
\limsup_{j\to+\infty}\,|\convex
(\tau_j,\intpart{s/\tau_j}) \Delta \convex(s)| \\
\le \epsilon 
+ |\convex(a_k)\Delta \convex(s)|,
$$
we get \eqref{irans_division}. Thus, by definition the family $\{\convex(s)\}_{s\ge0}$ is a $\GMM$.
\qed 

Suppose (compare with the next proposition)
that for any bounded convex set
$\initialconvexset,$
$\shorthandgeneralfunctional(\cdot;\initialconvexset,\tau)$ 
admits a convex minimizer in $\cS^*.$ 
In this case, the flat flows $\convex
(\tau,k)$ consisting of those convex sets are also flat flows in $\convexsets;$ by Theorem \ref{teo:gmm_convex_intro} $\convex
(\tau,\intpart{t/\tau})\to \convex
(t)$ 
as $\tau\to 0^+$ (because $\GMM(
\shorthandgeneralfunctional
,\cS^*,\initialconvexset)$ 
is a singleton). Then clearly, 
$\convex
(\cdot)\in \GMM(\shorthandgeneralfunctional,\convexsets,\initialconvexset)$
 (in fact is a minimizing movement). 
However, it is not clear whether $\GMM(\shorthandgeneralfunctional,
\convexsets,\initialconvexset)$ contains other generalized minimizing movements.

\begin{proposition}\label{prop:fortunate}
Suppose that $\norm$ is an elliptic $C^{3}$-anisotropy
and that $\forcing \equiv 0$.
Assume that $\affine$ is concave in $(0,+\infty).$  
Then for any bounded convex open set $\initialconvexset$ and $\tau>0$ the 
minimal and maximal minimizers of $\shorthandgeneralfunctional(\cdot;\initialconvexset,\tau)$ in $\cS^*$ are  convex.
\end{proposition}

\begin{proof}
We follow the notation of \cite{CCh:2006}. Since $\initialconvexset$ is convex, $\d_{\initialconvexset}$ is concave in $\initialconvexset,$ and hence, so is $u_0:=\affine(\d_{\initialconvexset}/\tau).$ 
Consider 
$$
\sE(v):=
\begin{cases}
\displaystyle \int_{\initialconvexset}\norm^o(Dv) + \int_{\initialconvexset}\frac{(v + u_0)^2}{2} ~dx & \text{if $v\in L^2(\initialconvexset)\cap BV(\initialconvexset),$}\\
+\infty & \text{if $L^2(\initialconvexset)\setminus BV(\initialconvexset),$}
\end{cases}
$$
where $\phi^o(Dv)$ is the $\phi^o$-total variation of $v.$
One checks that $\sE$ admits a unique minimizer $v_0\in L^2(\initialconvexset)\cap BV(\initialconvexset).$ Since $\sE$ is convex, by \cite{Anzellotti:1984} 
there exists $z\in L^\infty(\initialconvexset, \Rn)$ with $\div z\in L^2(\initialconvexset)$ such that
\footnote{Concerning
the notation $z \cdot Dv_0$ and $\dualnorm(Dv_0)$, see \cite{CCh:2006}.} 
$$
-\div z + v_0 + u_0=0, \quad \norm(z)\le 1,\quad 
\int_{\initialconvexset} z\cdot Dv_0 = \int_{\initialconvexset}\norm^o(Dv_0).
$$
Repeating the same arguments of \cite[Lemma 5.1]{CCh:2006} we can show that for a.e. $s<0$ the set $\{v_0<s\}$ is a minimizer of 
$$
\sE_s(F):=P_\norm(F) - \int_F(v_0+s)~dx, \qquad F \subseteq \initialconvexset,\quad s\in\R.
$$
If $v_0$ is twice continuously differentiable in $\initialconvexset$ with $|\nabla v_0|>0.$ Then 
$v_0$  is a viscosity solution of the equation 
$$
-\div\,\nabla\norm^o(Dv_0) + v_0 + u_0 = 0.
$$
Since $u_0$ is concave, using \cite[Theorem 1]{ALL:1997}  we find that
 $v_0$ is convex. In case $v_0$ is not sufficiently smooth, we can approximate the equation as in \cite{CCh:2006} and again 
get that $v_0$ is convex. In particular, $\sE_s$ admits a convex minimizer $\{v_0<s\}.$
One can readily check that the sets 
$$
\convex_*:=\bigcup_{s<0}\{v_0<s\}\quad\text{and}\quad 
\convex^*:=\bigcap_{s>0}\{v_0\le s\}
$$
are the minimal and maximal (convex) minimizers of the functional 
$\sE_0(\cdot)=\shorthandgeneralfunctional(\cdot;\initialconvexset,\tau).$
\end{proof}

As a corollary, we obtain the validity of Conjecture \ref{conj:3_1} 
under a restriction on $\DeGiorgiexponent>0$.

\begin{corollary}[\textbf{Conjecture \ref{conj:3_1} for $\DeGiorgiexponent
\in (0,1]$}]
\label{cor:de_giorgi_unique}
Let $\norm$ be Euclidean, $f(r)=r^\DeGiorgiexponent$ for $r>0$,
$$
\DeGiorgiexponent\in(0,1],
$$
and $\forcing \equiv 0$.
Then for any bounded convex $\initialconvexset\subset\Rn,$ $\GMM(\shorthandgeneralfunctional,\convexsets,\initialconvexset)$ is a singleton 
and coincides with the unique minimizing movement in $\GMM(\shorthandgeneralfunctional,\cS^*,\initialconvexset).$ 
\end{corollary}

\begin{proof}
Let $\convex(\tau,k)_*$ and $\convex
(\tau,k)^*$ 
be flat flows in $\cS^*$ 
consisting of the minimal and maximal minimizers of $\shorthandgeneralfunctional$ (in $\cS^*$), starting from $\initialconvexset.$ By Theorem \ref{teo:gmm_convex_intro} we have 
$$
\lim\limits_{\tau\to0^+}\,\convex
(\tau,\intpart{t/\tau})_* = \convex
(t)
\quad\text{and}\quad 
\lim\limits_{\tau\to0^+}\,\convex
(\tau,\intpart{t/\tau})^* = \convex
(t)\quad\text{in $L^1(\Rn)$ for all $t\ge0,$} 
$$
where $\{\convex
(t)\} = \MM(\shorthandgeneralfunctional,\cS^*,\initialconvexset).$ By Proposition
 \ref{prop:fortunate} both $\convex
(\tau,k)_*$ and $\convex
(\tau,k)^*$ are convex. Note that they are the
 minimal and maximal minimizers of $\shorthandgeneralfunctional$ also in $\convexsets.$ Therefore,
$\GMM(\shorthandgeneralfunctional,\convexsets,\initialconvexset) =
\MM(\shorthandgeneralfunctional,\convexsets,\initialconvexset) = \{\convex
(\cdot)\}.$ 
%
\end{proof}

\appendix
\section{Volume-distance inequality}
In this appendix 
we establish a couple of technical results needed in various 
proofs. In particular, the next result
is crucial, and
is an easy modification of the volume-distance inequality of Almgren-Taylor-Wang \cite{ATW:1993} (see also \cite{LS:1995}).

\begin{lemma}[\textbf{Volume-distance inequality}] \label{lem:volume_distance}
Let $r_0>0$, and $F\in \cS^*$  satisfy 
\begin{equation}\label{lower_density_estimates}
P(F,B_r(x)) \ge 
\lowerbounddensity
r^{n-1},\quad x\in \p F,\,\,\,r\in(0,r_0], 
\end{equation}
for some $\lowerbounddensity>0.$ Then 
for any measurable  $E\subset\Rn$ and any $p,\ell>0$ one has
\begin{equation}\label{eq:volume_distance_inequality}
|E\Delta F| \le 
\begin{cases}
\displaystyle 
\frac{c p^n\ell^n }{r_0^{n-1}}\,\anisotropicperimeter(F) +  \frac{1}{\affine(p)} \int_{E\Delta F}\veloc{F}{x}{\ell}~dx & \text{if $\ell>r_0, p\ell>r_0,$}\\[3mm] 
\displaystyle 
cp^n\ell\,\anisotropicperimeter(F) +  
\frac{1}{\affine(p)}
\int_{E\Delta F}\veloc{F}{x}{\ell}~dx
 & \text{if $\ell\in(0,r_0]$ and $p\ell>r_0,$}\\[3mm]
\displaystyle 
c p\ell \,\anisotropicperimeter(F) +  
\frac{1}{\affine(p)}
\int_{E\Delta F}\veloc{F}{x}{\ell}~dx & \text{if $p\ell\in(0,r_0],$} 
\end{cases}
\end{equation}
where $c:=\frac{10^n\omega_n}{\constnorm\lowerbounddensity}.$
\end{lemma}

\begin{proof}
Define 
$$
\Pvdi
:=\{x\in E\Delta F:\,\,\d_F(x)\ge p\ell\},\quad 
\Qvdi:=\{x\in E\Delta F:\,\,\d_F(x)<p\ell\}.
$$ 
Since $\affine$ is strictly increasing, we have 
$$
\Pvdi 
= \{x\in E\Delta F:\,\, \affine(\d_F(x)/\ell) \ge \affine(p)\},
$$
and hence, by the Chebyshev inequality
$$
|\Pvdi| \le \frac{1}{\affine(p)}\int_P \veloc{F}{x}{\ell}~dx \le  \frac{1}{\affine(p)} \int_{E\Delta F}\veloc{F}{x}{\ell}~dx.
$$
On the other hand, we cover $\Qvdi$ with balls $B_{2p\ell}$ 
of radius $2p\ell$ centered at points of $\p F.$ 
By the Vitali covering lemma we can take a countable subfamily $\{B_{10p\ell}'\}$ still covering $\Qvdi$ with a pairwise disjoint family $\{B_{p\ell}'\}$. 

If $\ell > r_0$ and $p\ell>r_0,$ then by the disjointness of $\{B_{r_0}'\}$ and the estimate \eqref{lower_density_estimates} 
(we cannot apply it with $\ell$)
$$
|\Qvdi| \le \sum_{B_{10p\ell}'} \omega_n(10p\ell)^n =  \tfrac{10^n\omega_n p^n\ell^n }{\lowerbounddensity r_0^{n-1}} \sum_{B_{r_0}'} 
\lowerbounddensity
 r_0^{n-1} \le 
\tfrac{10^n\omega_n p^n\ell^n }{\lowerbounddensity
 r_0^{n-1}} \sum_{B_{r_0}'} P(F, B_{r_0}')\le 
\tfrac{10^n\omega_n p^n\ell^n }{\lowerbounddensity r_0^{n-1}} \, P(F)
$$
On the other hand, if $\ell\le   r_0<p\ell,$ by disjointness of $\{B_\ell'\}$
$$
|\Qvdi| \le \sum_{B_{10p\ell}'} \omega_n(10p\ell)^n = \tfrac{10^n\omega_n p^n\ell }{\lowerbounddensity} \sum_{B_{p\ell}'} \lowerbounddensity \ell^{n-1}
\le 
\tfrac{10^n\omega_np \ell}{\lowerbounddensity} \sum_{B_{p\ell}'} P(F, B_{\ell}') \le 
\tfrac{10^n\omega_n p\ell }{\lowerbounddensity}\,P(F).
$$
Finally, if $p\ell\le r_0$ using the disjointness of $\{B_{p\ell}'\},$ 
$$
|\Qvdi| \le \sum_{B_{10p\ell}'} \omega_n(10p\ell)^n = \tfrac{10^n\omega_n p\ell }{\lowerbounddensity} \sum_{B_{p\ell}'} 
\lowerbounddensity
 p^{n-1}\ell^{n-1}
\le 
\tfrac{10^n\omega_n p\ell }{\lowerbounddensity
} \sum_{B_{p\ell}'} P(F, B_{p\ell}') \le 
\tfrac{10^n\omega_n p\ell }{\lowerbounddensity
}\,P(F).
$$
Now use \eqref{norm_bounds} and $|E\Delta F| = |\Pvdi| + |\Qvdi|$ 
to conclude the proof of estimate \eqref{eq:volume_distance_inequality}.
\end{proof}

\begin{lemma}[\textbf{Morrey-type estimate}]
Suppose that $\forcing$ satisfies (\ref{hyp:main}c). Then 
there exists $\constantforcingMorrey
 >0$ such that 
\begin{equation}\label{eq:on_existence_of_gammag}
\sup_{0<|A|<\omega_n
\constantforcingMorrey
^n} \tfrac{1}{|A|^{\frac{n-1}{n}}}
\int_{A}|\forcing|~dx \le \tfrac{
\constnorm
 n\omega_n^{1/n}}{4},
\end{equation}
where we write
$\constantforcingMorrey^n=(\constantforcingMorrey)^n$.
\end{lemma}

\begin{proof}
If $\forcingexponent=+\infty,$ then
$$
\int_{A}|\forcing|~dx \le \|\forcing\|_\infty |A|^{\frac{1}{n}}|A|^{\frac{n-1}{n}} \le \tfrac{\constnorm n\omega_n^{1/n}}{4}|A|^{\frac{n-1}{n}}
$$
provided $|A|\le \omega_n \constantforcingMorrey^n$ with $\constantforcingMorrey:=\tfrac{\constnorm n}{4(1+\|\forcing\|_\infty)}.$
If $\forcingexponent\in (n,+\infty),$ then by the H\"older inequality
$$
\int_{A}|\forcing|~dx \le \|\forcing\|_{L^\forcingexponent(\Rn)}|A|^{\frac{\forcingexponent
-1}{\forcingexponent
}}\le 
\tfrac{\constnorm n\omega_n^{1/n}}{4}|A|^{\frac{n-1}{n}}
$$
provided $|A|\le \omega_n \constantforcingMorrey^n$ 
with $\constantforcingMorrey
:=\omega_n^{-1/n} 
\big(\tfrac{\constnorm n}{4(1+\|\forcing \|_{L^\forcingexponent})}\big)^{\frac{p}{p-n}}.$
Finally, if $\forcingexponent
=n,$ then by the H\"older inequality for any $A\subset\Rn$
$$
\int_{A} |\forcing|~dx \le \Big(\int_{A} |\forcing|^n ~dx\Big)^{\frac{1}{n}}
|A|^{\frac{n-1}{n}}.
$$
By the absolute continuity of the Lebesgue integral, there exists $\constantforcingMorrey>0$ such that if $|A|<\omega_n
\constantforcingMorrey^n,$ then
$$
 \Big(\int_{A} |\forcing|^n ~dx\Big)^{\frac{1}{n}} \le \frac{\constnorm n\omega_n^{1/n}}{4}.
$$
\end{proof}

\section{Comparison principles} 
The next lemma is a generalization of \cite[Lemma 7.3]{ATW:1993}. 

\begin{lemma}\label{lem:discrete_comparison}
Let $\norm$ be a $C^3$-elliptic anisotropy, $\affine,\forcing$ satisfy 
Hypothesis \eqref{hyp:main} and let 
$\forcing$ be continuous. Let $E\in \cS^*,$ $\tau>0,$ $k\ge0$ and 
$E_\tau$ be a  
minimizer of $\shorthandgeneralfunctional(\cdot; E,\tau).$ Let $F$ and $F_\tau$ be $C^{2}$-sets. 

\begin{itemize}
\item[\rm(a)] Let $E\subset F,$ $E_\tau\subset F_\tau$ and 
\begin{equation}\label{outer_barrier}
\affine\Big(\tfrac{\sd_{F}}{\tau}\Big) > -\kappa_{F_\tau}^\norm - \forcing  \quad \text{on $\p F_\tau.$}
\end{equation}
Then $\p E_\tau\cap \p F_\tau=\emptyset.$

\item[\rm(b)] Let $F\subset E,$ $F_\tau\subset E_\tau$ and 
\begin{equation*}
\affine\Big(\tfrac{\sd_{F}}{\tau}\Big) < -\kappa_{F_\tau}^\norm -  \forcing \quad \text{on $\p F_\tau.$} 
\end{equation*}
Then $\p E_\tau\cap \p F_\tau=\emptyset.$
\end{itemize}
\end{lemma}

\begin{proof}
We prove only (a), assertion (b) being similar. By contradiction, assume that $x_0\in \p E_\tau\cap \p F_\tau.$ Since $\norm$ is smooth and elliptic, repeating the arguments of \cite[Chapter 17]{Giusti:1984} we can show that singular minimal cones locally minimizing $\norm$-perimeter cannot contain a halfspace. By the smoothness of $F_\tau,$ $x_0\in \p^*E_\tau.$ Moreover, by elliptic regularity, $\p^*E_\tau$ is $C^2,$
and hence, 
computing the first variation of $\shorthandgeneralfunctional(\cdot; E,\tau)$ at $E_\tau$ we get 
\begin{equation}\label{nepomnyu}
\affine\Big(\tfrac{\sd_{E}(x_0)}{\tau}\Big) = -\kappa_{F_\tau}^\norm(x_0) - \forcing(x_0).
\end{equation}
By the choice of $x_0,$ $\kappa_{E_\tau}^\norm(x_0)\ge \kappa_{F_\tau}^\norm(x_0)$ and by 
assumption $E\subset F,$ $\sd_{E}(x_0) \ge \sd_{F}(x_0).$ Therefore, combining \eqref{nepomnyu} and \eqref{outer_barrier} we get 
$$
\affine\Big(\tfrac{\sd_{F}(x_0)}{\tau}\Big) > -\kappa_{F_\tau}^\norm(x_0) - \forcing(x_0) \ge -\kappa_{E_\tau}^\norm(x_0) - \forcing(x_0)= \affine\Big(\tfrac{\sd_{E}(x_0)}{\tau}\Big), 
$$
a contradiction.
\end{proof}

Comparison principles for \eqref{eq:prescur_funco} are
well-established provided that the prescribed mean curvatures are comparable.

\begin{theorem}[\textbf{Comparison principles}]\label{teo:compare}
Let $\affine_1,\affine_2,\forcing_1,\forcing_2$ satisfy 
(\ref{hyp:main}a) and (\ref{hyp:main}c). Let $F_1,F_2\in \cS^*$ and $\tau>0.$ 
The following properties hold:

\begin{itemize}
\item[\rm(i)] If 
$F_1\subset F_2,$ $\affine_1\ge\affine_2$ and $\forcing_1>\forcing_2$ 
a.e. in $\Rn,$ then minimizers $F_\tau^i$ of $\sF_{\norm,\affine_i,\forcing_i}(\cdot;F_i,\tau)$ satisfy  
$
F_\tau^1\subset F_\tau^2.
$

\item[\rm(ii)] If $F_1\Subset F_2,$ $\affine_1\ge\affine_2$  and 
$\forcing_1\ge \forcing_2$ a.e. in $\Rn,$ then minimizers $F_\tau^i$ of 
$\sF_{\norm,\affine_i,\forcing_i}(\cdot;F_i,\tau)$ satisfy
$
F_\tau^1\subset F_\tau^2.
$

\item[\rm(iii)] If $F_1\subset F_2,$ $\affine_1\ge\affine_2$  
and $\forcing_1\ge \forcing_2$ a.e. in $\Rn,$ then there exist minimizers 
$F_{\tau*}^1$ of $\sF_{\norm,\affine_1,\forcing_1}(\cdot;F_1,\tau)$ and 
$F_\tau^{2*}$ of $\sF_{\norm,\affine_2,\forcing_2}(\cdot;F_2,\tau)$ such that
$
F_{\tau*}^1\subset F_\tau^2 
$
and
$
F_\tau^1\subset F_{\tau}^{2*}
$
for all minimizers $F_\tau^i$ of $\sF_{\norm,\affine_i,\forcing_i}(\cdot;F_i,\tau).$
\end{itemize}
\end{theorem}

Indeed, assumptions of (i) and (ii) 
imply that $h_{\norm,\affine_1,\forcing_1}>h_{\norm,\affine_2,\forcing_2}$ 
a.e. in $\Rn,$ while (iii) 
implies $h_{\norm,\affine_1,\forcing_1}\ge h_{\norm,\affine_2,\forcing_2},$
see \eqref{eq:h}. Thus
the proof  follows from standard arguments (see e.g. \cite{BKh:2018} and  references therein).

\begin{corollary}[\textbf{Minimal and maximal minimizers}]\label{cor:minmaxprinciple}
Let $\affine$ satisfy 
(\ref{hyp:main}a) and $\forcing$ 
satisfy (\ref{hyp:main}c), and let $\tau>0$. Then for any $F\in \cS^*$ 
there exist minimizers $F_{\tau*},F_\tau^*$ 
(called the minimal and maximal minimizer) 
of $\shorthandgeneralfunctional(\cdot;F,\tau)$ such that for every minimizer $F_\tau,$
$$
F_{\tau*} \subseteq F_\tau \subseteq F_\tau^*.
$$
\end{corollary}

\begin{corollary}\label{cor:strong_comparison}
Assume that $\forcing\equiv0,$ $F_1\Subset F_2$ with 
$\dist(\p F_1,\p F_2)=\epsilon>0.$ Then the minimizers $F_\tau^i$ of 
$\shorthandgeneralfunctional
(\cdot;F_i,\tau)$ 
satisfy $F_\tau^1\Subset F_\tau^2$ and $\dist(\p F_\tau^1,\p F_\tau^2)\ge\epsilon.$
\end{corollary}

Indeed, since $g \equiv 0$, 
$\shorthandgeneralfunctional(E;F,\tau) = \shorthandgeneralfunctional
(E+\xi;F+\xi,\tau) 
$
for any $\xi\in\Rn.$ By assumptions on $F_1$ and $F_2,$ 
for any $\xi\in B_\epsilon(0)$ we have $F_1+\xi\Subset F_2,$ 
and hence
$F_\tau^1 + \xi\Subset F_\tau^2$ 
 by Theorem \ref{teo:compare} (b).
Thus, $\dist(\p F_\tau^1,\p F_\tau^2)\ge\epsilon.$

As in \cite{BKh:2018} 
we can introduce a comparison principle for two $\GMM$s.

\begin{theorem}
Let $\affine_1,\affine_2,\forcing_1,\forcing_2$ 
satisfy Hypothesis \eqref{hyp:main} with $\affine_1\ge\affine_2$ and $\forcing_1\ge \forcing_2.$ Let $F_1^0,F_2^0\in \cS^*$ be such that $F_1^0\subset F_2^0.$ Then: 

\begin{itemize}
\item[\rm(a)] for any $F_1(\cdot)\in \GMM(\sF_{\norm,\affine_1,\forcing_1},\cS^*,F_1^0)$ 
there exists $F_2^*(\cdot)\in \GMM(\sF_{\norm,\affine_2,\forcing_2},\cS^*,F_2^0)$ such that 
$$
F_1(t) \subset F_2^*(t),\quad t\ge0;
$$

\item[\rm(b)] for any $F_2(\cdot)\in \GMM(\sF_{\norm,\affine_2,
\forcing_2},\cS^*,F_2^0)$ there exists $F_{1 *}(\cdot)\in \GMM(\sF_{\norm,\affine_1,\forcing_1},\cS^*,F_1^0)$ such that 
$$
F_{1 *}(t) \subset F_2(t),\quad t\ge0.
$$
\end{itemize}
\end{theorem}

We sketch the proof of only (a). Let $\{F_1(\tau_j,k)\}$ be flat flows starting from $F_1^0$ such that 
$$
\lim\limits_{j\to+\infty} |F_1(\tau_j,\intpart{t/\tau_j}) \Delta F_1(t)| = 0.
$$
Let also $F_2^*(\tau_j,k)$ be the flat flows starting from $F_2^0$ consisting of the maximal minimizers of $\sF_{\norm,\affine_2,\forcing_2}.$ By Theorem \ref{teo:compare} (d) 
\begin{equation}\label{uyoqqayuradi09}
F_1(\tau_j,k)\subset F_2^*(\tau_j,k),\quad k\ge0,\,j\ge1.
\end{equation}
Now consider the sequence $(F_2^*(\tau_j,\intpart{t/\tau_j})).$ 
In the proof of Theorem \ref{teo:existence_gmm}
we have constructed a not relabelled subsequence and a 
family $F_2^*(\cdot)\in 
\GMM(\mathcal F_{\norm,\affine_2,\forcing_2},\cS^*,F_2^0)$ such that 
$$
\lim\limits_{j\to+\infty} |F_2^*(\tau_j,\intpart{t/\tau_j}) \Delta F_2^*(t)| = 0
\qquad \forall t \geq 0
$$
(see also Remark \ref{rem:arbitrary_sequence}). 
Now the inclusion $F_1(\cdot)\subset F_2^*(\cdot)$ follows from \eqref{uyoqqayuradi09}.

\end{document}